\documentclass[reqno,12pt,a4paper]{amsart}
\usepackage{lineno}
\usepackage{amsmath,amssymb,amsthm,graphicx,mathrsfs,url,bbm}
\usepackage{a4wide}
\usepackage{enumitem}
\usepackage{mathtools, stackrel}
\usepackage[usenames,dvipsnames]{xcolor}
\usepackage[colorlinks=true,linkcolor=Blue,citecolor=Green]{hyperref}
\hypersetup{pdfstartview=XYZ}
\usepackage{tikz-cd}
\usetikzlibrary{arrows.meta}
\usepackage[colorinlistoftodos]{todonotes}

\theoremstyle{plain}
\newtheorem{thm}{Theorem}
\newtheorem{lem}{Lemma}[section]
\newtheorem{prop}[lem]{Proposition}
\newtheorem{cor}[lem]{Corollary}

\theoremstyle{definition}
\newtheorem{definition}[lem]{Definition}
\newtheorem{ex}[lem]{Example}

\theoremstyle{remark}
\newtheorem{rem}{Remark}[section]

\numberwithin{equation}{section}

\newcommand{\C}{\mathbb{C}}

\newcommand{\N}{\mathbb{N}}

\newcommand{\x}{\times}

\newcommand*\Laplace{\mathop{}\!\mathbin\bigtriangleup}

\let\Im=\Imag

\let\Re=\Real

\def\Ddots{\mathinner{\mkern1mu\raise\p@
    \vbox{\kern7\p@\hbox{.}}\mkern2mu
    \raise4\p@\hbox{.}\mkern2mu\raise7\p@\hbox{.}\mkern1mu}}
\makeatother

\renewcommand{\d}{\,d}

\newcommand{\F}{{\mathcal{F}}}

\renewcommand{\epsilon}{\vararepsilon}

\newcommand{\bdm}{\begin{displaymath}}
  \newcommand{\edm}{\end{displaymath}}
\newcommand{\bq}{\begin{equation}}
  \newcommand{\eq}{\end{equation}}
\newcommand{\bqn}{\begin{equation*}}
  \newcommand{\eqn}{\end{equation*}}

\newcommand{\eop}[1]{\ensuremath{\vec{e}^{\,\mathrm{op}}}}

\title[Residues of weighted dynamical Ihara zeta functions]{Residues of weighted dynamical Ihara zeta functions on finite regular graphs}

\author[G. Palmirotta]{Guendalina Palmirotta}
\address{Institute of Mathematics und Institute for Photonic Quantum Systems (PhoQS), Universit\"at Paderborn, Warburgerstr. 100, 33098 Paderborn, Germany} \email{gpalmi@math.uni-paderborn.de}

\date{\today}
\begin{document}

\begin{abstract}
A weighted dynamical version $\mathbf{Z}_f$ of the classical Ihara zeta function is presented on connected finite regular graphs, expressed in terms of weighted periodic orbit data.
We establish its meromorphic continuation, with poles given by the resonances of the associated non-backtracking transfer operator acting on a suitable Banach space, and compute its residues.
At simple spectral parameters, the residue of $\mathbf{Z}_f$ is identified with the invariant Ruelle distribution on the graph phase space. 
Combining this result with a relation obtained by Arends–Palmirotta, we further derive Patterson–Sullivan and Wigner residue formulae. This provides a finite-graph analogue of residue formulae for weighted dynamical zeta functions for geodesic flow on rank one locally symmetric spaces, in the spirit of Schütte–Barkhofen–Weich.\\


\end{abstract}

\maketitle

%\setcounter{tocdepth}{3}
%\tableofcontents

% -------------------------------%%--------------------------------------%
\section{Introduction} \label{sect:Intro}
Let $\mathfrak{G}_\Gamma=(\mathfrak{X}_\Gamma,\mathfrak{E}_\Gamma)$ be a connected finite $(q+1)$-regular graph with $q>1$, realised as the quotient of the $(q+1)$-regular tree ${\mathfrak{G}}$ by a discrete, cocompact, free action of a group $\Gamma$, and 
let $(\mathfrak{P}_\Gamma,\sigma)$ be the space of $\Gamma$-invariant non-backtracking bi-infinite paths, with the dynamics given by the shift map $\sigma$.
Such graphs admit a natural geometric and dynamical structure analogous to compact rank-one locally symmetric spaces.
Figure~\ref{fig:quotient_space} illustrates this basic geometric and dynamical correspondence.

\begin{figure}
  \centering

  % Continuation dots from a terminal vertex, in the local outward direction.
  \newcommand{\LeafDots}{%
    \foreach \x in {0.055,0.095,0.135}
      \fill[black] (\x,0) circle (0.1pt);
  }

  \begin{tikzpicture}[scale=2, line width=0.6pt]

    % ---------- Colors ----------
    \definecolor{myred}{HTML}{CC3311}
    \definecolor{myblue}{HTML}{0072B2}
    \definecolor{myorange}{HTML}{FFD700}
    \definecolor{mygreen}{HTML}{009E73}
    \definecolor{mygray}{HTML}{D9D9D9}

    % Colors encoded as:
    % 1 = red, 2 = blue, 3 = green, 4 = orange
    \newcommand{\ColorOf}[1]{%
      \ifcase#1\relax
        \or myred%
        \or myblue%
        \or mygreen%
        \or myorange%
      \fi
    }

    % Given the color of the current vertex (#1) and the color
    % of its parent (#2), set \childA and \childB equal to
    % the two remaining colors in {1,2,3,4}.
    \newcommand{\GetChildren}[2]{%
      \ifnum#1=1\relax
        \ifnum#2=2\relax
          \def\childA{3}\def\childB{4}%
        \fi
        \ifnum#2=3\relax
          \def\childA{2}\def\childB{4}%
        \fi
        \ifnum#2=4\relax
          \def\childA{2}\def\childB{3}%
        \fi
      \fi
      \ifnum#1=2\relax
        \ifnum#2=1\relax
          \def\childA{3}\def\childB{4}%
        \fi
        \ifnum#2=3\relax
          \def\childA{1}\def\childB{4}%
        \fi
        \ifnum#2=4\relax
          \def\childA{1}\def\childB{3}%
        \fi
      \fi
      \ifnum#1=3\relax
        \ifnum#2=1\relax
          \def\childA{2}\def\childB{4}%
        \fi
        \ifnum#2=2\relax
          \def\childA{1}\def\childB{4}%
        \fi
        \ifnum#2=4\relax
          \def\childA{1}\def\childB{2}%
        \fi
      \fi
      \ifnum#1=4\relax
        \ifnum#2=1\relax
          \def\childA{2}\def\childB{3}%
        \fi
        \ifnum#2=2\relax
          \def\childA{1}\def\childB{3}%
        \fi
        \ifnum#2=3\relax
          \def\childA{1}\def\childB{2}%
        \fi
      \fi
    }

    % ------------------------------------------------------------
    % Recursively draw and color a non-root subtree.
    %
    % #1 = number of levels still to draw
    % #2 = current edge length
    % #3 = color of the current vertex
    % #4 = color of the parent vertex
    % ------------------------------------------------------------
    \newcommand{\ColorSubtree}[4]{%
      \ifnum#1=0
        \LeafDots
      \else

        % Freeze all incoming arguments before recursive calls.
        \edef\thisdepth{#1}%
        \edef\thislength{#2}%
        \edef\thiscolor{#3}%
        \edef\parentcolor{#4}%

        % Data for the next level.
        \pgfmathtruncatemacro{\nextdepth}{\thisdepth-1}%
        \pgfmathsetmacro{\nextlength}{\thislength*0.66}%

        % Find the two colors different from both
        % the current color and the parent color.
        \GetChildren{\thiscolor}{\parentcolor}%
        \edef\firstchild{\childA}%
        \edef\secondchild{\childB}%

        % Freeze recursive calls before entering recursion.
        \edef\FirstRecursiveCall{%
          \noexpand\ColorSubtree
            {\nextdepth}%
            {\nextlength}%
            {\firstchild}%
            {\thiscolor}%
        }%

        \edef\SecondRecursiveCall{%
          \noexpand\ColorSubtree
            {\nextdepth}%
            {\nextlength}%
            {\secondchild}%
            {\thiscolor}%
        }%

        % First child
        \draw (0,0) -- ++(60:\thislength) coordinate (ca);
        \edef\firstcolor{\ColorOf{\firstchild}}%
        \fill[\firstcolor] (ca) circle (0.7pt);

        \begin{scope}[shift={(ca)},rotate=60]
          \FirstRecursiveCall
        \end{scope}

        % Second child
        \draw (0,0) -- ++(-60:\thislength) coordinate (cb);
        \edef\secondcolor{\ColorOf{\secondchild}}%
        \fill[\secondcolor] (cb) circle (0.7pt);

        \begin{scope}[shift={(cb)},rotate=-60]
          \SecondRecursiveCall
        \end{scope}

      \fi
    }

    % ------------------------------------------------------------
    % Trace a path, skipping the first two edges.
    %
    % #1 = first-edge length
    % #2,...,#7 = successive directions
    % #8 = trace color
    %
    % The first two edges are drawn invisibly as paths;
    % the third edge onward is highlighted.
    % ------------------------------------------------------------
    \newcommand{\TracePathSixSkipFirstTwo}[8]{%
      % First edge: traced
      \draw[#8, opacity=0.7, line width=5pt]
        (0,0) -- ++(#2:#1) coordinate (c1);

      \begin{scope}[shift={(c1)}, rotate=#2]

        % Second edge: traced
        \pgfmathsetmacro{\lTwo}{#1*0.66}
        \draw[#8, opacity=0.7, line width=5pt]
            (0,0) -- ++(#3:\lTwo) coordinate (c2);

        \begin{scope}[shift={(c2)}, rotate=#3]

          % Third edge: traced
          \pgfmathsetmacro{\lThree}{\lTwo*0.66}
          \draw[#8, opacity=0.7, line width=5pt]
            (0,0) -- ++(#4:\lThree) coordinate (c3);

          \begin{scope}[shift={(c3)}, rotate=#4]

            % Fourth edge: traced
            \pgfmathsetmacro{\lFour}{\lThree*0.66}
            \draw[#8, opacity=0.7, line width=5pt]
              (0,0) -- ++(#5:\lFour) coordinate (c4);
              
            \begin{scope}[shift={(c4)}, rotate=#5]

            % Fourth edge: traced
            \pgfmathsetmacro{\lFive}{\lThree*0.66}
            \draw[#8, opacity=0.7, line width=5pt]
              (0,0) -- ++(#5:\lFive) coordinate (c5);
            
          \end{scope}
          \end{scope}
        \end{scope}
      \end{scope}
    }

    % ============================================================
    % LEFT PANEL: fully colored universal cover
    % ============================================================
    \begin{scope}

      % Root: red, encoded by color 1.
      \node[
        circle,
        fill=myred,
        inner sep=1.4pt
      ] at (0,0) {};

      % Edge length at the second level.
      \pgfmathsetmacro{\levelone}{0.7*0.66}

      % ----------------------------------------------------------
      % Root neighbor 1: blue
      % ----------------------------------------------------------
      \draw (0,0) -- ++(90:0.7) coordinate (r1);
      \fill[myblue] (r1) circle (0.9pt);

      \begin{scope}[shift={(r1)},rotate=90]
        \ColorSubtree{4}{\levelone}{2}{1}
      \end{scope}

      % ----------------------------------------------------------
      % Root neighbor 2: green
      % ----------------------------------------------------------
      \draw (0,0) -- ++(210:0.7) coordinate (r2);
      \fill[mygreen] (r2) circle (0.9pt);

      \begin{scope}[shift={(r2)},rotate=210]
        \ColorSubtree{4}{\levelone}{3}{1}
      \end{scope}

      % ----------------------------------------------------------
      % Root neighbor 3: orange
      % ----------------------------------------------------------
      \draw (0,0) -- ++(-30:0.7) coordinate (r3);
      \fill[myorange] (r3) circle (0.9pt);

      \begin{scope}[shift={(r3)},rotate=-30]
        \ColorSubtree{4}{\levelone}{4}{1}
      \end{scope}

      % ----------------------------------------------------------
      % Orange geodesics.
      %
      % First two edges are omitted from the trace, so the
      % highlighted segment starts at the third edge.
      % ----------------------------------------------------------
      
      \TracePathSixSkipFirstTwo
        {0.7}{-30}{60}{60}{-60}{60}{-60}{mygray}

    \TracePathSixSkipFirstTwo
        {-0.7}{30}{60}{60}{-60}{60}{-60}{mygray}

      \node at (0,-1.5)
        {\small universal cover ${\mathfrak{G}}$ (tree)};

    \end{scope}

    % ============================================================
    % COVERING MAP ARROW
    % ============================================================
    \draw[
      -{Latex[length=3mm]},
      thick,
      xshift=2mm
    ]
      (1.5,0) --
      node[above]{\small $\pi$}
      (2.3,0);

    % ============================================================
    % RIGHT PANEL: finite quotient K_4
    % ============================================================
    \begin{scope}[xshift=3.4cm, yshift=0.3cm]

      % Vertices
      \node[
        circle,
        fill=myred,
        inner sep=1.6pt
      ] (w1) at (0.05,-0.3) {};

      \node[
        circle,
        fill=myorange,
        inner sep=1.6pt
      ] (w2) at (0.75,-0.7) {};

      \node[
        circle,
        fill=mygreen,
        inner sep=1.6pt
      ] (w3) at (-0.5,-0.7) {};

      \node[
        circle,
        fill=myblue,
        inner sep=1.6pt
      ] (w4) at (0.05,0.5) {};

      % Edges of K_4
      \draw (w1)--(w2)--(w3)--(w4)--(w1);
      \draw (w1)--(w3);
      \draw (w2)--(w4);

      % ----------------------------------------------------------
      % Grey projected paths
      % ----------------------------------------------------------
      \draw[
        mygray,
        opacity=0.7,
        line width=4pt
      ]
        (w1) -- (w2);
    
      \draw[
        mygray,
        opacity=0.7,
        line width=4pt
      ]
        (w1) -- (w3);

    \draw[
        mygray,
        opacity=0.7,
        line width=4pt
      ]
        (w2) -- (w4);

    \draw[
        mygray,
        opacity=0.7,
        line width=4pt
      ]
        (w1) -- (w4);

    \draw[
        mygray,
        opacity=0.7,
        line width=4pt
      ]
        (w2) -- (w3);

    \draw[
        mygray,
        opacity=0.7,
        line width=4pt
      ]
        (w3) -- (w4);
        
      % Label
      \node at (0,-1.8) {%
        \small finite quotient
        $\mathfrak{G}_\Gamma
        =\Gamma\backslash {\mathfrak{G}}=K_4$%
      };

    \end{scope}
    \end{tikzpicture}
    \caption{The complete graph $K_4$ is a finite $3$-regular graph, and its universal cover is the $3$-regular tree. 
    Closed non-backtracking cycles (shadowed in grey) on the finite quotient lift to non-backtracking infinite paths (analogous to geodesic segments)
    in the tree whose endpoints lie on the boundary. 
    The \emph{phase space} $\mathfrak{P}_\Gamma$ may therefore be viewed as the quotient, by the deck group $\Gamma$, of the space of bi-infinite geodesics in the universal covering tree. 
    This is the discrete analogue of the unit tangent bundle of a compact negatively curved manifold. For more details, we refer to Section~\ref{sect:qoutient_spaces}.}
  \label{fig:quotient_space}
\end{figure}
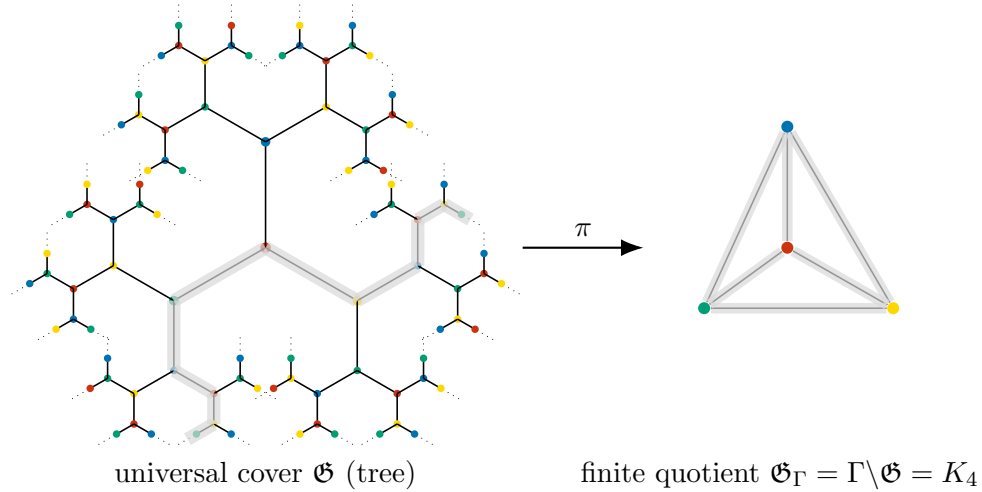

We introduce a weighted version\footnote{Note that the term \emph{weighted} Ihara zeta function is also used in the algebraic combinatorics literature (e.g., the weighted Ihara zeta function of Stark--Terras \cite{StarkTerras1996} or of Mizuno--Sato~\cite{MizunoSato04}) in a different sense.
There, one attaches weights to the edges (or to the adjacency matrix) of a finite graph and derives a determinant (Ihara--Bass type) expression. %, often motivated by quantum walks. 
Our weighted zeta function $\zeta(s,\beta)$, by contrast, is a dynamical (thermodynamic-formalism) construction in which $\beta$ deforms the periodic-orbit weights by a locally constant observable $f\in C^{\mathrm{lc}}(\mathfrak{P}_\Gamma)$, in analogy with weighted Ruelle zeta functions on hyperbolic manifolds.}
of the classical Ihara zeta function by incorporating a suitable observable into the periodic orbit weights.
More precisely, for a locally constant weight $f\in C^{\mathrm{lc}}(\mathfrak{P}_\Gamma)$ and a parameter $\beta\in\mathbb{C}$, we
consider the \emph{$\beta$-parametrized Ihara zeta function} given by
\begin{equation*} %\label{eq:beta-Ihara}
    \zeta(s,\beta) 
        \coloneqq 
        \prod_{[\mathbf{p}_0]\in\mathscr{P}_0}
        \left(1-q^{(\frac12+is)\ell_0} \mathbf{e}_{\beta,f}(\mathbf{p}_0)
        \right)^{-1}, \qquad s \in \C
\end{equation*}
where the product runs over all primitive closed non-backtracking finite paths $[\mathbf{p}_0]\in\mathscr{P}_0$ in $\mathfrak{G}_\Gamma$ (that is, primitive cycles on which one “goes around only once”) and $\ell_0$ denotes the length of the underlying primitive cycle. 
The exponent weight is defined by
$$\mathbf{e}_{\beta,f}(\mathbf{p}_0) \coloneqq
        \exp\left(-\beta \sum_{k=0}^{\ell_0-1} f\big(\sigma^k(\mathbf{p}_0)\big)\right),
$$
where the sum weights each primitive cycle by the values of $f$ along its orbit under $\sigma.$
When $\beta=0$, we recover the Ihara zeta function introduced in \cite{Ihara66}, which is the discrete, non-Archimedean counterpart of the \emph{Ruelle zeta function} \cite{Ruelle76} attached to the geodesic flow on a compact hyperbolic surface, or more generally on a compact rank-one locally symmetric space.

Its logarithmic derivative evaluated at $\beta=0$ gives
\begin{equation} \label{eq:intro_dyn_zeta}
    \mathbf{Z}_f(s)
        \coloneqq
        \left.\frac{\partial}{\partial\beta}\log\zeta(s,\beta)\right|_{\beta=0}
        =
        \sum_{[\mathbf{p}_0]\in\mathscr{P}_0}
        \frac{q^{(\frac12+is)\ell_0}}
             {1-q^{(\frac12+is)\ell_0}}
        \sum_{k=0}^{\ell_0-1} f(\sigma^k(\mathbf{p}_0)),
\end{equation}
which we call the \emph{weighted dynamical Ihara zeta function} associated with $f$.
Expanding the geometric series shows that $\mathbf{Z}_f$ is in fact a sum over \emph{all} (not necessarily primitive) closed orbits of $\sigma$, each weighted by its underlying primitive period.
This makes $\mathbf{Z}_f$ the discrete, non-Archimedean analogue of the \emph{dynamical weighted trace} of Schütte--Barkhofen--Weich~\cite{SchutteWeichBarkhofen23_weightedzeta}.
In the setting of compact rank-one locally symmetric spaces, the meromorphic continuation of that trace encodes Pollicott–Ruelle resonances as poles and phase-space distributions as residues.

The purpose of the present paper is to establish the analogous meromorphic and residue theory for our weighted dynamical Ihara zeta function $\mathbf{Z}_f$ in the graph setting. 

\subsection{Statement of the two main results} 

Consider the past and future non-backtracking one-sided infinite paths $\mathfrak{P}_\Gamma^\pm$, together with their respective shifts $\sigma_\pm$, such that $\mathfrak{P}_\Gamma \subset \mathfrak{P}_\Gamma^+ \times \mathfrak{P}_\Gamma^-$. We refer to Section~\ref{sect:shiftspaces} for the formal definitions of these shift spaces.

Following the approach of Bux--Hilgert--Weich \cite{BHW23} (see also the references therein), one can define on $\mathfrak{G}_\Gamma$ a discrete subset of the complex plane $\C$, namely the spectrum of the transfer operator $\mathcal{L}_{\Gamma,\pm}$ acting on a suitable Banach space $\F_\vartheta(\mathfrak{P}^\pm_\Gamma)$, for fixed $\vartheta\in (0,1)$. Its elements are often called Ruelle \emph{resonances} in the literature.
These resonances arise as the poles of the meromorphic continuation of the resolvent
$$
    \mathbf{R}(s)
        \coloneqq
        \bigl(\operatorname{Id}
        -
        q^{\frac12+is}\mathcal{L}_{\Gamma,\pm}\bigr)^{-1}: \F_\vartheta(\mathfrak{P}^\pm_\Gamma) \longrightarrow \F_\vartheta(\mathfrak{P}^\pm_\Gamma), \qquad \vartheta\in (0,1)
$$
and each such resonance $s_0$, with corresponding eigenvalue $\lambda(s_0)\coloneqq q^{-(\frac12+is_0)}$ of $\mathcal{L}_{\Gamma,\pm}$, is associated with a finite-rank residue operator $\Pi_{s_0}.$
Our first main result concerns the meromorphic continuation of \eqref{eq:intro_dyn_zeta}, together with its residues, and reads as follows.
\begin{thm}[Meromorphic continuation of the weighted dynamical Ihara zeta] \label{thm:mero_dynfct}
    Let $\mathfrak{G}_\Gamma$ be a connected, finite $(q+1)$-regular graph with $q>1$ and let $f\in C^{\mathrm{lc}}(\mathfrak{P}_{\Gamma})$.
        \begin{enumerate}
            \item[$\mathrm{(i)}$] The weighted dynamical Ihara zeta function %defined in \eqref{eq:intro_dyn_zeta} for $\mathfrak{G}_\Gamma$ 
            converges absolutely for $\mathrm{Im}(s)>\tfrac{3}{2}.$ 
            \item[$\mathrm{(ii)}$] For every fixed $\vartheta\in(0,1)$, the function $\mathbf{Z}_f(s)$ continues meromorphically to the half-plane 
            $$
            \mathscr{H}_{\vartheta,1}
            =
            \left\{
                s\in\mathbb C
                \;\middle|\;
                \operatorname{Im}(s)>
                \frac32+\frac{\log\vartheta}{\log q}
            \right\}.
            $$
            Its poles in this half-plane are contained in the set of resonances of the transfer operator 
            $\mathcal{L}_{\Gamma,\pm}$ acting on the Banach space
            $\F_\vartheta(\mathfrak{P}^\pm_\Gamma)$.
            \item[$\mathrm{(iii)}$] 
            If the resolvent 
            $\mathbf{R}(s)$ 
            on  $\F_\vartheta(\mathfrak{P}^\pm_{\Gamma})$
            has a pole of order $J(s_0)$ at $s_0$, then,
            for every $0 \leq k \leq J(s_0)$ we have
            \begin{equation} \label{eq:residue_R}
                \mathrm{Res}_{s=s_0} \big[\mathbf{Z}_f(s)(\lambda(s)-\lambda(s_0))^k \big] = \frac{i}{\log q} \mathrm{Tr}\big[(\mathcal{L}_{\Gamma, \pm}-\lambda(s_0))^{k} \Pi_{s_0}f\big].
            \end{equation}
        \end{enumerate}
\end{thm}

Theorem~\ref{thm:mero_dynfct} may be seen as a discrete analogue of \cite[Thm.~1.2]{SchutteWeichBarkhofen23_weightedzeta}, which was established in the setting of open hyperbolic systems.

Moreover, the Riesz spectral projector $\Pi_{s_0}$ associated with $\mathbf{R}(s)$ can be identified with the Ruelle projector appearing in the definition of the invariant Ruelle distribution in \cite[Def.~4.3]{ArendsPalmirotta26}, see Proposition \ref{prop:Riesz_Ruelle_identification} for details.
Under this identification, when $J(s_0)=1$, Theorem~\ref{thm:mero_dynfct}~$\mathrm{(iii)}$ allows us to identify the trace appearing in the residue formula \eqref{eq:residue_R} with this invariant Ruelle distribution.
This identification is precisely what drives our second main result, which we describe next.

\subsubsection{Residue formula for phase-space distributions}
To formulate the second main result, let $\Laplace$ be the vertex Laplace operator on $\mathfrak{G}$.
For a spectral parameter $s_0\in \C$, write $\chi(s_0) \in \C\backslash\{\pm 1\}$ for the corresponding eigenvalue of $\Laplace$, and consider its associated eigenspace $\mathcal{E}_{\chi(s_0)}(\Laplace;\mathrm{Maps}(\mathfrak{X}, \mathbb{C}))$.

On $\mathfrak{G}$, the classical dynamics are described through the associated transfer operators, while the quantum dynamics is modelled by $\Laplace$.
Bux--Hilgert--Weich \cite{BHW23} and Arends--Frahm--Hilgert \cite{AFH23} established a quantum-classical correspondence between these two descriptions.

Relying on this quantum-classical correspondence, \cite[Thm.~4]{ArendsPalmirotta26} give a dynamical description of \emph{Patterson--Sullivan distribution} $\mathrm{PS}_{\phi, \phi} \in \mathcal{D}'(\mathfrak{P}_\Gamma)$ associated to $\Gamma$-invariant Laplace eigenfunctions $\phi \in \mathcal{E}_{\chi(s_0)}(\Laplace;\mathrm{Maps}(\mathfrak{X}, \mathbb{C}))^\Gamma$. %of the induced  Laplacian $\Laplace_\Gamma$ on the quotient graph $\mathfrak{G}_\Gamma.$
For a formal definition of these objects, we refer the reader to Subsection~\ref{sect:residue}.
Our second main result is stated as follows.

\begin{thm}[Patterson-Sullivan distributions as residues]
\label{thm:residue_PS}
    Let $\mathfrak{G}_\Gamma=(\mathfrak{X}_\Gamma, \mathfrak{E}_\Gamma)$ be a connected, finite $(q+1)$-regular graph with $q> 1$, and let $\mathfrak{P}_\Gamma$ be the $\Gamma$-invariant shift space of non-backtracking bi-infinite chains as in Definition~\ref{def:space_biinfinite_chains}.
    Consider a resonance $s_0 \in \mathbb{C}$ of multiplicity $\mathfrak{m} \in \mathbb{N}_0$ such that $q^{\frac{1}{2} + is_0} \notin \{\pm 1, \pm q\}$ and such that there is no Jordan block associated with $s_0$.
    
    Then 
    for any $\ell^2(\mathfrak{X}_\Gamma)$-orthonormal basis $\phi_1 , \ldots , \phi_{\mathfrak{m}}$ of 
    $\mathcal{E}_{\chi(s_0)}(\Laplace;$ $\mathrm{Maps}(\mathfrak{X}, \mathbb{C}))^\Gamma$, and for $f\in C^{\mathrm{lc}}(\mathfrak{P}_\Gamma)$, the following residue formula holds:
    \begin{equation} \label{eq:residue_PS}
        \mathrm{Res}_{s=s_0} \Big[\mathbf{Z}_{f}(s)\Big]
        = c_{q,s_0} \sum_{\ell = 1}^{\mathfrak{m}} \mathrm{P S}_{\phi_{\ell} , \phi_{\ell}}(f),
    \end{equation}
    where the constant $c_{q,s_0}\in \C$ is given explicitly by $c_{q,s_{0}} \coloneqq \frac{i}{\log q} \frac{q^{1 + 2 is_{0}} - 1}{q^{1 + 2 is_0} - q}$.
\end{thm}

Thus the Patterson--Sullivan distributions introduced in \cite{ArendsPalmirotta26} are recovered from weighted periodic-orbit data, shedding new light on the spectral analysis of $\Laplace_\Gamma$ on $\mathfrak{G}_\Gamma$.

Moreover, since by \cite[Thm.~6]{ArendsPalmirotta26} these distributions are exactly identified with microlocal lifts, the \emph{Wigner distributions} $W_{\phi, \phi}$, we also derive a semiclassical residue formula for $W_{\phi, \phi}$ in Corollary~\ref{cor:Wigner}.

This new perspective addresses the problem stated in \cite[Problem 6.33]{Hilgert23}.
Furthermore, Theorem~\ref{thm:residue_PS} can be viewed as a discrete counterpart in the Archimedean setting. 
For compact hyperbolic surfaces of rank one, Anantharaman--Zelditch~\cite[Thm.~1.3]{AZ07} first showed, by representation-theoretic methods, that these residues coincide with {Patterson--Sullivan} distributions. 
Subsequently, Emonds, in his thesis~\cite{Emonds14}, extended the residue formula of~\cite{AZ07} to hyperbolic manifolds of arbitrary dimension, using a suitably restricted class of test functions. %relying on techniques from representation theory. %, under a corresponding restriction on the space of test functions.
%namely that $f$ be $K$-finite. 
%and working with a more restricted class of test functions.
By purely microlocal analysis methods from~\cite{DyatlovZworski19,
DyatlovGuillarmou16}, %better suited to the meromorphic continuation of (weighted) zeta functions, 
Sch\"utte--Barkhofen--Weich~\cite[Thm.~4.1]{SchutteWeichBarkhofen23_weightedzeta} generalised this residue
formula to invariant Ruelle distributions, for general open hyperbolic systems and arbitrary smooth weights. 
Moreover, using the relation established by Guillarmou--Hilgert--Weich~\cite{GHWb} for compact rank-one locally symmetric spaces, they recover the Patterson–Sullivan interpretation as a special case. 
This generalises both Anantharaman–Zelditch's and Emonds' results.
We also refer to~\cite{BarkhofenSchutteWeich22semiclassical, SchutteWeich23, DP24} for further extensions and numerical aspects.

In addition, in Section~\ref{sect:dyn_determinants}, we express \eqref{eq:intro_dyn_zeta} as a dynamical determinant, making $\mathbf{Z}_f$ applicable for numerics. Mirroring the algorithm of \cite{SchutteWeich23}, it would therefore be interesting to compute invariant Ruelle distributions on finite regular graphs numerically via weighted dynamical zeta functions.
We intend to pursue this numerical aspect in future work.

\subsection{Two key ingredients and proof strategy}
Inspired by the Archimedean setting \cite[Thm.~1.2 and Thm.~4.1]{SchutteWeichBarkhofen23_weightedzeta},
the proofs of our main theorems, in particular  Theorem~\ref{thm:mero_dynfct}, rely on two key ingredients.

\subsubsection{Weighted trace formula (Proposition~\ref{prop:weighted_trace})}

The first key is to derive a \emph{weighted trace formula}
\begin{equation} \label{eq:intro_trace}
    \mathbf{Z}_f(s) = \operatorname{Tr}[\mathbf{R}(s)f], \qquad f\in C^{\mathrm{lc}}(\mathfrak{P}_\Gamma).
\end{equation}
This is established precisely in Proposition~\ref{prop:weighted_trace}, whose proof rewrites \eqref{eq:intro_dyn_zeta} in terms of traces of
$\mathcal{L}_{\Gamma,+}$ on one-sided future path spaces, see Lemma~\ref{lem:trace_unweighted_multiplicative} for details.

The key point behind this reduction is that any two-sided, locally constant function $f\in C^{\mathrm{lc}}(\mathfrak{P}_\Gamma)$ depends, after applying a sufficiently large power of the shift, only on future paths: compactness of $\mathfrak{P}_\Gamma$ forces $f$ to depend on finitely many coordinates, so that shifting far enough into the future \emph{absorbs all} of its dependence on the past. 
Consequently, $f$ agrees with a one-sided potential $f_+\in C^{\mathrm{lc}}(\mathfrak{P}_\Gamma^+)$ composed with the canonical future projection $\pi_+:\mathfrak{P}_\Gamma\to\mathfrak{P}_\Gamma^+$, and the periodic-orbit sum generated by $f$ coincides with the one generated by $f_+$:
$$ \sum_{k=0}^{\ell_0-1}(f\circ\sigma^m)(\sigma^k(\mathbf{p}_0)) 
    = \sum_{k=0}^{\ell_0-1} f_+(\sigma_+^k(\pi_+(\mathbf{p}_0))), \qquad \forall [\mathbf{p}_0] \in \mathscr{P}_0,$$
for some large power $m\geq 1$. We refer to Lemma~\ref{lem:one-sided-reduction} for the precise statement.

This reduction to a one-sided potential is what allows us to work throughout with the transfer operator $\mathcal{L}_{\Gamma,+}$ on the one-sided space $\mathfrak{P}_\Gamma^+$, rather than having to introduce a resolvent on the bi-infinite shift space $\mathfrak{P}_\Gamma$.

Let us also mention that the discrete (Selberg) trace formula on regular graphs has been the subject of considerable interest over the years, starting with Ahumada~\cite{Ahumada87} up to the recent work of Gong--Li--Liu~\cite{GongLiLiu24}, to which we refer for a historical overview and further references.
Our weighted trace formula \eqref{eq:intro_trace} generalises these results to a general function $f\in C^{\mathrm{lc}}(\mathfrak{P}_\Gamma)$, which is not assumed to be radial or determined by a scalar spectral function, so that the periodic-orbit side retains the full dependence of $f$ on the periodic primitive path $\mathbf{p}_0$ and its shifts $\sigma^k(\mathbf{p}_0)$.

\subsubsection{Meromorphic continuation of the weighted resolvent (Theorem~\ref{thm:mero_resolvent})}
The second key ingredient concerns the meromorphic continuation of the weighted resolvent, generalising the setting of \cite{BHW23} to weighted transfer operators. 
It is worth noting that the proof of Theorem~\ref{thm:mero_dynfct} only requires the meromorphic continuation of the \emph{unweighted} resolvent, which is already established in \cite[Thm.~8.3]{BHW23}. 
Nevertheless, we develop the weighted continuation in full generality, since it %is not more difficult to establish and 
is needed to prove the holomorphic continuation of the dynamical determinant $d_f(s,\beta)= \zeta(s,\beta)^{-1}$  
in Proposition~\ref{prop:dyn_det_Fredholm}.

Concretely, for a weight $\omega \in C^{\mathrm{lc}}(\mathfrak{P}_\Gamma^\pm)$, we introduce the \emph{weighted transfer operator} $\mathcal{L}_{\Gamma,\pm, \omega}$ acting on a suitable Banach space 
$\F_\vartheta(\mathfrak{P}^\pm_\Gamma)$, %of observables on the space of non-backtracking one-sided future (respectively past) infinite paths, 
and consider $\mathbf{R}_{\omega}(s)\coloneqq
        \bigl(\operatorname{Id}
        -
        q^{\frac12+is}\mathcal{L}_{\Gamma,\pm, \omega}\bigr)^{-1}$
on $\F_\vartheta(\mathfrak{P}^\pm_\Gamma)$
the associated weighted resolvent. 
In Theorem~\ref{thm:mero_resolvent}, we show that $\mathbf{R}_{\omega}(s)$ admits a meromorphic continuation to the half-plane \eqref{eq:region}, which depends on $\omega$, thus extending the unweighted case, $\omega \equiv 1$, of \cite[Thm.~8.3]{BHW23} to a genuinely weighted setting. 
Let us also point out that the meromorphic continuation of the unweighted transfer-operator resolvent obtained in \cite[Thm.~8.3]{BHW23} arises in a setting that may be viewed as a dynamical analogue of the classical Ihara theory: in the finite regular case, the poles of the unweighted resolvent are encoded by the zeros of the Ihara–Bass determinant \cite{Hashimoto89}.

Still on the topic of transfer operators, we mention two further directions in which the
present framework could be extended. Hilgert--Kahl--Weich~\cite{HilgertKahlWeich26} have recently introduced and studied transfer operators on Euclidean buildings, a higher-rank generalisation of graphs, acting on a suitable Banach space; and
\cite{ArendsPetersonWeich26} have introduced a framework of geometrically finite graphs, the discrete analogues of geometrically finite hyperbolic surfaces (including funnels and cusps). 
In both settings, it would be interesting to establish a meromorphic continuation of the resolvent via the transfer-operator machinery, along the lines of Theorem~\ref{thm:mero_resolvent} below.
In the geometrically finite case, this would in particular allow one to define the \emph{resonances} of the associated transfer operator.

% -------------------------------%%--------------------------------------%
\subsection*{Structure of the paper}
The paper is organised as follows. 

In Section~\ref{sect:finite_graphs}, we begin by introducing
the geometry of regular graphs and the space of non-backtracking (bi-)infinite paths with their corresponding shift dynamics. 
We then present, in Section~\ref{sect:resonances}, the weighted transfer operator acting on some suitable Banach spaces and develop, analogously to \cite{BHW23}, in Subsection~\ref{sect:function_spaces_key_ineq}, the tools needed to show the meromorphic continuation of the weighted resolvent, which is introduced in Subsection~\ref{sect:weightedRuelletransfer}. 
The final part of the paper, Section~\ref{sect:dyn_zeta}, is devoted to our weighted dynamical zeta function.
After expressing it in terms of traces of transfer operators, we prove its meromorphic continuation (Theorem~\ref{thm:mero_dynfct}) in Subsection~\ref{sect:proof_merom}, and relate it to the dynamical determinants in Subsection~\ref{sect:dyn_determinants}. 
We then derive the associated residue formulae in Subsections~\ref{sect:residue}--\ref{sect:Wigner}; in particular, Subsection~\ref{sect:proof_Thm_residuePS} contains the proof of our main result, Theorem~\ref{thm:residue_PS}.

% -------------------------------%%--------------------------------------%
\subsection*{AI use statement}
ChatGPT-5.5 and Claude Sonnet 5 were used only to assist with proofreading and language improvements of the manuscript's exposition, and to refine the TikZ Figures~\ref{fig:quotient_space} and~\ref{fig:mero_weighted}.
All such assistance was reviewed and manually edited by the author.
%The author made all final decisions and wrote the manuscript entirely. 
%In particular, all mathematical content, theorems, and proofs are entirely the author's own work.

% -------------------------------%%--------------------------------------%
\subsection*{Acknowledgments}
I would like to thank Christian Arends for helpful discussions at the beginning of this project, my colleagues Daniel Kahl and Luca Wasmuth for valuable exchanges on transfer operators and shift spaces,
and Tobias Weich for many helpful comments. Last but not least, I am grateful to Joachim Hilgert for suggesting this interesting problem, which was mentioned in \cite[Problem 6.33]{Hilgert23}.\\
This work is funded by the Deutsche Forschungsgemeinschaft (DFG, German Research Foundation) via the grants SFB-TRR 358/1 2023 - 491392403 (CRC ``Integral Structures in Geometry and Representation Theory'') and WE 6173/1-1 (Emmy Noether group ``Microlocal Methods for Hyperbolic Dynamics'').

% -------------------------------%%--------------------------------------%
\section{Finite regular graphs and shift spaces}  \label{sect:finite_graphs}
In this section, we review the basic notions of finite regular graphs and shift spaces and fix the relevant notation.

% -------------------------------%%--------------------------------------%
\subsection{Regular graphs} \label{sect:regular_graphs}
Let $\mathfrak{G} \coloneqq(\mathfrak{X}, \mathfrak{E})$ be a connected \emph{graph} consisting of a vertex set $\mathfrak{X}$ and a set $\mathfrak{E} \subseteq \mathfrak{X}^2$ of directed edges, which is $(q + 1)$-regular for $q\geq 1$, i.e., every vertex has $q + 1$ neighbors.
Assume that $\mathfrak{G}$ is a \emph{simple graph}, meaning that it
\begin{itemize}
    \item has no loops, i.e., $\mathfrak{E} \cap \; \left\{(x,x) \; \middle\vert \;  x \in \mathfrak{X} \right\} = \emptyset $,
    \item has no multiple edges, and
    \item is symmetric under the switch of vertices, i.e., if $(x,y) \in \mathfrak{E}$, then $(y,x) \in \mathfrak{E}$.
\end{itemize}
For each directed edge $\vec{e} \coloneqq(a,b)$, with $a,b \in \mathfrak{X},$ we denote by $\iota(\vec{e}) \coloneqq a$ the \emph{initial} and $\tau(\vec{e}) \coloneqq b$ the \emph{terminal} vertex of $\vec{e}$. 
Its \emph{opposite} is denoted by $\eop{} \coloneqq (b,a)$.

Recall that a connected graph without cycles is called a \emph{tree}. 
Equivalently, a tree is a connected graph in which any two vertices are joined by a unique path.
Moreover, one can relate a connected $(q+1)$-regular graph to a \emph{$(q+1)$-regular tree} through its universal cover. 
More precisely, one first takes the universal cover of the underlying undirected graph, obtained by identifying every directed edge with its opposite, and then replaces each undirected edge of the cover by a pair of directed edges with opposite orientations. 
This endows the universal cover with an oriented edge structure.

% -------------------------------%%--------------------------------------%
\subsection{Shift spaces} \label{sect:shiftspaces}
Next, we introduce the set of all one-sided infinite chains of edges.
Since $\mathfrak{G}$ is finite, this space is compact.

\begin{definition}[Shift spaces, {\cite[Def.~1.7]{AFH23}}] \label{def:shiftspaces}
    The \emph{shift space} $\mathfrak{P}^+$ is the set of all chains of edges, that is, all infinite sequences $\mathbf{p}_{+}\coloneqq(\vec{e}_1, \vec{e}_2, \vec{e}_3, \dots)$ of directed edges
    such that consecutive edges are
    \emph{concatenated} $\tau(\vec{e}_i) = \iota(\vec{e}_{i+1})$, and satisfy the \emph{non-backtracking condition} $\tau(\vec{e}_{i + 1}) \neq \iota(\vec{e}_i)$,
    for all $i \geq 1$.
    On $\mathfrak{P}^+$, we define the \emph{shift operator} 
    \[
    \sigma_+\colon\mathfrak{P}^+\longrightarrow\mathfrak{P}^+,
    \qquad
    (\vec{e}_1,\vec{e}_2,\vec{e}_3,\ldots)
    \longmapsto
    (\vec{e}_2,\vec{e}_3,\ldots).
    \]
    
    By reversing the orientation, we obtain the \emph{opposite shift space} $\mathfrak{P}^-$, consisting of (left-) infinite, non-backtracking sequences of concatenated edges of the form $\mathbf{p}_{-}\coloneqq(\dots, \vec{e}_2, \vec{e}_1)$.
    Its corresponding shift operator is 
    \[
    \sigma_-\colon\mathfrak{P}^-\longrightarrow\mathfrak{P}^-,
    \qquad
    (\ldots,\vec{e}_3,\vec{e}_2,\vec{e}_1)
    \longmapsto
    (\ldots,\vec{e}_3,\vec{e}_2).
    \]
\end{definition}

On a tree, infinite non-backtracking chains can be interpreted as geodesic rays. 
Two such rays determine the same end if their tails coincide, possibly after a shift of their indices. 
In this sense, $\mathfrak{P}^+$ can be regarded as a graph analogue of a space of forward geodesic rays, and $\sigma_+$ as the corresponding forward-time dynamics.
Thus, the bi-infinite path space is the discrete
analogue of the unit sphere bundle of a Riemannian manifold equipped with the geodesic flow.

We equip $\mathfrak{P}^+$ with the topology induced by the ultrametric
\begin{equation} \label{eq:ultrametric}
    d_{\vartheta}(\mathbf{p}_+, \mathbf{p}'_{+}) \coloneqq \vartheta^{\kappa(\mathbf{p}_+, \mathbf{p}'_{+})},
\end{equation}
where $\vartheta\in(0,1)$ is fixed and, for distinct paths $\mathbf{p}_{+} = (\vec{e}_1, \vec{e}_2, \ldots), \mathbf{p}'_{+} = (\vec{b}_1, \vec{b}_2, \ldots) \in \mathfrak{P}^+$, we set
$\kappa(\mathbf{p}_+, \mathbf{p}'_{+}) \coloneqq \inf\{i \geq 1 : \vec{e}_i \neq \vec{b}_i\}$, see \cite[§5]{BHW23}. 
By convention, $d_\vartheta(\mathbf{p}_+,\mathbf{p}'_{+})=0$ if $\mathbf{p}_+=\mathbf{p}'_{+}$.
Thus, the distance between two paths decreases exponentially with the length of their common initial segment, and it is straightforward to check that $d_\vartheta$ satisfies the strong triangle inequality
$$
d_\vartheta(\mathbf{p}_+, \mathbf{r}_+) \le \max\{d_\vartheta(\mathbf{p}_+,\mathbf q_+), d_\vartheta(\mathbf q_+,\mathbf{r}_+)\}, \quad \mathbf{p}_+,\mathbf q_+,\mathbf{r}_+ \in \mathfrak{P}^+
$$
so that $d_\vartheta$ is indeed an ultrametric on $\mathfrak{P}^+$.

The space $\mathfrak{P}^-$ is equipped with the analogous ultrametric, defined by comparing edges
starting from the right end of the chain.

With this notation, we define the \emph{space of bi-infinite chains}.

\begin{definition}[Space of bi-infinite chains, {\cite[§~1.3]{AFH23Pairing}}]
\label{def:space_biinfinite_chains}
    The \emph{space of bi-infinite chains} is defined as the set of pairs consisting of a forward chain $\mathbf{p}_+ \in \mathfrak{P}^+$, representing the future, and a backward chain $\mathbf{p}_- \in \mathfrak{P}^-$, representing the past:
    \begin{equation} \label{eq:space_biinfinite_chains}
      \mathfrak{P} \coloneqq \left\{
        (\mathbf{p}_+, \mathbf{p}_-) \in \mathfrak{P}^+ \times \mathfrak{P}^-
        \mid
        \iota\!\left(\pi^{\mathfrak{E}}_+(\mathbf{p}_+)\right)
        =
        \tau\!\left(\pi^{\mathfrak{E}}_-(\mathbf{p}_-)\right), \pi^{\mathfrak{E}}_+(\mathbf{p}_+) \neq \!\left(\pi^{\mathfrak{E}}_-(\mathbf{p}_-)\right)^{\mathrm{op}}
        \right\} \subset \mathfrak{P}^+ \times \mathfrak{P}^-,
    \end{equation}
    where $\pi^{\mathfrak{E}}_\pm\colon \mathfrak{P}^\pm \rightarrow \mathfrak{E}$ are the projections onto the first oriented edge of a chain.

    For 
    $(\mathbf{p}_+, \mathbf{p}_-) \in \mathfrak{P}$ with 
    $\mathbf{p}_+ = (\vec{e}_1, \vec{e}_2, \ldots) \in \mathfrak{P}^+$ 
    and 
    $\mathbf{p}_- = (\ldots, \vec{b}_2, \vec{b}_1) \in \mathfrak{P}^-$, 
    the shift $\sigma$ on $\mathfrak{P}$ is defined by
    $$\sigma(\mathbf{p}_+, \mathbf{p}_-) \coloneqq  \bigl((\vec{e}_2, \vec{e}_3, \ldots),\ (\ldots, \vec{b}_2, \vec{b}_1, \vec{e}_1)\bigr),
    $$
    where
    $\sigma_+(\mathbf{p}_+) = (\vec{e}_2, \vec{e}_3, \ldots)$ drops the first edge of $\mathbf{p}_+$, shifting the basepoint from $\iota(\vec{e}_1)$ to $\tau(\vec{e}_1) = \iota(\vec{e}_2)$, and $(\ldots, \vec{b}_2, \vec{b}_1, \vec{e}_1)$ appends $\vec{e}_1$ to the right end of $\mathbf{p}_-$, recording $\vec{e}_1$ as the most recently traversed edge in the past.
\end{definition}
Thus, $\sigma$ moves the basepoint one step \emph{forward} along the future chain.

Note that the two conditions in \eqref{eq:space_biinfinite_chains} ensure that the two one-sided chains share a basepoint vertex and do not immediately backtrack into each other.\\
We endow $\mathfrak{P}$ with the corresponding subspace topology induced from the product topology on $\mathfrak{P}^+ \times \mathfrak{P}^-.$

\begin{rem} \label{rem:biinfinite_chains}
    \begin{itemize}
        \item[(i)] The description in \cite[§2.2.3]{ArendsPalmirotta26} for two forward chains $(\mathbf{p}_+,\mathbf{p}'_+)\in({\mathfrak{P}^{+}})^2$
        is equivalent to the description \eqref{eq:space_biinfinite_chains} of $\mathfrak{P}$. 
        Indeed,
        the second component $\mathbf{p}'_+ \in \mathfrak{P}^{+}$ encodes the past of a bi-infinite chain by running forward in the opposite direction from the basepoint, which is by definition $\mathbf{p}_- \in \mathfrak{P}^{-}$.
        \item[(ii)] The shift $\sigma$ cannot be defined componentwise as $(\sigma_+(\mathbf{p}_+), \sigma_-(\mathbf{p}_-))$ since it would violate the shared basepoint condition $\iota\!\left(\pi^{\mathfrak{E}}_+(\mathbf{p}_+)\right)
        =
        \tau\!\left(\pi^{\mathfrak{E}}_-(\mathbf{p}_-)\right)$.
        In fact, 
        the shift map $\sigma_+$ to $\mathbf{p}_+$ would drop the first edge $\vec{e}_1$ and move the basepoint of $\mathbf{p}_+$ forward from $\iota(\vec{e}_1)$ to $\iota(\vec{e}_2)=\tau(\vec{e}_1)$,
        whereas applying $\sigma_-$ to $\mathbf{p}_-$ would drop the last edge $\vec{b}_1$ and move the basepoints of $\mathbf{p}_-$ from $\iota(\vec{b}_1)$ to $\iota(\vec{b}_2)=\tau(\vec{b}_2)$. 
        In general, however, these vertices do not coincide 
        $\tau(\vec{e}_1) \neq \tau(\vec{b}_1)$, 
        since $\mathbf{p}_+$ and $\mathbf{p}_-$ start at the same vertex $\iota(\vec{e}_1) = \iota(\vec{b}_1)$ but go in different directions.
        \item[(iii)] The shift map
        $\sigma: \mathfrak{P} \longrightarrow \mathfrak{P}$ is well-defined and bijective.
        Indeed, two defining conditions of $\mathfrak{P}$ are preserved:
        \begin{itemize}
            \item the new chains share the basepoint because $\iota(\vec{e}_2) = \tau(\vec{e}_1)$, and the last edge of $(\ldots, \vec{b}_2, \vec{b}_1, \vec{e}_1)$ is $\vec{e}_1$ with $\tau(\vec{e}_1) = \iota(\vec{e}_2)$,
            \item there is no backtracking at the new basepoint because $\vec{e}_2\neq\vec{e}_1^{\,\mathrm{op}}$
            by the non-backtracking condition on $\mathbf{p}_+$.
        \end{itemize}
        Moreover, for $(\mathbf{p}_+, \mathbf{p}_-) \in \mathfrak{P}$ with $\mathbf{p}_+ = (\vec{e}_1, \vec{e}_2, \ldots) \in \mathfrak{P}^+$ and 
        $\mathbf{p}_- = (\ldots, \vec{b}_2, \vec{b}_1) \in \mathfrak{P}^-$, note that the inverse of $\sigma$ is
        $$\sigma^{-1}(\mathbf{p}_+, \mathbf{p}_-) = \bigl((\vec{b}_1,\vec{e}_1, \vec{e}_2, \ldots),\ (\ldots, \vec{b}_3, \vec{b}_2)\bigr).$$
        So $\sigma^{-1}$ moves the basepoint one step \emph{backward} along the future edge $\vec{b}_1$.
            \end{itemize}
\end{rem}

% -------------------------------%%--------------------------------------%
\subsection{Quotient spaces} \label{sect:qoutient_spaces}

Let $\widetilde{\mathfrak{G}} \coloneqq (\widetilde{\mathfrak{X}},\widetilde{\mathfrak{E}})$
be a $(q+1)$-regular tree, and let $\Gamma \leq G \coloneqq \mathrm{Aut}(\widetilde{\mathfrak{G}})$ be a discrete subgroup of automorphisms of $\widetilde{\mathfrak{G}}$, acting freely, without edge inversions, and cocompactly on $\widetilde{\mathfrak{G}}$. 
In this setting, passing to the quotient by $\Gamma$ produces \emph{finite $(q+1)$-regular graphs} $ \mathfrak{G}_\Gamma \coloneqq \Gamma\backslash\widetilde{\mathfrak{G}}$,
whose universal cover is $\widetilde{\mathfrak{G}}$. If needed, we additionally assume that $\mathfrak{G}_\Gamma$ is simple.
More precisely, the quotient graph
$\mathfrak{G}_\Gamma = (\mathfrak{X}_\Gamma, \mathfrak{E}_\Gamma)$
is defined 
\begin{itemize}
    \item vertex-wise 
    $\mathfrak{X}_\Gamma \coloneqq \Gamma\backslash\widetilde{\mathfrak{X}} =\{\Gamma x\mid x\in\widetilde{\mathfrak{X}}\}$ and
    \item edge-wise 
    $\mathfrak{E}_\Gamma \coloneqq \Gamma\backslash\widetilde{\mathfrak{E}} =\{\Gamma\vec{e}\mid\vec{e}\in\widetilde{\mathfrak{E}}\}.$
\end{itemize}
The quotient maps are
$$\pi_{\mathfrak{X}} \colon
\widetilde{\mathfrak{X}} \longrightarrow \mathfrak{X}_\Gamma, x \mapsto \Gamma x \quad \text{ and } \quad \pi_{\mathfrak{E}}\colon
\widetilde{\mathfrak{E}}\longrightarrow\mathfrak{E}_\Gamma, \vec{e} \mapsto \Gamma \vec{e}$$
and satisfy 
$\pi_{\mathfrak{X}} \circ \gamma = \pi_{\mathfrak{X}},$ for every $\gamma \in \Gamma.$
The incidence maps on the quotient
$
\iota_\Gamma(\Gamma\vec{e})
\coloneqq
\Gamma\iota(\vec{e})$ and
$\tau_\Gamma(\Gamma\vec{e})
\coloneqq
\Gamma\tau(\vec{e})
$
are well-defined and satisfy 
$\iota_\Gamma\circ\pi_{\mathfrak{E}}= \pi_{\mathfrak{X}}\circ\iota,$ and 
$\tau_\Gamma\circ\pi_{\mathfrak{E}} = \pi_{\mathfrak{X}}\circ\tau.$

Since the action is free, without inversions, and cocompact, the quotient
map
$$
\pi\colon
\widetilde{\mathfrak{G}}\longrightarrow\mathfrak{G}_\Gamma
$$
is a graph covering, and $\Gamma$ is its group of deck transformations, see Figure~\ref{fig:quotient_space}.
In particular, the covering map preserves the local structure at every
vertex, so $\mathfrak{G}_\Gamma$ is $(q+1)$-regular. Cocompactness implies
that $\mathfrak{G}_\Gamma$ has finitely many vertices and edges.

The tree $\widetilde{\mathfrak{G}}$ is called a \emph{homogeneous tree} because every vertex has the same degree $q+1$.
As a tree, it is $0$-hyperbolic in the sense of Gromov, meaning that all triangles are ``thin''. 
Notice that these statements concern the covering tree $\widetilde{\mathfrak{G}}$, not in general the quotient graph $\mathfrak{G}_\Gamma$.

% -------------------------------%%--------------------------------------%
\subsubsection{Quotient shift spaces} \label{sect:qoutient_shift_spaces}

The action of $\Gamma$ on $\widetilde{\mathfrak{G}}$ induces actions on
the associated shift spaces. 
Indeed, the action of $\Gamma$ on $\widetilde{\mathfrak{X}}$ induces an
action on $\widetilde{\mathfrak{E}}$ by
$\gamma(x,y)\coloneqq(\gamma x,\gamma y),$
and hence an action on $\widetilde{\mathfrak{P}}^+$ given by
$
\gamma\cdot(\vec{e}_1,\vec{e}_2,\ldots)
\coloneqq
(\gamma\vec{e}_1,\gamma\vec{e}_2,\ldots).
$
The action on $\widetilde{\mathfrak{P}}^-$ is defined analogously, and
$\Gamma$ acts diagonally on
$
\widetilde{\mathfrak{P}}
\subseteq
\widetilde{\mathfrak{P}}^+\times\widetilde{\mathfrak{P}}^-.$\\
We define the corresponding quotient spaces
by
$$
\mathfrak{P}^\pm_\Gamma
\coloneqq
\Gamma\backslash\widetilde{\mathfrak{P}}^\pm,
\qquad
\mathfrak{P}_\Gamma
\coloneqq
\Gamma\backslash\widetilde{\mathfrak{P}}.
$$
Specifically, the quotient map on edges  $\pi_{\mathfrak{E}}$ induces a map on forward chains
$$
\pi_{\mathfrak{E}}^{+}\colon \widetilde{\mathfrak{P}}^{+} \longrightarrow \mathfrak{P}^{+}_{\mathfrak{G}_\Gamma},
\qquad
\mathbf{p}_+ \mapsto (\pi_{\mathfrak{E}}(\vec{e}_1), \pi_{\mathfrak{E}}(\vec{e}_2), \dots),
$$
where $\mathfrak{P}^{+}_{\mathfrak{G}_\Gamma}$ denotes the space of one-sided infinite non-backtracking chains of edges in $\mathfrak{E}_\Gamma$.
Since $\Gamma$ acts freely on $\mathfrak{G}$, the projection
$\pi_{\mathfrak{E}}$ preserves the concatenation and non-backtracking condition, so the map $\pi_{\mathfrak{E}}^{+}$ is well-defined.
Note that every chain in $\mathfrak{P}^+_{\mathfrak{G}_\Gamma}$ admits a lift to
$\widetilde{\mathfrak{P}}^+$.

Moreover, two lifted chains $(\vec{e}_i)_{i\geq 1}$ and $(\vec{e}^{\;'}_i)_{i\geq 1}$ have the same image under $\pi_{\mathfrak{E}}^{+}$ if and only if they differ by the action of a single element of $\Gamma$, i.e., there exists a unique $\gamma \in \Gamma$ such that $\vec{e}_i^{\;'}=\gamma \vec{e}_i$ for all $i\geq 1.$
Consequently, $\pi_{\mathfrak{E}}^{+}$ descends to a bijection 
$\Gamma \backslash \widetilde{\mathfrak{P}}^+ \stackrel{\sim}{\longrightarrow} \mathfrak{P}^{+}_{\mathfrak{G}_\Gamma}.$
With respect to the quotient topology on the left and the shift-space
topology on the right, this bijection is a homeomorphism, as $\pi_{\mathfrak{E}}$ is a local homeomorphism.
Furthermore, the shift map $\sigma_+$ commutes with the projection $\pi_{\mathfrak{E}}^{+}$, i.e.,
$\pi_{\mathfrak{E}}^+\circ\widetilde{\sigma}_+ =\sigma_+\circ\pi_{\mathfrak{E}}^+,$
so the induced homeomorphism is shift-equivariant. 

The same argument applies to the backward $\mathfrak{P}^-$ and bi-infinite path spaces $\mathfrak{P}$. 
We therefore obtain
the natural shift-equivariant identifications
$\mathfrak{P}^\pm_\Gamma \cong \mathfrak{P}^\pm_{\mathfrak{G}_\Gamma}$ and $\mathfrak{P}_\Gamma \cong \mathfrak{P}_{\mathfrak{G}_\Gamma}$.\\

From now on, we will omit the tilde, except when it is needed to distinguish the universal cover from the underlying space.
% -------------------------------%%--------------------------------------%
\section{Resonances for finite graphs} \label{sect:resonances}

To investigate resonances for finite graphs, we introduce a weighted version of the Ruelle transfer operators from \cite[§3]{BHW23} and analyse their action on the associated shift spaces $\mathfrak{P}^\pm$ as well as on suitable Banach spaces of functions.
This functional-analytic framework allows us to define \emph{resonances} 
as the isolated eigenvalues of finite algebraic multiplicity lying outside
the essential spectrum of the transfer operator.

% -------------------------------%%--------------------------------------%
\subsection{Weighted transfer operators} \label{sect:weightedRuelletransfer}
Let $\mathrm{Maps}(\mathfrak{P}^\pm, \C) \coloneqq \C^{\mathfrak{P}^\pm}$ be the vector space of complex-valued functions $\varphi\colon \mathfrak{P}^\pm \rightarrow \C$ equipped with pointwise operations.

\begin{definition}[Weighted Ruelle transfer operator] \label{def:Ruelle_transfer_op}
    Let
    $\omega \colon \mathfrak{P}^+ \to \mathbb{C}$
    be a potential depending only on finitely many initial edges.
    On the space $\mathrm{Maps}(\mathfrak{P}^+ , \mathbb{C})$, we define the \emph{weighted Ruelle transfer operator}
    $\mathcal{L}_{+,\omega} \colon \mathrm{Maps}(\mathfrak{P}^+ , \mathbb{C}) \rightarrow \mathrm{Maps}(\mathfrak{P}^+ , \mathbb{C})$ with potential $\omega$ as
    \begin{equation*}
        \mathcal{L}_{+,\omega} \varphi(\mathbf{p}_{+}) 
        \coloneqq 
        \sum_{\substack{\vec{e}_1 \colon \tau(\vec{e}_1)= \iota(\vec{e}_2)\\
        \vec{e}_1\neq\vec{e}_2^{\,\mathrm{op}}}}
        \omega (\mathbf{p}'_{+}) \varphi(\mathbf{p}'_{+}) 
        = \sum_{\sigma_+(\mathbf{p}'_+)=\mathbf{p}_+} \omega(\mathbf{p}'_{+}) \varphi(\mathbf{p}'_{+}).
    \end{equation*}
    where $\sigma_+$ is the shift operator in Definition~\ref{def:shiftspaces}.
    %where $\mathbf{p}'_{+}=(\vec{e}_0 , \vec{e}_1 , \ldots) \in \mathfrak{P}^+ $.
    
    Reversing the orientations of the edge paths, we get the operator $\mathcal{L}_{-,\omega}$ on $\mathrm{Maps}(\mathfrak{P}^-, \mathbb{C})$ with the corresponding potential $\omega$ on $\mathfrak{P}^-$ and shift operator $\sigma_-.$
\end{definition}

\begin{rem}[Two-sided weighted transfer operator on $\mathfrak{P}$]
    The transfer operator on bi-infinite chains is defined by
    \begin{eqnarray*}
            \mathscr{L}_{\omega, \text{future}}: \mathrm{Maps}(\mathfrak{P} , \mathbb{C}) &\longrightarrow& \mathrm{Maps}(\mathfrak{P} , \mathbb{C}) \\
            \varphi &\longmapsto& \sum_{\sigma(\mathbf{p}')=\mathbf{p}} \omega(\mathbf{p}') \varphi(\mathbf{p}')=\omega(\sigma^{-1}(\mathbf{p})) \varphi(\sigma^{-1}(\mathbf{p})),
            \qquad \mathbf{p},\mathbf{p}' \in \mathfrak{P},
    \end{eqnarray*}
    where $\sigma$ is the shift operator in Definition~\ref{def:space_biinfinite_chains}.
    Since $\sigma$ is invertible, there is no sum over preimages. 
    Hence
    $\mathscr{L}_{1, \text{future}}=\sigma^{-1}_*.$
    Such transfer operators are commonly called (weighted) \emph{future shift pullback or composition operators} on $\mathrm{Maps}(\mathfrak{P} , \mathbb{C})$, because they induce from the \emph{future} positive one-sided space $\mathfrak{P}^+.$

    Similarly, we call weighted transfer operators coming from the \emph{past} negative one-sided space $\mathfrak{P}^-$
    $$\mathscr{L}_{\omega, \text{past}} \coloneqq \omega(\sigma(\mathbf{p})) \varphi(\sigma(\mathbf{p})),  \qquad \mathbf{p},\mathbf{p}' \in \mathfrak{P},$$
    the \emph{past shift pullback operators} on $\mathrm{Maps}(\mathfrak{P} , \mathbb{C})$ and thus $\mathscr{L}_{1, \text{past}}=\sigma_*.$
\end{rem}

Note that the weighted Ruelle transfer operator $\mathcal{L}_{\pm, \omega}$ can be viewed as
\begin{equation} \label{eq:decomposition}
    \mathcal{L}_{\pm, \omega}\varphi=\mathcal{L}_{\pm,1}(\omega \varphi),
\quad \varphi \in \mathrm{Maps}(\mathfrak{P}^\pm, \mathbb{C}),
\end{equation}
where $\mathcal{L}_{\pm,1} \eqqcolon \mathcal{L}_{\pm}$ is the classical (unweighted) Ruelle transfer operator in {\cite[§3]{BHW23} and \cite[§1.2]{AFH23}}.

To give the weighted transfer operators $\mathcal{L}_{\pm,\omega}$ a dual interpretation, analogous to that of the unweighted Ruelle transfer operator in \cite[§9]{BHW23}, we first introduce the \emph{space of locally constant functions} on $\mathfrak{P}^\pm$.
A function $\varphi \in \mathrm{Maps}(\mathfrak{P}^\pm, \mathbb{C})$ is said to be \emph{locally constant} if $\mathfrak{P}^\pm$ admits a cover by
districts on each of which $\varphi$ is constant.
Here a \emph{district} is the set of all paths sharing a fixed finite initial segment.
Since $\mathfrak{P}^\pm$ is compact, such a cover can be chosen to be finite.
We denote the resulting space by $C^{\mathrm{lc}}(\mathfrak{P}^\pm)$.

\begin{lem} \label{lem:invariant_lc_space}
    The weighted transfer operator $\mathcal{L}_{\pm, \omega}$ leaves the space $C^{\mathrm{lc}}(\mathfrak{P}^\pm)$ invariant.
\end{lem}

\begin{proof}
    Since $\omega$ depends only on finitely many initial edges, it belongs to
    $C^{\mathrm{lc}}(\mathfrak{P}^{\pm})$. 
    Hence multiplication by $\omega$ preserves $C^{\mathrm{lc}}(\mathfrak{P}^{\pm})$, i.e., 
    $\omega\varphi\in C^{\mathrm{lc}}(\mathfrak{P}^\pm).$
    
    Moreover, by \cite[Eq.~(4) in §9]{BHW23}, the unweighted transfer operator $\mathcal{L}_{\pm,1}$ leaves this space invariant.
    Therefore, using the composition \eqref{eq:decomposition}, we obtain
    $
    \mathcal{L}_{\pm,\omega}\varphi
    =\mathcal{L}_{\pm,1}(\omega\varphi)\in C^{\mathrm{lc}}(\mathfrak{P}^\pm)$ and thus
    \[
    \mathcal{L}_{\pm,\omega}
     \bigl(C^{\mathrm{lc}}(\mathfrak{P}^{\pm})\bigr)
     \subseteq C^{\mathrm{lc}}(\mathfrak{P}^{\pm}),
    \]
    as claimed.
\end{proof}

Following the notation of \cite[§9]{BHW23}, we denote the algebraic dual of
$C^{\mathrm{lc}}(\mathfrak{P}^\pm)$ by $\mathcal{D}'(\mathfrak{P}^\pm) \coloneqq (C^{\mathrm{lc}}(\mathfrak{P}^\pm))'$ and the dual of the operator $\mathcal{L}_{\pm,\omega}$ by
    \begin{equation*}
      \mathcal{L} '_{\pm,\omega} \colon \mathcal{D}'(\mathfrak{P}^\pm) \rightarrow \mathcal{D}'(\mathfrak{P}^\pm),
    \end{equation*}
which is well-defined by Lemma~\ref{lem:invariant_lc_space}.

Let $M_f$ denote multiplication by a locally constant function $f\in C^{\mathrm{lc}}(\mathfrak{P}^\pm)$
\begin{equation} \label{eq:multipication_f}
    M_f \varphi = f\varphi \qquad \forall \varphi \in \mathrm{Maps}(\mathfrak{P}^\pm, \C).
\end{equation}
We have the following fundamental covariance identity.

\begin{lem}[Transfer operator covariance] \label{lem:L_covariance}
    For every $f\in C^{\mathrm{lc}}(\mathfrak{P}^\pm)$
    $$\mathcal{L}_{\pm,\omega} M_{f \circ \sigma_\pm} = M_f \mathcal{L}_{\pm,\omega}.$$
\end{lem}

\begin{proof}
    By definition, for $\mathbf{p}_\pm \in \mathfrak{P}^\pm$ and $\varphi \in C^{\mathrm{lc}}(\mathfrak{P}^\pm)$, we obtain
    \begin{eqnarray*}
        (\mathcal{L}_{\pm,\omega} M_{f \circ \sigma_\pm} \varphi)(\mathbf{p}_\pm)
        = \sum_{\sigma_\pm(\mathbf{p}'_\pm)=\mathbf{p}_\pm}
        \omega(\mathbf{p}'_\pm)f(\sigma_\pm(\mathbf{p}'_\pm))\varphi(\mathbf{p}'_\pm)
        &=& f(\mathbf{p}_\pm) \sum_{\sigma_\pm(\mathbf{p}'_\pm)=\mathbf{p}_\pm}
        \omega(\mathbf{p}'_\pm)f\varphi(\mathbf{p}'_\pm) \\
        &=& (M_f\mathcal{L}_{\pm,\omega} \varphi)(\mathbf{p}_\pm). \qquad \qquad \qedhere
    \end{eqnarray*}
\end{proof}

% -------------------------------%%--------------------------------------%
\subsection{Function spaces and key inequality} \label{sect:function_spaces_key_ineq}
As mentioned in \cite{BHW23} and in the spirit of \cite[§1.3]{Baladi20}, to derive the intrinsic discrete spectrum, known as \emph{Ruelle resonances}, we need to consider the action of the weighted Ruelle transfer operator $\mathcal{L}_{\pm, \omega}$ on function spaces with certain regularity.
A standard choice is the scale of Banach spaces of Lipschitz continuous functions.

\begin{definition}[Lipschitz continuous functions {\cite[Def.~5.2]{BHW22}}] \label{def:F_vartheta}
    For some $0 < \vartheta < 1$, we define the \emph{space of Lipschitz continuous functions} on $\mathfrak{P}^\pm$ by
    $$
    \F_{\vartheta}(\mathfrak{P}^\pm) \coloneqq \left\{ f\colon \mathfrak{P}^\pm \rightarrow \C \;\middle|\; \exists C_f > 0 \;\text{s.t.}\; \forall \mathbf{p}_\pm, \mathbf{p}'_\pm \in \mathfrak{P}^\pm,\; |f(\mathbf{p}_\pm) - f(\mathbf{p}'_\pm)| \leq C_f\, d_{\vartheta}(\mathbf{p}_\pm, \mathbf{p}'_\pm) \right\},
    $$
    where $d_{\vartheta}$ on $\mathfrak{P}^\pm$ is the metric defined in \eqref{eq:ultrametric}.
    %$$
    Its norm is given by
    \begin{equation} \label{eq:norm_vartheta}
        \|f\|_{\vartheta} 
        \coloneqq |f|_{\vartheta} + \|f\|_\infty, 
        \qquad f \in \F_{\vartheta}(\mathfrak{P}^\pm),
    \end{equation}
    where $\|\cdot \|_\infty$ is the supremum norm and
    \begin{equation} \label{eq:Hölder_seminorm}
        |f|_{\vartheta} \coloneqq 
    \inf\{C_f \;|\; f \text{ is $C_f$-Lipschitz} \}
    =
    \sup_{\mathbf{p}_\pm \neq \mathbf{p}'_\pm} \frac{|f(\mathbf{p}_\pm) - f(\mathbf{p}'_\pm)|}{d_\vartheta(\mathbf{p}_\pm,\mathbf{p}'_\pm)}.
    \end{equation}
\end{definition}

By \cite[Prop.~5.4]{BHW23}, $(\F_{\vartheta}(\mathfrak{P}^\pm), \| \cdot \|_{\vartheta})$ is a Banach space for each $\vartheta \in (0,1).$
Furthermore, for $\vartheta \in (0,1)$, we have $C^{\mathrm{lc}}(\mathfrak{P}^\pm) \subseteq \F_{\vartheta}(\mathfrak{P}^\pm)$.
Consequently, its topological dual $\F'_{\vartheta}(\mathfrak{P}^\pm)$ is a subset of  $\mathcal{D}'(\mathfrak{P}^\pm)$ and is again a Banach space \cite[Rems.~5.3 and 5.6]{BHW22}.

\begin{rem}
    If $0<\vartheta<\vartheta'<1$, then
    $$\F_\vartheta(\mathfrak{P}^\pm) \subseteq \F_{\vartheta'}(\mathfrak{P}^\pm).$$
    Thus, as $\vartheta$ \emph{increases}, the space $\F_{\vartheta}(\mathfrak{P}^\pm)$ becomes larger and the corresponding regularity requirement becomes \emph{weaker}.
    While its dual $\F'_{\vartheta}(\mathfrak{P}^\pm)$ becomes \emph{smaller}, \cite[p.~29]{BHW23}.
    Consequently, the spaces $\F_{\vartheta}(\mathfrak{P}^\pm)$ can be seen as spaces with increasing regularity and the locally constant functions belong to $\F_\vartheta(\mathfrak{P}^\pm)$ for every $\vartheta\in(0,1)$, \cite[Rem.~5.3]{BHW22}.
\end{rem}

With this in mind, consider the \emph{space of bounded functions} $\mathcal{B}(\mathfrak{P}^\pm,\C) \subseteq \mathrm{Maps}(\mathfrak{P}^\pm, \C)$.
Equipped with the supremum norm $\| \cdot \|_\infty$, it turns into a Banach space.\\
Following \cite[§4 and §9]{BHW23}, the space $C^{\mathrm{lc}}(\mathfrak{P}^\pm)$ is filtered by the finite-dimensional subspaces 
$$
F_m(\mathfrak{P}^\pm) \coloneqq \left\{ \varphi \in \mathcal{B}(\mathfrak{P}^\pm,\C) \;\middle|\; \varphi \text{ depends only on the first $m$ edges}\right\}, \qquad m \in \N,
$$
of \emph{locally constant functions of size $m$},
according to the size of the initial segment on which a locally constant function depends.
In other words, $\varphi \in F_m(\mathfrak{P}^\pm)$ if and only if it is constant on every district of level $m$, namely on
every set of paths in $\mathfrak{P}^\pm$ whose first $m$ directed edges agree with a fixed admissible finite edge-sequence $(\vec{e}_1, \ldots, \vec{e}_m)$. 
Here \emph{admissibility} means that the sequence is non-backtracking and concatenated, as in the definition of $\mathfrak{P}^\pm$.
Clearly $F_m(\mathfrak{P}^\pm) \subseteq F_{m+1}(\mathfrak{P}^\pm)$ for every $m$, and by compactness of $\mathfrak{P}^\pm$ every locally constant function is constant on the districts of \emph{some} finite level, so that
\begin{equation} \label{eq:Clc_Fm}
    C^{\mathrm{lc}}(\mathfrak{P}^\pm) = F_\infty(\mathfrak{P}^\pm) \coloneqq \bigcup_{m=1}^\infty F_m(\mathfrak{P}^\pm),
\end{equation}
see \cite[Prop.~4.2~(i)]{AFH23}.
Notice that $F_1(\mathfrak{P}^\pm) \cong \mathrm{Maps}(\mathfrak{E},\C)$, \cite[Rem.~4.1.~(i)]{AFH23}.
The spaces $F_m(\mathfrak{P}^\pm)$ enjoy several important structural properties. 
Namely, for every $m\in \N$, the space $F_m(\mathfrak{P}^\pm)$ is a closed subspace of both, the Banach space $\mathcal{B}(\mathfrak{P}^\pm,\mathbb{C})$ \cite[Lem.~4.1]{BHW23} and, for every $\vartheta\in(0,1)$, of the Banach space $\F_{\vartheta}(\mathfrak{P}^\pm)$ \cite[Lem.~5.2]{BHW23}.
Consequently, we have the following inclusions between the function spaces
\begin{equation} \label{eq:inclusion}
    F_1(\mathfrak{P}^\pm) \subseteq 
    F_2(\mathfrak{P}^\pm) \subseteq \cdots \subseteq \F_\vartheta(\mathfrak{P}^\pm) \subseteq 
    \F_{\vartheta'}(\mathfrak{P}^\pm) \subseteq 
    \mathcal{B}(\mathfrak{P}^\pm,\mathbb{C}) \subseteq 
    \mathrm{Maps}(\mathfrak{P}^\pm,\mathbb{C})
\end{equation}
for $0<\vartheta<\vartheta'<1$.

Furthermore, by \cite[Cor.~5.5]{BHW23} and Lemma~\ref{lem:invariant_lc_space}, we have the following invariance.

\begin{lem} \label{lem:invariant_F}
    Suppose that $\omega\in F_r(\mathfrak{P}^\pm)$ for some $r\geq1$.
    Then, for every $m\geq1$
    $$
        \mathcal{L}_{\pm,\omega}\bigl(F_m(\mathfrak{P}^\pm)\bigr)
        \subseteq
        F_{\max\{m,r\}-1}(\mathfrak{P}^\pm),
    $$
    where $F_0$ is interpreted as the space of constant functions.
    In particular, $F_m(\mathfrak{P}^\pm)$ is
    $\mathcal{L}_{\pm,\omega}$-invariant whenever $m\geq r$.
\end{lem}

\begin{proof}
    If $\varphi\in F_m(\mathfrak{P}^\pm)$ and $\omega\in F_r(\mathfrak
    P^\pm)$, then
    $
    \omega\varphi\in F_{\max\{m,r\}}(\mathfrak{P}^\pm).
    $
    The unweighted transfer operator lowers the level of local constancy by
    one. Therefore, by \eqref{eq:decomposition},
    \[
    \mathcal{L}_{\pm,\omega}\varphi
    = \mathcal{L}_{\pm,1}(\omega\varphi) \in F_{\max\{m,r\}-1}(\mathfrak{P}^\pm).
    \]
    If $m\geq r$, then
    $
    F_{\max\{m,r\}-1}(\mathfrak{P}^\pm)=F_{m-1}(\mathfrak{P}^\pm)\subseteq F_m(\mathfrak{P}^\pm),
    $
    which proves the invariance.
\end{proof}

% -------------------------------%%--------------------------------------%
\subsubsection{Key inequality and approximation by projection} \label{sect:key_ineq}
Next, we will study the action of the weighted transfer operator on these different function spaces and thus prove a suitable \emph{key inequality} as for the unweighted case \cite[Cor.~5.8, 6.2 and 7.2]{BHW23}.

We begin by showing that $\mathcal{L}_{\pm, \omega}$ is a bounded operator with operator norm not exceeding $q \|\omega\|_\infty.$

\begin{prop} \label{prop:L_bound}
    Let $\mathfrak{G}$ be a finite $(q+1)$-regular graph and
    $\omega \in C^{\mathrm{lc}}(\mathfrak{P}^\pm)$. %, i.e.,
    %$\|\omega\|_\infty < \infty$, 
    Then
    $$
    \|\mathcal{L}_{\pm,\omega} \varphi\|_{\infty}
    \leq
    q \|\omega\|_\infty \|\varphi\|_\infty, \qquad \varphi \in \mathcal{B}(\mathfrak{P}^\pm, \mathbb{C}).
    $$
    Hence $\mathcal{L}_{\pm,\omega}$ leaves $(\mathcal{B}(\mathfrak{P}^\pm, \mathbb{C}), \|\cdot\|_\infty)$ invariant
     and $\|\mathcal{L}_{\pm,\omega}\|_{\mathrm{op}, \infty} \leq q \|\omega\|_\infty.$
\end{prop}
\begin{proof}
    For every $\mathbf{p}_\pm \in \mathfrak{P}^\pm$ and $\varphi \in \mathcal{B}(\mathfrak{P}^\pm, \mathbb{C})$, we have
    \begin{eqnarray*}
        |\mathcal{L}_{\pm,\omega}\varphi(\mathbf{p}_\pm)|
        \leq \sum_{\sigma_\pm(\mathbf{p}'_\pm)=\mathbf{p}_\pm}  |\omega(\mathbf{p}'_{\pm}) \varphi(\mathbf{p}'_{\pm})|
        \leq \sum_{\sigma_\pm(\mathbf{p}'_\pm)=\mathbf{p}_\pm} \|\omega\|_\infty \|\varphi\|_\infty
        = q\|\omega\|_\infty \|\varphi\|_\infty,
    \end{eqnarray*}
    because $\mathbf{p}_\pm$ has exactly $q$ admissible non-backtracking
    one-edge predecessors. Taking the supremum over
    $\mathbf{p}_\pm\in\mathfrak{P}^\pm$ proves the assertion.
\end{proof}

In particular, by \eqref{eq:inclusion}, Proposition~\ref{prop:L_bound}
applies also to $\F_\vartheta(\mathfrak{P}^\pm)$ and thus
$\mathcal{L}_{\pm,\omega}$ is a bounded operator on $(\F_\vartheta(\mathfrak{P}^\pm), \|\cdot\|_\infty)$.

After this, for $n\geq 1$, define
\begin{equation} \label{eq:w_m_product}
     \omega_n(\mathbf{p}_\pm)
    \coloneqq
    \prod_{k=0}^{n-1}
    \omega\bigl(\sigma_\pm^k(\mathbf{p}_\pm)\bigr), \quad \forall \mathbf{p}_\pm \in \mathfrak{P}^\pm,
\end{equation}
which is the \emph{Birkhoff product} of $\omega$ along the length-$n$ orbit segment $\mathbf{p}_\pm, \sigma(\mathbf{p}_\pm), \ldots, \sigma^{n-1}_\pm(\mathbf{p}_\pm)$.
By induction on $n$, we obtain the $n$-th power of $\mathcal{L}_{\pm, \omega}$:
\begin{equation} \label{eq:iterated_transfer_op}
    \mathcal{L}_{\pm,\omega}^n\varphi(\mathbf{p}_\pm)
    =
    \sum_{\sigma_\pm^n(\mathbf{p}'_\pm)=\mathbf{p}_\pm}
    \omega_n(\mathbf{p}'_\pm)\varphi(\mathbf{p}'_\pm).
\end{equation}

Recall the norm $\|\cdot\|_{\vartheta}$ of $\F_\vartheta(\mathfrak{P}^\pm)$ defined in \eqref{eq:norm_vartheta} and replace its Lipschitz seminorm by the level-1 Lipschitz seminorm $|\cdot|_\vartheta^1$ of \cite[§5]{BHW23}:
\[
\|\varphi\|_\vartheta^1
\coloneqq
|\varphi|_\vartheta^1+\|\varphi\|_\infty.
\]
With this in mind, we prove the following key inequality, which is the main ingredient to describe the spectral theory for the weighted transfer operator.

\begin{prop}[Key inequality] \label{prop:key_ineq}
    Let $\mathfrak{G}$ be a finite $(q+1)$-regular graph, $\vartheta \in (0,1)$, and $\omega \in C^{\mathrm{lc}}(\mathfrak{P}^\pm)$ be a potential.
    Then, for every $\varphi \in \F_\vartheta(\mathfrak{P}^\pm)$, we have
    \begin{equation} \label{eq:key_ineq}
        \|\mathcal{L}_{\pm, \omega} \varphi \|^1_{\vartheta} \leq 
        q\vartheta\|\omega\|_\infty|\varphi|^1_\vartheta
        +
        q\bigl(
        \|\omega\|_\infty+\vartheta|\omega|^1_\vartheta
        \bigr)\|\varphi\|_\infty.
    \end{equation}
    Moreover, for every $n\geq1$,
    \begin{equation} \label{eq:key_ineq_m}
        \|\mathcal{L}_{\pm,\omega}^n\varphi\|_\vartheta^1
        \leq{}
        (q\vartheta\|\omega\|_\infty)^n
        |\varphi|_\vartheta^1
        +
        q^n\|\omega\|_\infty^{n-1}
        \left(
        \|\omega\|_\infty
        +
        \frac{\vartheta}{1-\vartheta}
        |\omega|_\vartheta^1
        \right)
        \|\varphi\|_\infty.
    \end{equation} 
    In particular, $\mathcal{L}^n_{\pm, \omega}$ is a bounded  operator on $(\F_\vartheta(\mathfrak{P}^\pm), \|\cdot\|_{\vartheta}^1)$ for every $n\geq 1$.
\end{prop}

\begin{proof}
    By \cite[Lem.~5.7 and Cor.~5.8]{BHW23}, the unweighted transfer
    operator satisfies
    $$
    \|\mathcal{L}_{\pm, 1} \psi\|^1_\vartheta 
    = |\mathcal{L}_{\pm, 1} \psi|^1_\vartheta + \|\mathcal{L}_{\pm, 1} \psi\|_\infty
    \leq q(\vartheta|\psi|^1_\vartheta + \|\psi\|_\infty),
    \quad  \text{ for every } \psi \in \F_\vartheta(\mathfrak{P}^\pm).
    $$
    Applying this inequality to $\psi = \omega \varphi \in \F_\vartheta(\mathfrak{P}^\pm)$ and using \eqref{eq:decomposition}, we get 
    $$\|\mathcal{L}_{\pm, \omega} \varphi\|^1_\vartheta 
    \leq q(\vartheta|\omega \varphi|^1_\vartheta + \|\omega \varphi\|_\infty).
    $$
    Since $\omega$ is locally constant, it belongs to $\F_\vartheta(\mathfrak{P}^\pm)$, hence $|\omega|^1_\vartheta < \infty$.
    Therefore, the supremum norm satisfies
    \begin{equation} \label{eq:estimate_supnorm}
        \|\omega \varphi\|_\infty \leq \|\omega \|_\infty \|\varphi \|_\infty
    \end{equation}
    and for the $\vartheta$-Hölder seminorm \eqref{eq:Hölder_seminorm},
    we have
    \begin{eqnarray*}
        |(\omega \varphi)(\mathbf{p}_\pm)-(\omega \varphi)(\mathbf{p}'_\pm)|
        &=&|\omega(\mathbf{p}_\pm)\varphi(\mathbf{p}_\pm)-\omega(\mathbf{p}'_\pm)\varphi(\mathbf{p}'_\pm)| \\
        &\leq& |\omega(\mathbf{p}_\pm)|\;|\varphi(\mathbf{p}_\pm)-\varphi(\mathbf{p}'_\pm)|+|\varphi(\mathbf{p}'_\pm)|\;|\omega(\mathbf{p}_\pm)-\omega(\mathbf{p}'_\pm)|.
    \end{eqnarray*}
    Dividing by $d_\vartheta(\mathbf{p},\mathbf{p}')$ and taking the supremum over $\mathbf{p}_\pm \neq \mathbf{p}'_\pm \in \mathfrak{P}^\pm$ gives
    \begin{equation} \label{eq:estimate_norm_vartheta}
        |\omega \varphi |^1_\vartheta \leq \|\omega\|_\infty |\varphi|^1_\vartheta + |\omega|^1_\vartheta \|\varphi\|_\infty.
    \end{equation}
    Combining these estimates \eqref{eq:estimate_supnorm} and \eqref{eq:estimate_norm_vartheta} proves the first assertion \eqref{eq:key_ineq}.

    For the second assertion \eqref{eq:key_ineq_m}, this follows from \eqref{eq:key_ineq} by induction on $n \geq 1$ with $\mathcal{L}^n_{\pm, \omega}\varphi = \mathcal{L}_{\pm, \omega}(\mathcal{L}^{n-1}_{\pm, \omega}\varphi)$  with convention $\mathcal{L}^0_{\pm, \omega}= \mathrm{Id}$.
    More precisely, for $n \geq 0$, set 
    \begin{equation} \label{eq:setting_Lm}
        \|\mathcal{L}^n_{\pm, \omega} \varphi \|^1_\vartheta = \|\mathcal{L}^n_{\pm, \omega} \varphi \|_\infty+ |\mathcal{L}^n_{\pm, \omega} \varphi |^1_\vartheta \coloneqq a_n + b_n,
    \end{equation}
    with convention $a_0=\|\varphi\|_\infty$ and $b_0=|\varphi|^1_\vartheta.$\\
    For the bound $a_n$, by applying 
    \eqref{eq:estimate_supnorm} $n$ times gives 
    \begin{equation}\label{eq:am_bound}
    a_n
    \leq (q\|\omega\|_\infty)^n \|\varphi\|_\infty.
    \end{equation}
    Regarding the bound $b_n$, apply \eqref{eq:key_ineq}, in particular
    $|\mathcal{L}_{\pm,\omega}\varphi|_\vartheta^1
    \leq
    q\vartheta\|\omega\|_\infty|\varphi|_\vartheta^1
    +
    q\vartheta|\omega|_\vartheta^1\|\varphi\|_\infty$
    to $\mathcal{L}^{n-1}_{\pm,\omega}\varphi$, this gives
    $$
    b_n 
    \leq q \vartheta \|\omega\|_\infty b_{n-1} + q \vartheta |\omega|^1_\vartheta a_{n-1} 
    \stackrel{\eqref{eq:am_bound}}{\leq}
    q \vartheta \|\omega\|_\infty b_{n-1} + q^n \|\omega\|_\infty^{n-1} (\vartheta |\omega|^1_\vartheta)^n \|\varphi\|_\infty.
    $$
    Iterating this recurrence, we find
    \begin{eqnarray} \label{eq:bm_bound}
        b_n
    &\leq{}&
    (q\vartheta\|\omega\|_\infty)^n|\varphi|_\vartheta^1
    +
    q\vartheta|\omega|_\vartheta^1
    \sum_{j=0}^{n-1}
    (q\vartheta\|\omega\|_\infty)^{n-1-j}
    (q\|\omega\|_\infty)^j
    \|\varphi\|_\infty \nonumber \\
    &=& 
    (q\vartheta\|\omega\|_\infty)^n|\varphi|_\vartheta^1
    +
    (q\|\omega\|_\infty)^{n-1} q\vartheta|\omega|_\vartheta^1
    \frac{1-\vartheta^n}{1-\vartheta}
    \|\varphi\|_\infty.
    \end{eqnarray}
    Since $\vartheta\in (0,1)$, the factor $\frac{1-\vartheta^n}{1-\vartheta}$ stays bounded and increases to $\frac{1}{1-\vartheta}$ as $n \rightarrow \infty.$
    Finally, adding \eqref{eq:am_bound} and \eqref{eq:bm_bound}
    in \eqref{eq:setting_Lm} proves \eqref{eq:key_ineq_m}.
\end{proof}

\begin{definition}[Projection operator, {\cite[§7]{BHW23}}] \label{def:projection_Pi_m}
    For every admissible finite edge path $\mathbf{p}_m=(\vec{e}_1,\ldots,\vec{e}_m)$ (also know in \cite[§7]{BHW23} as postal code) of length $m \geq 1$, choose a preferred infinite continuation 
    $$\mathbf{p}_m^\bullet
    = (\vec{e}_1,\ldots,\vec{e}_m,\vec{e}_{m+1}^{\,\bullet},\vec{e}_{m+2}^{\,\bullet}, \ldots) \in\mathfrak{P}^\pm.$$
    For $\mathbf{p}_\pm\in\mathfrak{P}^\pm$, let
    $[\mathbf{p}_\pm]_m$ denote its initial segment of length $m$.
    The \emph{projection operator}
    \begin{equation*}\label{eq:projection_Pi_m}
        \Pi_m: \F_\vartheta(\mathfrak{P}^\pm) \rightarrow F_m(\mathfrak{P}^\pm)
    \end{equation*}
    is defined by $(\Pi_m\varphi)(\mathbf{p}_\pm) \coloneqq \varphi\bigl([\mathbf{p}_\pm]_m^\bullet\bigr).$
\end{definition}
This means that $\Pi_m \varphi$ is constant on all level-$m$ districts and
depends only on the first $m$ edges, hence $\Pi_m \varphi \in F_m(\mathfrak{P}^\pm)$ and $\Pi_m\varphi=\varphi.$
Moreover,
$$\Pi^2_m=\Pi_m, \quad \mathrm{Ran}(\Pi_m)=F_m(\mathfrak{P}^\pm), \quad \Pi_m|_{F_m(\mathfrak{P}^\pm)}=\mathrm{Id}_{F_m(\mathfrak{P}^\pm)}.$$
Since $F_m(\mathfrak{P}^\pm)$ is finite-dimensional, $\Pi_m$ is a finite-rank projection.

\begin{lem}[Approximation by projection and its inequalities,  {\cite[Obs.~7.1.]{BHW23}}] \label{lem:Pim_ineq}
    For every $\varphi\in\F_\vartheta(\mathfrak{P}^\pm)$, we have
    $$\|\varphi-\Pi_m\varphi\|_\infty \leq \vartheta^m|\varphi|_\vartheta^1.$$
    On the other hand,
    $\|\Pi_m\varphi\|_\vartheta^1 \leq \|\varphi\|_\vartheta^1.$
\end{lem}

\begin{proof}
    For the first inequality, we take $\mathbf{p}_\pm, [\mathbf{p}_\pm]_m^\bullet \in \mathfrak{P}^\pm$ such that $(\Pi_m\varphi)(\mathbf{p}_\pm) = \varphi\bigl([\mathbf{p}_\pm]_m^\bullet\bigr)$. Both agree on their first $m$ edges. Thus $d_\vartheta\bigl(\mathbf{p}_\pm,[\mathbf{p}_\pm]_m^\bullet\bigr)\leq\vartheta^m,$
    and consequently
    $$
    \|(\varphi-\Pi_m\varphi)(\mathbf{p}_\pm)\|_\infty
    =
    \|\varphi(\mathbf{p}_\pm)-\varphi\bigl([\mathbf{p}_\pm]_m^\bullet\bigr)\|_\infty
    \leq d_\vartheta\bigl(\mathbf{p}_\pm,[\mathbf{p}_\pm]_m^\bullet\bigr) |\varphi|^1_\vartheta
    \leq
    \vartheta^m|\varphi|_\vartheta^1.
    $$
    
    For the second inequality, the supremum-norm estimate follows immediately from the definition:
    $$\|(\Pi_m\varphi)(\mathbf{p}_\pm)\|_\infty
    = \| \varphi\bigl([\mathbf{p}_\pm]_m^\bullet\bigr) \|_\infty 
    \leq \|\varphi\|_\infty.$$
    For compatible preferred continuations, passing from
    $\mathbf{p}_\pm$ to $[\mathbf{p}_\pm]_m^\bullet$ does not decrease the
    length of the common initial segment of two paths. Hence
    $
    |\Pi_m\varphi|_\vartheta^1
    \leq
    |\varphi|_\vartheta^1.
    $
\end{proof}

We derive the following inequalities, mirroring those of \cite[Cor.~7.2]{BHW23}. %, for the weighted transfer operator.

\begin{prop}[Finite-rank approximation inequalities] \label{prop:Km_ineq}
    For $\omega \in C^{\mathrm{lc}}(\mathfrak{P}^\pm)$ and $\varphi \in \F_\vartheta(\mathfrak{P}^\pm)$, we have for every  $m \geq 1$
    \begin{itemize}
        \item[(i)] 
        $\|\mathcal{L}^m_{\pm,\omega}\varphi - \mathcal{L}^m_{\pm,\omega}\Pi_m \varphi \|_\infty
        %\leq (q\vartheta  \|\omega\|_\infty)^m |\varphi|^1_\vartheta 
        \leq (q\vartheta  \|\omega\|_\infty)^m \|\varphi\|^1_\vartheta$,
        \item[(ii)] 
        $\|\mathcal{L}^m_{\pm,\omega}\varphi - \mathcal{L}^m_{\pm,\omega}\Pi_m \varphi \|^1_\vartheta
        \leq C_{\vartheta, \omega}(q\vartheta  \|\omega\|_\infty)^m \|\varphi\|^1_\vartheta,$\\%\left(3+\vartheta\|\omega\|_\infty^{-1}|\omega|^1_\vartheta\frac{1}{1-\vartheta}\right)$ for $\|\omega\|_\infty>0.$
        where $ C_{\vartheta, \omega}>0$ is a constant depending on $\vartheta\in (0,1)$ and $\omega$ (but not $m$).
    \end{itemize}
\end{prop}

\begin{proof}
    \begin{itemize}
        \item[$(i)$] Using Proposition~\ref{prop:L_bound} and Lemma~\ref{lem:Pim_ineq}, we obtain
        \begin{eqnarray*}
            \|\mathcal{L}_{\pm,\omega}\varphi - \mathcal{L}_{\pm,\omega}\Pi_m \varphi \|_\infty 
            = \|\mathcal{L}_{\pm,\omega}(\varphi - \Pi_m \varphi) \|_\infty
            &\leq& q\|\omega\|_\infty \|\varphi - \Pi_m \varphi\|_\infty \\
            &\leq& q\|\omega\|_\infty \vartheta^m |\varphi|^1_\vartheta \\
            &\leq& \vartheta^m q\|\omega\|_\infty \|\varphi\|^1_\vartheta.
        \end{eqnarray*}
        For the generalisation, it is the same argument as above, except that Proposition~\ref{prop:L_bound} is applied repeatedly.
        \item[$(ii)$] Assume that $\|\omega\|_\infty >0$.
        Using the key inequality \eqref{eq:key_ineq_m} in Proposition~\ref{prop:key_ineq}, and Lemma~\ref{lem:Pim_ineq}, together with $|\varphi - \Pi_m \varphi|^1_\vartheta \leq |\varphi|_\vartheta^1+|\Pi_m \varphi|^1_\vartheta \leq 2|\varphi|_\vartheta^1 \leq 2\|\varphi\|_\vartheta^1$, we get
        \begin{eqnarray*}
            \|\mathcal{L}_{\pm,\omega}\varphi - \mathcal{L}_{\pm,\omega}\Pi_m \varphi \|_\vartheta^1
            &=& \|\mathcal{L}_{\pm,\omega}(\varphi - \Pi_m \varphi) \|_\vartheta^1\\
            &\leq& (q\vartheta\|\omega\|_\infty)^m
            |\varphi - \Pi_m \varphi|_\vartheta^1 \\
            &&+
            q^m\|\omega\|_\infty^{m-1}
            \left(
            \|\omega\|_\infty
            +
            \frac{\vartheta}{1-\vartheta}
            |\omega|_\vartheta^1
            \right)
            \|\varphi - \Pi_m \varphi\|_\infty \\
            &\leq& 2(q\vartheta\|\omega\|_\infty)^m
            \|\varphi\|_\vartheta^1 \\
            &&+
            (q\|\omega\|_\infty\vartheta)^m\|\omega\|_\infty^{-1}
            \left(
            \|\omega\|_\infty
            +
            \frac{\vartheta}{1-\vartheta}
            |\omega|_\vartheta^1
            \right)
            \|\varphi\|_\vartheta^1 \\
            &= & C_{\vartheta, \omega} (q\|\omega\|_\infty\vartheta)^m \|\varphi\|_\vartheta^1,
        \end{eqnarray*}
        where, in the last line, we set
        $C_{\vartheta, \omega} \coloneqq \big(3
            +
            \|\omega\|_\infty^{-1}
            \frac{\vartheta}{1-\vartheta}
            |\omega|_\vartheta^1
            \big) >0.$ \qedhere
    \end{itemize}
\end{proof}

\begin{lem} \label{lem:Pim_commuting}
    Let $m \geq 1$ and suppose that $\omega\in F_{m+1}(\mathfrak{P}^\pm)$.
    Assume that the preferred continuations defining the projections are chosen compatibly with the transfer operator. 
    Then, for every $n\geq 1$, we have
    $$\Pi_m\mathcal{L}_{\pm,\omega}^{n}=\mathcal{L}_{\pm,\omega}^{n}\Pi_{n+m}.
    $$
\end{lem}

\begin{proof}
    By the compatibility of the projections with the unweighted transfer operator \cite[Obs.~9.8]{BHW23}, we have 
    $\Pi_m\mathcal{L}_{\pm}=\mathcal{L}_{\pm}\Pi_{m+1}.$ 
    Moreover, multiplication by $\omega$ commutes with $\Pi_{m+1}$.
    More precisely, for every $\varphi \in \F_\vartheta(\mathfrak{P}^\pm)$ and $\mathbf{p}_\pm\in\mathfrak{P}^\pm$
    \begin{align*}
    \Pi_{m+1}(\omega\varphi)(\mathbf{p}_\pm)
    =
    \omega\bigl([\mathbf{p}_\pm]_{m+1}^\bullet\bigr)
    \varphi\bigl([\mathbf{p}_\pm]_{m+1}^\bullet\bigr)
    =
    \omega(\mathbf{p}_\pm)
    (\Pi_{m+1}\varphi)(\mathbf{p}_\pm),
    \end{align*}
    because $\omega$ depends only on the first $m+1$ edges.
    Thus $\Pi_{m+1}M_\omega=M_\omega \Pi_{m+1}$, with $M_\omega \varphi \coloneqq \omega \varphi.$
    Therefore, using \eqref{eq:decomposition}, we obtain
    \begin{align*}
    \Pi_m\mathcal{L}_{\pm,\omega}\varphi
    =
    \Pi_m\mathcal{L}_{\pm,1}(\omega \varphi)
    =
    \mathcal{L}_{\pm,1}\Pi_{m+1}(\omega \varphi)
    =
    \mathcal{L}_{\pm,1}M_\omega\Pi_{m+1}\varphi
    =
    \mathcal{L}_{\pm,\omega}\Pi_{m+1}\varphi.
    \end{align*}
    Since this holds for every $\varphi$, it follows that $\Pi_m\mathcal{L}_{\pm,\omega}= \mathcal{L}_{\pm,\omega}\Pi_{m+1}.$
    Iterating this identity gives
    \begin{align*}
    \Pi_m\mathcal{L}_{\pm,\omega}^{n}
    =
    \mathcal{L}_{\pm,\omega}\Pi_{m+1}
       \mathcal{L}_{\pm,\omega}^{n-1}
    =
    \mathcal{L}_{\pm,\omega}^{2}\Pi_{m+2}
       \mathcal{L}_{\pm,\omega}^{n-2}
    =\cdots=
    \mathcal{L}_{\pm,\omega}^{n}\Pi_{n+m},
    \end{align*}
    as claimed.
\end{proof}

% -------------------------------%%--------------------------------------%
\subsection{Towards the meromorphic continuation of the weighted resolvent} \label{sect:discrete_spectrum}
With the Banach space framework and the key inequality in place,
we are now in a position to study the spectral behaviour of %the weighted transfer operator
$\mathcal{L}_{\pm,\omega}$ on $\F_{\vartheta}(\mathfrak{P}^\pm)$.

% -------------------------------%%--------------------------------------%

% -------------------------------%%--------------------------------------%
\subsubsection{Spectral and essential spectral radii bounds} \label{sect:bounds_spectral_essential_radii}

We begin by reviewing some standard notions from spectral analysis.

\begin{definition}[Basic notions from spectral theory] %[Fredholm operator and spectrum]
    Let $T$ be a bounded operator on a Banach space $X$.
    \begin{itemize}
        \item[(a)] $T$ is a \emph{Fredholm operator of index zero} if $$\mathrm{Im}(T) \subset X \text{ is closed and  } \mathrm{dim(ker(}T))=\mathrm{dim(coker(}T)),$$
        where $\operatorname{coker} T \coloneqq X/\operatorname{Im}(T).$
        \item[(b)] The \emph{spectrum} of $T$ is defined as
        $\sigma(T) \coloneqq \{\lambda \in \C \;|\; \lambda \mathrm{Id} - T \text{ is not invertible}\}$
        and its \emph{spectral radius} is 
        $r_{\mathrm{spec}}(T) \coloneqq \sup\{\lambda \in \sigma(T) \;|\; |\lambda|\}.$
        \item[(c)] The \emph{essential spectrum} of $T$ is defined as
        $$\sigma_{\mathrm{ess}}(T) \coloneqq \{\lambda \in \C \;|\; \lambda \mathrm{Id} - T \text{ is not Fredholm of index zero}\}$$
        and its \emph{essential spectral radius} is $r_{\mathrm{ess}}(T) \coloneqq \sup\{\lambda \in \sigma_{\mathrm{ess}}(T) \;|\; |\lambda|\}.$
    \end{itemize}
\end{definition}

Recall that by Gelfand's formula, the spectral radius of a bounded operator $T$ on a Banach space $X$ can be characterised as
\begin{equation} \label{eq:Gelfand}
    r_{\mathrm{spec}}(T)= \lim_{n\rightarrow \infty} \|T^n\|^{1/n}
\end{equation}
and by Nussbaum's formula, the essential
spectral radius as
\begin{equation} \label{eq:Nussbaum}
    r_{\mathrm{ess}}(T)=\lim_{n\to\infty}\inf_{K \mathrm{compact}}\|T^n-K\|^{1/n}.
\end{equation}

With these spectral-theoretic preliminaries in hand, let us establish bounds for both the spectral and essential spectral radii of the weighted transfer operator.

\begin{lem}[Spectral radius bound] \label{lem:spec_radius_bound}
    Let $\mathfrak{G}$ be a finite $(q+1)$-regular graph with $q>1$ and $0\neq\omega \in C^{\mathrm{lc}}(\mathfrak{P}^\pm)$. 
    Then, the spectral radius of $\mathcal{L}_{\pm,\omega}$ restricted to $F_1(\mathfrak{P}^\pm)$ satisfies
    \begin{equation*}
        r_{\mathrm{spec}}\bigl(\mathcal{L}_{\pm,\omega}|_{F_1}\bigr) \;\le\; q\|\omega\|_\infty.
    \end{equation*}
\end{lem}

\begin{proof}
    For $\varphi \in F_1(\mathfrak{P}^\pm)$ one has $|\varphi|^1_\vartheta=0$, so the key inequality
    \eqref{eq:key_ineq_m} reduces to
    $$
        \|\mathcal{L}_{\pm,\omega}^n\varphi\|^1_\vartheta
        \;\le\;
        (q\|\omega\|_\infty)^{n}
        \left(1+\|\omega\|_\infty^{-1}\frac{\vartheta}{1-\vartheta}|\omega|^1_\vartheta \right)\|\varphi\|_\infty
        %\leq C_{\vartheta, \omega} (q\|\omega\|_\infty)^{m} 
        \quad \text{ with } \|\omega\|_\infty >0
    $$
    and since $\vartheta\in(0,1)$ the bracketed factor stays bounded as $n\to\infty$. 
    Taking $m$-th roots and using the equivalence of norms on the finite-dimensional space
    $F_1(\mathfrak{P}^\pm)$, Gelfand's formula \eqref{eq:Gelfand} yields the claimed bound. 
\end{proof}

\begin{lem}[Essential spectral radius bound] \label{lem:ess_spec_radius_bound}
    Let $\mathfrak{G}$ be a finite $(q+1)$-regular graph with $q>1$ and $0 \neq \omega \in C^{\mathrm{lc}}(\mathfrak{P}^\pm)$. 
    The essential spectral radius of the  $\mathcal{L}_{\pm,\omega}$ satisfies
    \begin{equation*}
        r_{\mathrm{ess}}(\mathcal{L}_{\pm,\omega}) \leq q \vartheta \|\omega\|_\infty < q\|\omega\|_\infty, \quad \vartheta\in (0,1).
    \end{equation*}
\end{lem}

\begin{proof}
   Let $\Pi_m$ be the canonical level-$m$ projection as in Definition~\ref{def:projection_Pi_m} and set 
   $K_m \coloneqq \mathcal{L}_{\pm, \omega}^m \Pi_m.$
   Since $\mathfrak{G}$ is finite, $F_m(\mathfrak{P}^\pm)$ is finite-dimensional, so $K_m$ is a finite-rank operator on $\F_\vartheta(\mathfrak{P}^\pm)$, in particular compact for every $m \geq 1.$

   By Proposition~\ref{prop:Km_ineq}~(ii), for every $\varphi \in \F_\vartheta(\mathfrak{P}^\pm)$ and $m\geq 1$, we have
   $$\|\mathcal{L}_{\pm,\omega}^m \varphi - K_m\varphi\|_\vartheta^1 
   \leq C_{\vartheta, \omega} (q\vartheta \|\omega\|_\infty)^m \|\varphi\|^1_\vartheta, 
   $$
   where $C_{\vartheta, \omega}= (3+ \|\omega\|_\infty^{-1}|\omega|^1_\vartheta \frac{\vartheta}{1-\vartheta})>0$ is a finite constant independent of $m$.
   Consequently, letting $m \rightarrow \infty$ gives $C_{\vartheta, \omega}^{1/m} \rightarrow 1$ and thus by
   Nussbaum's formula \eqref{eq:Nussbaum} we get the claimed bound.
\end{proof}

% -------------------------------%%--------------------------------------%
\subsubsection{Weighted resolvent on finite graphs} 

Extending \cite[Prop.~8.1 and Thm.~8.3]{BHW23}, for any potential $\omega$, one expects analogous spectral localization 
results.

\begin{thm}[Meromorphic structure of the weighted resolvent]
\label{thm:mero_resolvent}
    Let $\mathfrak{G}$ be a finite $(q+1)$-regular graph with $q> 1$ and 
    fix the parameter $\vartheta \in (0,1)$. For $\omega \in C^{\mathrm{lc}}(\mathfrak{P}^\pm)$, consider $\mathcal{L}_{\pm,\omega}$ on $\F_{\vartheta}(\mathfrak{P}^\pm)$.

    \begin{itemize}
        \item If $\omega \neq 0$, then,
    the resolvent 
    \begin{equation} \label{eq:resolvent}
        \mathbf{R}_\omega(s)\coloneqq (\mathrm{Id}-q^{\frac{1}{2}+is}\mathcal{L}_{\pm,\omega})^{-1}\colon \F_{\vartheta}(\mathfrak{P}^\pm) \longrightarrow \F_{\vartheta}(\mathfrak{P}^\pm)
    \end{equation}
    admits a meromorphic continuation, as a family of bounded operators on
    $\F_\vartheta(\mathfrak{P}^\pm)$, to the half-plane
    \begin{equation} \label{eq:region}
    \mathscr{H}_{\vartheta, \omega} \coloneqq \left\{ s \in \C \;\Big|\; \mathrm{Im}(s) > \frac{3}{2} + \frac{\log(\vartheta \|\omega\|_\infty)}{\log q}\right\}
    \end{equation}
    and its poles have finite-rank principal parts.
    \item If $\omega=0$, then $\mathcal{L}_{\pm,\omega}=0$ and $\mathbf{R}_\omega(s)=\mathrm{Id}$ for every $s\in\C$.
    \end{itemize}
\end{thm}

Notice that decreasing $\vartheta\|\omega\|_\infty$ (or, more specifically, shrinking $\vartheta$ with $\|\omega\|_\infty$ held fixed) moves the boundary of $\mathscr{H}_{\vartheta, \omega}$ downward and therefore enlarges the domain of meromorphicity, see Figure~\ref{fig:mero_weighted} (right panel).

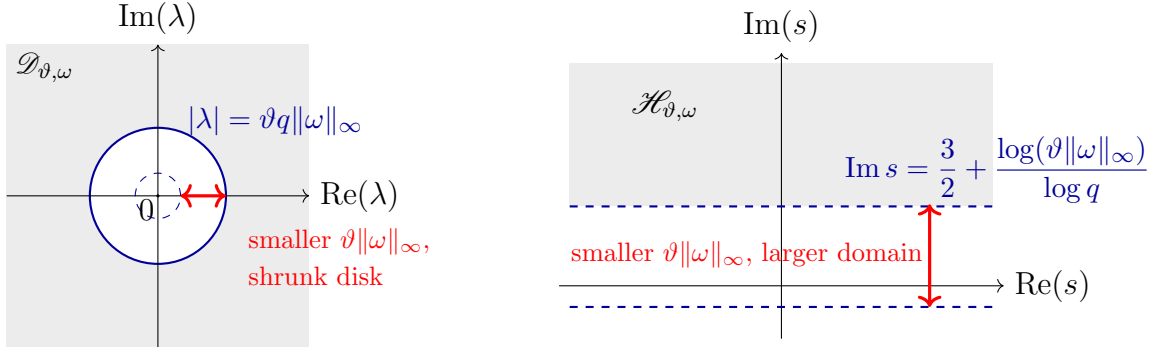
\begin{figure}[h!]
\centering
\begin{minipage}{0.46\textwidth}
\centering
\begin{tikzpicture}[scale=1.0]
    \def\Rout{2.0}
    \def\r{0.9}
    \def\rsmall{0.5}
    % shade the exterior region up to a bounding box, with a hole
    \begin{scope}
        \clip (-\Rout,-\Rout) rectangle (\Rout,\Rout);
        \fill[gray!15] (-\Rout,-\Rout) rectangle (\Rout,\Rout);
        \fill[white] (0,0) circle (\r);
    \end{scope}
    \draw[thick,blue!60!black] (0,0) circle (\r);
    %\draw[dashed,blue!60!black] (0,0) circle (1.3);
    \draw[dashed,blue!60!black] (0,0) circle (\rsmall-0.2);
    \draw[->] (-\Rout,0)--(\Rout,0) node[right] {$\Re(\lambda)$};
    \draw[->] (0,-\Rout)--(0,\Rout) node[above] {$\Im(\lambda)$};

    \fill (0,0) circle (0.6pt);
    \node[below left] at (0.1,0.1) {$0$};
    \node[blue!60!black] at (1.55,1.0)
    {\small $|\lambda|=\vartheta q\|\omega\|_\infty$};
    \node at (-1.5,1.7)
    {\small $\mathscr{D}_{\vartheta,\omega}$};
    \draw[<->,very thick, red] (0.9,0)--(0.3,0);
    \node[align=left, red] at (2.4,-0.8)
    {\footnotesize smaller $\vartheta\|\omega\|_\infty,$\\ \footnotesize shrunk disk};
\end{tikzpicture}
\end{minipage}
\hfill
\begin{minipage}{0.5\textwidth}
\centering
\begin{tikzpicture}[scale=0.7]
    \def\th{1.5}
    \fill[gray!15] (-4,\th) rectangle (4,4.2);
    \draw[->] (-4.2,0)--(4.2,0) node[right] {$\Re(s)$};
    \draw[->] (0,-1.0)--(0,4.4) node[above] {$\Im(s)$};
    \draw[dashed,blue!60!black,thick]
    (-4,\th)--(4,\th);
    \node[blue!60!black,right] at (1,\th+0.7)
    {\small $\Im s=\dfrac32+\dfrac{\log(\vartheta\|\omega\|_\infty)}{\log q}$};
    \node at (-2.2,3.4)
    {$\mathscr{H}_{\vartheta,\omega}$};
    \draw[<->,very thick, red]
    (2.8,\th+0.04)--(2.8,-0.4);
    \draw[dashed,blue!60!black,thick]
    (-4,-0.4)--(4,-0.4);
    \node[align=left, red] at (-0.65,0.6)
    {\footnotesize smaller $\vartheta\|\omega\|_\infty$, larger domain};
\end{tikzpicture}
\end{minipage}
    \caption{Meromorphicity domains for
    $\mathbf{R}_\omega(s)=(\mathrm{Id}-q^{\frac12+is}\mathcal{L}_{\pm,\omega})^{-1} \in \C$.
    Writing $\lambda=(q^{\frac12+is})^{-1}$, the resolvent is
    meromorphic in the exterior domain $\mathscr{D}_{\vartheta,\omega}=
    \{\lambda \in \C \;|\; |\lambda|>\vartheta q\|\omega\|_\infty\}$ (left panel), which
    corresponds to the half-plane $\mathscr{H}_{\vartheta,\omega}$ (right
    panel). Decreasing $\vartheta$ (or $\|\omega\|_\infty$) shrinks the
    excluded disk resp. lowers the threshold line, enlarging both domains of meromorphicity.
    When $\omega=0$, the resolvent is entire, i.e., the domain is all of $\C$.}
    \label{fig:mero_weighted}
\end{figure}

\begin{rem}[Translation to the variable $\lambda$] \label{rem:disk_domain}
    In the framework of \cite{BHW23}, set $\lambda \coloneqq \lambda(s)= (q^{\frac12+is})^{-1} \in \C$. The half-plane $\mathscr{H}_{\vartheta,\omega}$ \eqref{eq:region} corresponds under this change of variable to the exterior disk-domain (see Figure~\ref{fig:mero_weighted}, left panel)
    $$
    \mathscr{D}_{\vartheta,\omega} \coloneqq \bigl\{ \lambda \in \C \;\big|\; |\lambda| > \vartheta q \|\omega\|_{\infty} \bigr\},
    $$
    and $\mathbf{R}_\omega(s)$, viewed as a function of $\lambda$, is meromorphic on $\mathscr{D}_{\vartheta,\omega}$.
    In other words, $\lambda \in \mathscr{D}_{\vartheta,\omega}$ if and only if $s\in \mathscr{H}_{\vartheta,\omega}.$
\end{rem}

\begin{proof}[Proof of Theorem~\ref{thm:mero_resolvent}]
The case $\omega = 0$ is immediate. 

Assume $\omega\neq 0$, so $\|\omega\|_\infty>0$.
Since $\omega\in C^{\mathrm{lc}}(\mathfrak{P}^\pm)$, it belongs to $\F_\vartheta(\mathfrak{P}^\pm)$ for every $\vartheta\in(0,1)$.
By Proposition~\ref{prop:key_ineq}, $\mathcal{L}_{\pm,\omega}$ is a bounded operator on $\F_\vartheta(\mathfrak{P}^\pm)$ and
by Lemma~\ref{lem:ess_spec_radius_bound}, $r_{\mathrm{ess}}(\mathcal{L}_{\pm,\omega})\leq \rho_0\coloneqq q\vartheta\|\omega\|_\infty$ for fixed parameter $\vartheta\in (0,1).$

Since $\mathcal{L}_{\pm,\omega}$ is bounded, $\sigma(\mathcal{L}_{\pm,\omega})$ is a compact subset of $\C$.
So $$\sigma(\mathcal{L}_{\pm,\omega})\cap\{\lambda(s)=q^{-(\frac12+is)} \in \C :|\lambda(s)|>\rho_0\}$$ is not merely discrete but \emph{finite}, consisting of eigenvalues of finite algebraic multiplicity. 
Note that $q^{\frac12+is}$ is entire and nowhere vanishing.
The proof of \cite[Prop.~8.1]{BHW23} then shows that
$
\mathbf{R}_\omega(s) = \bigl(\operatorname{Id}-q^{\frac12+is}\mathcal{L}_{\pm,\omega}\bigr)^{-1}
$
extends meromorphically in $s$ to the disk $\{ s\in\C : |q^{\frac12+is}| < \rho_0^{-1}\}$, with finitely many poles, each of finite-rank principal part.

Using $|q^{\frac12+is}|=q^{\frac12-\operatorname{Im}(s)}$, the condition becomes
$$
q^{\frac12-\operatorname{Im}(s)} < \frac{1}{q\vartheta\|\omega\|_\infty}
\iff
\operatorname{Im}(s) > \frac12+\frac{\log(q\vartheta\|\omega\|_\infty)}{\log q}
= \frac32+\frac{\log(\vartheta\|\omega\|_\infty)}{\log q},
$$
i.e., precisely $s\in\mathscr H_{\vartheta,\omega}$. Hence $\mathbf{R}_\omega(s)$ extends meromorphically to $\mathscr H_{\vartheta,\omega}$ with finite-rank principal parts at its poles.
\end{proof}

\begin{rem}[Unweighted case and parametrization] \label{rem:parametrization}
\leavevmode
If $\omega \equiv 1$ is the constant multiplicative weight, then Theorem~\ref{thm:mero_resolvent} is precisely \cite[Prop.~8.1 and Thm.~8.3]{BHW23}. In particular, the resolvent
$\mathbf{R}(s) \coloneqq \mathbf{R}_{1}(s)$ extends meromorphically as a family of bounded operators on $\F_\vartheta(\mathfrak{P}^\pm)$ to 
$\mathscr{H}_{\vartheta, 1}$.
    \begin{itemize}
        \item[(i)] Note that our resolvent convention used here differs from that of \cite{BHW23}, where their resolvent is defined as 
        $\mathbf{R}^{\mathrm{BHW}}(\lambda) \coloneqq (\mathcal{L}_\pm - \lambda)^{-1}$
        and is meromorphic on $|\lambda|> \vartheta q.$
        Passing from their parameter $\lambda \in \C\backslash\{0\}$ to ours via the inversion $\lambda = z^{-1}$, the two resolvents are related by
        $$
        \mathbf{R}(z) 
        = (\mathrm{Id}-z\mathcal{L}_\pm)^{-1} 
        = -z^{-1}(\mathcal{L}_\pm-z^{-1}\mathrm{Id})^{-1}
        = -z^{-1}\mathbf{R}^{\mathrm{BHW}}(z^{-1}).
        $$
        At $z=0$, the family is defined directly by $\mathbf{R}(0)=\mathrm{Id}.$
        We then adopt our parametrization $z \coloneqq q^{\frac{1}{2}+is} \in \C$. In this setting, the meromorphic disk-domain $|z|>(\vartheta q)^{-1}$
        translates into the $s$-plane $\mathscr{H}_{\vartheta,1}$ \eqref{eq:region}, see Figure~\ref{fig:mero_weighted}.
        
        \item[(ii)] Let $\mathfrak{G}$ be a finite $(q+1)$-regular graph, i.e., $|\mathfrak{G}| < \infty$, and
        $$\ell^2(\mathfrak{E}) \coloneqq 
        \Big\{f \colon \mathfrak{E} \rightarrow \C \;\Big|\; \sum_{\vec{e} \in \mathfrak{E}} |f(\vec{e})|^2 < \infty \Big\} \subset \mathrm{Maps}(\mathfrak{E},\C).$$
        Since $\mathfrak{G}$ is finite, by \cite[Rem. 4.1(i)]{AFH23}
        $$F_1(\mathfrak{P}^\pm) \cong \mathrm{Maps}(\mathfrak{E},\C) \coloneqq \C^{\mathfrak{E}}=\ell^2(\mathfrak{E})$$ 
        is finite-dimensional, so all norms on it are equivalent, and
        the spectral radius $r_{\mathrm{spec}}(\mathcal{L}_\pm|_{F_1(\mathfrak{P}^\pm)}) = q$ is independent of the ambient norm, \cite[Thm.~8.3]{BHW23}.
        Hence for every
        $\lambda \in \C$ with $|\lambda| > q$, the unweighted resolvents
        $$
            \mathbf{R}_\vartheta(\lambda) \coloneqq (\mathcal{L}_\pm-\lambda)^{-1}\big|_{\F_\vartheta(\mathfrak{P}^\pm)}
            \qquad\text{and}\qquad
            \mathbf{R}_{\ell^2}(\lambda) \coloneqq (\mathcal{L}_\pm-\lambda)^{-1}\big|_{\ell^2(\mathfrak{E})}
        $$
        coincide as operators on $F_1(\mathfrak{P}^\pm)$: both are given there by the
        same absolutely convergent Neumann series
        $$
            (\mathcal{L}_\pm - \lambda)^{-1}
            = \sum^\infty_{n=0}\lambda^{-n}\mathcal{L}_\pm^n,
        $$
        which converges to the unique algebraic inverse of the finite matrix
        $(\mathcal{L}_\pm|_{F_1(\mathfrak{P}^\pm)} - \lambda\,\mathrm{Id})$ regardless of which equivalent norm on the finite-dimensional space $F_1(\mathfrak{P}^\pm)$ is used to test convergence.
        Thus $\mathbf{R}_\vartheta(\lambda)$ and $\mathbf{R}_{\ell^2}(\lambda)$ agree on $F_1(\mathfrak{P}^\pm)$ for all $|\lambda|>q$, by \cite[Prop.~7.3]{BHW23}.
        They differ only upon meromorphic continuation into $\vartheta q < |\lambda| \le q$, a region accessible
        to $\mathbf{R}_\vartheta$ but not to $\mathbf{R}_{\ell^2}$.
    \end{itemize}
\end{rem}

In view of Remark~\ref{rem:parametrization}~(ii), we can generalize it for any potential $\omega.$

\begin{cor}[Meromorphic structure of the weighted $\ell^2$-resolvent]
\label{cor:mero_l2_resolvent_weighted}
Let $\mathfrak{G}$ be a finite $(q+1)$-regular graph with $q>1$ and $\omega \in C^{\mathrm{lc}}(\mathfrak{P}^\pm)$. Then the resolvent
$$
    \mathbf{R}_{\ell^2,\omega}(s) \coloneqq \bigl(\mathrm{Id} - q^{\frac12+is}\,\mathcal{L}_{\pm,\omega}\bigr)^{-1}
    \colon \ell^2(\mathfrak{E}) \longrightarrow \ell^2(\mathfrak{E})
$$
extends meromorphically as a family of bounded operators to \emph{all} of $s\in\C$.
\begin{itemize}
    \item If $\omega\neq 0$, its poles have finite-rank principal parts and are located exactly at the eigenvalues of $\mathcal{L}_{\pm,\omega}|_{\ell^2(\mathfrak{E})}$, in particular in
$$
    \left\{\, s\in\C \;\middle|\; q^{\frac12+is} = \mu^{-1} \text{ for some } \mu\in\sigma\bigl(\mathcal{L}_{\pm,\omega}|_{\ell^2(\mathfrak{E})}\bigr) \,\right\}.
$$
    \item If $\omega=0$, then $\mathcal{L}_{\pm,\omega}=0$ and $\mathbf{R}_{\ell^2,\omega}(s)=\mathrm{Id}$ for every $s\in\C$, in particular it is holomorphic (pole-free) on all of $\C$.
\end{itemize}
\end{cor}
\begin{proof}
Since $\mathfrak{G}$ is finite, the edge set $\mathfrak{E}$ is finite, so
$$\ell^2(\mathfrak{E}) = F_1(\mathfrak{P}^\pm) \cong \mathrm{Maps}(\mathfrak{E}, \C)$$
is already finite-dimensional. 
Thus
$\mathbf{R}_{\ell^2,\omega}(s)$
is a rational function of the entire function $s\mapsto q^{\frac12+is}$ and hence meromorphic on all of $\C$, with poles exactly where $q^{\frac12+is}$ hits the reciprocal of an eigenvalue of $\mathcal{L}_{\pm,\omega}|_{\ell^2(\mathfrak{E})}$. 
They are finite-rank residues, since $\ell^2(\mathfrak{E})$ is itself finite-dimensional. 
If $\omega = 0$, then $\mathcal{L}_{\pm,\omega}|_{\ell^2(\mathfrak{E})}=0$.
So $\mathbf{R}_{\ell^2,\omega}(s)=\mathrm{Id}$ and thus there are no poles.
\end{proof}

\begin{rem}[Coincidence with $\mathbf{R}_\omega(s)$ for $\mathrm{Im}(s)>\tfrac32+\tfrac{\log\|\omega\|_\infty}{\log q}$]
\label{rem:resolvent_coincidence_weighted}
Unlike in the unweighted case \cite[Thm.~8.3]{BHW23} (see also Remark~\ref{rem:parametrization}), we do not know the exact spectral radius
$r_{\mathrm{spec}}(\mathcal{L}_{\pm,\omega}|_{F_1})$ for a general weight
$\omega \in C^{\mathrm{lc}}(\mathfrak{P}^\pm)$. We do, however, have the rigorous bound by Lemma~\ref{lem:spec_radius_bound}.
Note that equality holds, e.g. when $\omega$ is constant, by homogeneity of the spectrum, but not in general.

The condition $|\lambda| > q\|\omega\|_\infty$, i.e., $|z|^{-1}> q\|\omega\|_\infty$ for
$z=q^{\frac12+is}$, $s\in\C$ (see Remark~\ref{rem:parametrization}), translates via
$$
    |z| = q^{\frac12-\mathrm{Im}(s)} < (q\|\omega\|_\infty)^{-1}
    \quad\Longleftrightarrow\quad
    \mathrm{Im}(s) > \tfrac32+\frac{\log\|\omega\|_\infty}{\log q},
$$
which is exactly the boundary case $\vartheta=1$ of $\mathscr{H}_{\vartheta,\omega}$ \eqref{eq:region} in
Theorem~\ref{thm:mero_resolvent}. On this region, the bound above guarantees that the
Neumann series for $\mathbf{R}_\omega(s)$ converges absolutely with respect to every norm on
the finite-dimensional space $F_1(\mathfrak{P}^\pm)$, so
$$
    \mathbf{R}_\omega(s)\big|_{F_1(\mathfrak{P}^\pm)} = \mathbf{R}_{\ell^2,\omega}(s),
    \qquad \mathrm{Im}(s) > \tfrac32+\frac{\log\|\omega\|_\infty}{\log q}.
$$
Corollary~\ref{cor:mero_l2_resolvent_weighted} shows that the meromorphic continuation of
$\mathbf{R}_{\ell^2,\omega}(s)$ is a triviality once $\mathfrak{G}$ is finite: it is nothing but
the statement that a finite matrix pencil has a rational resolvent.

The genuine content of Theorem~\ref{thm:mero_resolvent} lies instead in extending the meromorphic domain of $\mathbf{R}_\omega(s)$ as an operator on the \emph{infinite-dimensional}
space $\F_\vartheta(\mathfrak{P}^\pm)$ (built from the path space $\mathfrak{P}^\pm$, not
merely from the finite edge set $\mathfrak{E}$) down to
$\mathrm{Im}(s) > \tfrac32+\tfrac{\log(\vartheta\|\omega\|_\infty)}{\log q}$, still with finite-rank principal parts. 
Whereas $\mathbf{R}_{\ell^2,\omega}(s)$, though meromorphic on all of $\C$ by Corollary~\ref{cor:mero_l2_resolvent_weighted}, no longer represents the same operator once one leaves $F_1(\mathfrak{P}^\pm)$: the identification above holds only on
$F_1(\mathfrak{P}^\pm)$ and only for $\mathrm{Im}(s)>\tfrac32+\tfrac{\log\|\omega\|_\infty}{\log q}$.
\end{rem}

By Theorem~\ref{thm:mero_resolvent} and Remark~\ref{rem:disk_domain}, the spectrum of $\mathcal{L}_{\pm,\omega}$ in $\mathscr{D}_{\vartheta,\omega}$ is \emph{discrete} and consists of isolated eigenvalues of finite algebraic multiplicity. 
In the unweighted case, this discrete spectrum is called the \emph{spectrum of Ruelle resonances} of $\mathcal{L}_\pm$ on $\F_\vartheta(\mathfrak{P}^\pm)$, see \cite[Rem.~8.2 and Thm.~8.3]{BHW23}. This motivates the following weighted analogue.

\begin{definition}[Weighted Ruelle resonances] \label{def:resonance}
Assume $\omega\neq 0$ and fix $\vartheta\in(0,1)$.
The \emph{$\omega$-weighted Ruelle resonances} of
$\mathcal{L}_{\pm,\omega}$ on
$\F_\vartheta(\mathfrak{P}^\pm)$ are the spectral values of
$
 \sigma(\mathcal{L}_{\pm,\omega})
 \cap\mathscr{D}_{\vartheta,\omega}.
$
They form a discrete set of isolated eigenvalues of finite algebraic
multiplicity.\\
Writing $\lambda(s_0)=q^{-(\frac12+is_0)}\in \C$, %(see Remark~\ref{rem:disk_domain}), 
a point $\lambda(s_0) \in \mathscr{D}_{\vartheta,\omega}$ is an $\omega$-weighted Ruelle resonance if and only if every $s_0\in\mathscr{H}_{\vartheta,\omega}$ %, with $\lambda(s_0)$, 
is a pole of the meromorphic continuation of $\mathbf{R}_\omega(s)$ from Theorem~\ref{thm:mero_resolvent}. 
Its multiplicity is denoted by $ m_{\pm,\omega}(s_0)$.
\end{definition}

Note that the order of the pole of $\mathbf{R}_\omega(s)$ at $s_0$ equals the size of the largest Jordan block of $\mathcal{L}_{\pm,\omega}$ associated with
the eigenvalue $\lambda(s_0)$.\\

For $\lambda\in\C$ and a linear subspace $U\subseteq\mathrm{Maps}(\mathfrak{P}^\pm,\C)$, let
$$ \mathcal{E}_\lambda(\mathcal{L}_{\pm, \omega};U) \coloneqq \left\{ \varphi\in U: \mathcal{L}_{\pm, \omega}\varphi=\lambda\varphi \right\} $$ 
denote the $\lambda$-eigenspace of the transfer operator $\mathcal{L}_{\pm, \omega}$.

\begin{rem}[Unweighted case: (co-)resonant states] \label{rem:unweighted_res_states}
    In the unweighted case $\omega\equiv1$, the %automatic 
    regularity result that, whenever $\lambda\in\sigma(\mathcal{L}_\pm)$ satisfies $|\lambda|>\vartheta q$, we have
    \begin{equation*} \label{eq:eigenspaces_isom}
        \mathcal{E}_\lambda\bigl(\mathcal{L}_\pm;\F_\vartheta(\mathfrak{P}^\pm)\bigr)
        =\mathcal{E}_\lambda(\mathcal{L}_{\pm}; {F}_1(\mathfrak{P}^{\pm})) 
        =
        \mathcal{E}_\lambda(\mathcal{L}_{\pm}; C^{\mathrm{lc}}(\mathfrak{P}^{\pm})) 
        \cong
        \mathcal{E}_\lambda(\mathcal{L}_{\pm}'; \mathcal{D}'(\mathfrak{P}^{\pm})),
    \end{equation*}
    where the second equality follows from \cite[Prop.~4.2]{AFH23} and the isomorphism from \cite[Cor.~9.4]{BHW23}.
    In particular, every eigenfunction associated with a Ruelle resonance outside the essential spectral radius depends only on the first edge. \\
    This makes it possible to naturally define resonances via $\mathcal{L}_{\pm}'$ acting on distributions, \cite[Def.~1.3]{AFH23Pairing}.\\
    Adaption our parametrization with $\lambda(s_0)=q^{-(\frac12+is_0)} \in \C$, the \emph{(co-)resonant states} are the eigendistributions of $\mathcal{L}_+'$ (resp. $\mathcal{L}_-'$), i.e., non-zero elements of the eigenspaces
    $$
    \mathcal{E}_{s_0}(\mathcal{L}_{\pm}'; \mathcal{D}'(\mathfrak{P}^{\pm})) 
    \coloneqq 
    \left\{ u \in \mathcal{D}'(\mathfrak{P}^{\pm}) \;\middle|\; 
    \mathcal{L}_{\pm}' u = \lambda(s_0) u \right\}.
    $$
    If $\mathcal{E}_{s_0}(\mathcal{L}_{\pm}'; \mathcal{D}'(\mathfrak{P}^{\pm})) \neq \{0\}$, then $\lambda(s_0)$ is %called 
    a {resonance}.

    Moreover, by \cite[Prop.~3.1]{AFH23Pairing} the (co-)resonant states of $\mathfrak{G}_\Gamma$ correspond to $\Gamma$-invariant (co-)resonant states of $\mathfrak{G}$ via the canonical projection
    \begin{equation} \label{eq:pi_shift_proj}
        \pi_{\pm} \colon \mathfrak{P}^\pm \to \mathfrak{P}^\pm_\Gamma.
    \end{equation}
    and we have
     $$
    \pi_{\pm}^{*} \colon \mathcal{E}_{s_0}(\mathcal{L}_{\Gamma, \pm}';\mathcal{D}'( \mathfrak{P}^{\pm}_\Gamma)) \; \stackrel{\sim}{\longrightarrow} \; \mathcal{E}_{s_0}(\mathcal{L}_{\pm}';\mathcal{D}'(\mathfrak{P}^{\pm}))^{\Gamma},
  $$
    where $\mathcal{L}'_{\Gamma , \pm}$ denotes transfer operator on $\mathfrak{P}^{\pm}_\Gamma.$
\end{rem}

Before describing the structure of our weighted resolvent, let us first discuss the structure of the \emph{usual} resolvent 
$\mathscr{R}(\lambda) \coloneqq (\lambda\operatorname{Id}-T)^{-1}$
of an operator $T$ on a %\textcolor{blue}{finite-dimensional} 
Banach space near an eigenvalue $\lambda_0$.
Then near $\lambda_0$, the %finite meromorphic 
resolvent $\mathscr{R}(\lambda)$ admits a finite-rank Laurent expansion
\begin{equation} \label{eq:Laurent_series_Rlambda}
        \mathscr{R}(\lambda)
        =
        \mathscr{R}_H(\lambda)
        +
        \sum_{j=1}^{J(\lambda_0)}
        \frac{A_j}{(\lambda-\lambda_0)^j},
\end{equation}
where $\mathscr{R}_H(\lambda)$ is holomorphic family of operators near $\lambda_0$, and
\begin{equation} \label{eq:Aj}
        A_j
        \coloneqq \left(T -\lambda_0\operatorname{Id} \right)^{j-1} \Pi_{\lambda_0}.
\end{equation}
are finite-rank operators with $\Pi_{\lambda_0}$ the corresponding Riesz spectral projection onto the $\lambda_0$-generalized eigenspaces.
The number $J(\lambda_0)$ is the size of the largest Jordan block of $T$ at $\lambda_0$, equivalently the 
the order of the pole of $\mathscr{R}_\omega(\lambda)$ at $\lambda=\lambda_0$.\\
Moreover, the restriction of
$(T-\lambda_0\operatorname{Id})$ to the generalized
eigenspace $\operatorname{Ran}(\Pi_{\lambda_0})$ is nilpotent of order $J(\lambda_0)$. In particular,
since $\Pi_{\lambda_0}$ commutes with $T$ and thus with $(T-\lambda_0\operatorname{Id})$, we have
$$(T-\lambda_0\operatorname{Id})^{J(\lambda_0)}\Pi_{\lambda_0}
=0
= ((T-\lambda_0\operatorname{Id})\Pi_{\lambda_0})^{J(\lambda_0)}.$$

Coming back to our weighted resolvent  $\mathbf{R}_\omega(s)$ on the Banach space $\F_\vartheta(\mathfrak{P}^\pm)$.
By Theorem~\ref{thm:mero_resolvent}, $\mathbf{R}_\omega(s)$
admits a meromorphic continuation to $\mathscr{H}_{\vartheta,\omega}$ \eqref{eq:region}.

By setting $\lambda(s) = q^{-(\frac12+is)}$ for $s\in \C$, the resolvent $\mathbf{R}_\omega(s)$ can be rewritten as
\begin{eqnarray*}
    \mathbf{R}_\omega(s) 
    = (\operatorname{Id}-q^{\frac12+is}\mathcal{L}_{\pm,\omega})^{-1}
    %&=&-\lambda(s)(\mathcal{L}_{\pm,\omega}-\lambda(s)\operatorname{Id})^{-1} \\
    =\lambda(s)(\lambda(s)\operatorname{Id}-\mathcal{L}_{\pm,\omega})^{-1}
    \eqqcolon \lambda(s) \mathscr{R}_\omega(\lambda(s)).
\end{eqnarray*}
This extra factor $\lambda(s)$ changes the Laurent coefficients compared with the usual resolvent $\mathscr{R}_\omega(\lambda(s))$.
Moreover, since
\begin{equation} \label{eq:lambda_prime}
    \lambda'(s)=-i\log(q)\,\lambda(s), \qquad \lambda(s)\neq 0 \ \text{ for all } s\in\C,
\end{equation}
the map $s\mapsto\lambda(s)$ is everywhere a local biholomorphism near $s_0$.
We use this to transfer the Laurent expansion of $\mathscr{R}_\omega(\lambda(s))$ in the variable $\lambda(s)$ into one in the variable $s$ directly.

\begin{lem}[Laurent expansion of the weighted resolvent] \label{lem:resolvent_decomp_hol_finiteranks}
    Fix $\vartheta\in(0,1)$.
    Let $\mathbf{R}_\omega(s)$ be the weighted resolvent on
    $\F_\vartheta(\mathfrak{P}^\pm)$.   
    For each pole $s_0 \in \mathscr{H}_{\vartheta,\omega}$, let $\lambda(s_0)$ be an isolated 
    $\omega$-weighted Ruelle resonance of
    $\mathcal{L}_{\pm,\omega}$ of finite algebraic multiplicity, and let $J(s_0)$ be the size of the largest Jordan block of
    $\mathcal{L}_{\pm,\omega}$ at $\lambda(s_0)$.
    Denote by $\Pi_{s_0}$ the corresponding Riesz spectral projector.

    Then,
    near $s_0$, the resolvent $\mathbf{R}_\omega(s)$ admits a finite-rank Laurent expansion of the form
    $$
        \mathbf{R}_\omega(s)
        =
        \mathbf{R}_H(s)
        +
        \sum_{j=1}^{J(s_0)} \frac{B_j}{(\lambda(s)-\lambda(s_0))^j},
    $$
    for some family $\mathbf{R}_H(s)$ holomorphic in $s$ near $s_0$ and finite-rank operators
    $$B_j= \mathcal{L}_{\pm,\omega}
    \left(\mathcal{L}_{\pm,\omega}-\lambda(s_0)\operatorname{Id}\right)^{j-1}\Pi_{s_0}.$$
    In particular, $s_0$ is a pole of $\mathbf{R}_\omega(s)$ of the same order $J(s_0)$ as the pole of $\mathscr{R}_\omega(\lambda(s))$ at $\lambda(s_0)$.
\end{lem}

\begin{proof}
    Using the Laurent expansion \eqref{eq:Laurent_series_Rlambda} of $\mathscr{R}_\omega(\lambda(s))$ near $\lambda=\lambda(s_0)$, we get
    $$
    \mathbf{R}_\omega(s)
            =
            \lambda(s)\mathscr{R}_H(\lambda(s))
            +
            \lambda(s)\sum_{j=1}^{J(s_0)}
            \frac{A_j}{(\lambda(s)-\lambda(s_0))^j}.
    $$
    Since $\lambda(s)=\lambda(s_0)+(\lambda(s)-\lambda(s_0))$, the principal part can be rewritten as
    $$
    \lambda(s)\sum_{j=1}^{J(s_0)}\frac{A_j}{(\lambda(s)-\lambda(s_0))^j}
    =\sum_{j=1}^{J(s_0)}\frac{B_j}{(\lambda(s)-\lambda(s_0))^j},
    $$
    where $B_j=\lambda(s_0)A_j+A_{j+1} \stackrel{\eqref{eq:Aj}}{=}\mathcal{L}_{\pm,\omega}
    \left(\mathcal{L}_{\pm,\omega}-\lambda(s_0)\operatorname{Id}\right)^{j-1}\Pi_{s_0}$ with convention $A_{J(s_0)+1}=~0.$
    
    Setting $\mathbf{R}_H(s)\coloneqq\lambda(s)\mathscr{R}_H(\lambda(s))$, it remains to check that $\mathbf{R}_H(s)$ is holomorphic in $s$ near $s_0$.
    Since $\mathscr{R}_H$ is holomorphic near $\lambda(s_0)$ and, by \eqref{eq:lambda_prime}, $s\mapsto\lambda(s)$ is a local biholomorphism near $s_0$, the composition $\lambda(s)\mapsto\mathscr{R}_H(\lambda(s))$, and hence $\mathbf{R}_H(s)$, is holomorphic in $s$ near $s_0$.

    Finally, as $\lambda(s)-\lambda(s_0)$ has a simple zero at $s=s_0$ by \eqref{eq:lambda_prime}, each summand $\frac{B_j}{(\lambda(s)-\lambda(s_0))^j}$ has a pole of exact order $j$ at $s=s_0$ whenever $B_j\neq0$, so the principal part above has a pole of order exactly $J(s_0)$ at $s=s_0$.
\end{proof}

\begin{rem} \label{rem:Riesz}
    The Riesz spectral projector associated with
    $\lambda(s_0)$ is given by
    \[
        \Pi_{s_0}
        =\frac{1}{2\pi i}
        \int_{\mathscr{C}} \mathscr{R}(\lambda(s))\,\mathrm{d}\lambda,
    \]
    where $\mathscr{C}$ is a sufficiently small positively oriented circle around
    $\lambda(s_0)$ containing no other spectral point of
    $\mathcal{L}_{\pm,\omega}$. 
    Its range is given by
    $$
    \operatorname{Ran}(\Pi_{s_0})
    = \ker\!\left(\mathcal{L}_{\pm,\omega}-\lambda(s_0)\right)^{J(s_0)}
    \subset \F_\vartheta(\mathfrak{P}^\pm),
    $$
    which is precisely the finite-dimensional generalised eigenspace associated with the eigenvalue $\lambda(s_0)$.
    
    Moreover, assume that $\omega\equiv1$ and $J(s_0)=1$, i.e., the resonance is simple. Then by the regularity of resonant states in Remark~\ref{rem:unweighted_res_states} and using \cite[Cor.~9.4]{BHW23}, one has
    \begin{equation} \label{eq:rangPis0_Es0}
        \operatorname{Ran}(\Pi_{s_0})
    = \mathcal{E}_{s_0}
    (\mathcal{L}_\pm;C^{\mathrm{lc}}(\mathfrak{P}^\pm))
    \cong
    \mathcal{E}_{s_0}
    (\mathcal{L}_\pm';\mathcal{D}'(\mathfrak{P}^\pm)).
    \end{equation}
    Thus a simple resonance correspond to a genuine eigenvalue of $\mathcal{L}_\pm$ (respectively $\mathcal{L}'_\pm$) %with eigenvalue $\lambda(s_0)$ 
    and the corresponding (co-)resonant states are genuine eigenfunctions (respectively eigendistributions).
\end{rem}

The Riesz spectral projector can be represented as a finite sum of rank-one operators as follows.

\begin{lem}[Finite-rank tensor decomposition of a Riesz projector]
\label{lem:Pi_tensor_decomp}
    %Fix $m$ large enough and
    Fix $\vartheta\in (0,1)$.
    Let $s_0\in \mathscr{H}_{\vartheta,\omega}$ be a resonance of order $J(s_0)=1$.
    Assume that $\Pi_{s_0}$ is a genuine spectral projector
    of finite rank $r\coloneqq\dim(\operatorname{Ran}(\Pi_{s_0}))<\infty$.
    
    Then there exist a basis $(u_\ell)_{\ell=1}^{r}$ of $\operatorname{Ran}(\Pi_{s_0})\subset \F_\vartheta(\mathfrak{P}^\pm)$ and a canonical dual family $(u_\ell^*)_{\ell=1}^{r} \subset \F'_\vartheta(\mathfrak{P}^\pm)$ uniquely characterized by $\langle u_\ell^*,u_n\rangle=\delta_{\ell n}$ and $u_\ell^*|_{\ker(\Pi_{s_0})}=0$ such that
    $$
    \Pi_{s_0} = \sum_{\ell=1}^{r} u_\ell \otimes u_\ell^*.
    $$
\end{lem}

\begin{proof}
    Since $\Pi_{s_0}$ is a bounded idempotent (i.e., $\Pi_{s_0}^2=\Pi_{s_0}$) of finite rank $r$, the Banach space $\F_\vartheta(\mathfrak{P}^\pm)$ decomposes as a topological direct sum 
    $$\F_\vartheta(\mathfrak{P}^\pm) = \operatorname{Ran}(\Pi_{s_0}) \oplus \ker(\Pi_{s_0}).$$ 
    Choose any basis $(u_\ell)_{\ell=1}^r$ of the $r$-dimensional vector space $\operatorname{Ran}(\Pi_{s_0})$. 
    Since every $\Pi_{s_0}(\varphi)$ for
     $\varphi\in \F_\vartheta(\mathfrak{P}^\pm)$ belongs to $\operatorname{Ran}(\Pi_{s_0})$, there are unique coefficients $c_\ell(\varphi)\in\C$, %, i.e., coordinates with respect to the basis $(u_\ell)$, 
     such that
     \begin{equation} \label{eq:Pis_phi}
         \Pi_{s_0}(\varphi) = \sum_{\ell=1}^r c_\ell(\varphi)\, u_\ell.
     \end{equation}
    Each $c_\ell \colon \F_\vartheta(\mathfrak{P}^\pm) \to \C$ is a bounded linear functional on $\F_\vartheta(\mathfrak{P}^\pm)$.
    Indeed, let $e_\ell^*: \operatorname{Ran}(\Pi_{s_0}) \to \C$ be the corresponding coordinate functionals. 
    By construction $c_\ell = e_\ell^* \circ \Pi_{s_0}.$
    Since $\operatorname{Ran}(\Pi_{s_0})$ is a finite-dimensional space and $\Pi_{s_0}$ is bounded, every linear functional is continuous, hence bounded.
    Thus $e_\ell^* \in \operatorname{Ran}(\Pi_{s_0})^*$ and their composition is bounded, i.e., $c_\ell \in \F'_\vartheta(\mathfrak{P}^\pm)$.
    
    Setting $u_\ell^* \coloneqq c_\ell \in \F'_\vartheta(\mathfrak{P}^\pm)$ in \eqref{eq:Pis_phi}, this gives
    $$\Pi_{s_0}(\varphi)
    =\sum_{\ell=1}^ru_\ell^*(\varphi)u_\ell
    =\sum_{\ell=1}^r \langle u^*_\ell, \varphi \rangle u_\ell
    =\sum_{\ell=1}^r (u_\ell\otimes u_\ell^*) (\varphi).$$
    So the finite-rank projector is expressed as a sum of $r$ rank-one operators.
    
    It remains to show that these functionals are biorthogonal to the basis. \\
    Idempotency $\Pi_{s_0}^2=\Pi_{s_0}$ applied to $u_n \in \operatorname{Ran}(\Pi_{s_0})$, gives $\Pi_{s_0}u_n=u_n.$\\
    On the other hand, using the above decomposition, we get
    $\Pi_{s_0}u_n= \sum_{\ell=1}^r \langle u^*_\ell, u_n \rangle u_\ell.$
    Thus, since $\Pi_{s_0}u_n=u_n$ and $(u_\ell)^r_{\ell=1}$ is a basis, we have 
    $$u_n=\sum_{\ell=1}^r \langle u^*_\ell, u_n \rangle u_\ell = \sum_{\ell=1}^r \delta_{\ell n} u_\ell.$$
    Hence using uniqueness of coordinates in the basis $(u_\ell)^r_{\ell=1}$ gives $\langle u_\ell^*,u_n\rangle=\delta_{\ell n}$, i.e., $u_\ell^*$ are biorthogonal to $u_n$.
    
    Finally, the dual basis
    $(e_\ell^*)^r_{\ell=1}$ is unique in $\operatorname{Ran}(\Pi_{s_0})^*$ and consequently, the functionals $u^*_\ell=e_\ell^* \circ \Pi_{s_0}$
    are the unique extensions of this dual basis to $\F'_\vartheta(\mathfrak{P}^\pm)$ that vanish on $\ker(\Pi_{s_0}),$ i.e., $u^*_\ell|_{\ker(\Pi_{s_0})}=0.$
\end{proof}

Moreover, for a simple resonance, by combining Lemma~\ref{lem:Pi_tensor_decomp} and Remark~\ref{rem:Riesz}, yields that
$
\Pi_{s_0} $
is the rank-one projection onto the eigenspace of $\mathcal{L}_\pm$ associated with $\lambda(s_0)$.\\

Note that since $\mathcal{L}_{\pm,\omega}$ commutes with the $\Gamma$-action on $\mathfrak{P}^\pm$, it induces a well-defined bounded operator $\mathcal{L}_{\Gamma, \pm, \omega}$ on the quotient space $\F_\vartheta(\mathfrak{P}^\pm_\Gamma)$ via the canonical projection $\pi_{\pm}$ given in \eqref{eq:pi_shift_proj}.
The same holds for the resolvent family $\mathbf{R}_\omega(s)$, which therefore descends to a holomorphic (resp. meromorphic) family of operators on $\F_\vartheta(\mathfrak{P}^\pm_\Gamma)$. In particular, all results on meromorphic continuation and properties,
obtained above remain valid in the finite graph setting.

% -------------------------------%%--------------------------------------%
\section{Weighted dynamical Ihara zeta functions} \label{sect:dyn_zeta}
In this section, we define the weighted dynamical Ihara zeta functions on finite $(q+1)$-regular graph $\mathfrak{G}_\Gamma$, establish their meromorphic continuation, and show that their zeros correspond to the resonance parameters in Definition~\ref{def:resonance}. 
Moreover in Subsection~\ref{sect:proof_Thm_residuePS}, we prove that the residues of these weighted dynamical zeta functions recover the phase-space distributions.

% -------------------------------%--------------------------------------%
\subsection{Periodic and primitive closed paths} \label{sect:periodic_paths}
Recall from Definition~\ref{def:projection_Pi_m} that the projection operator $\Pi_m$ reduces functions on $\mathfrak{P}^\pm$ to functions on admissible finite edge paths $[\mathbf{p}]^\bullet_m \in \mathfrak{P}^\pm$ of length $m$. 
Passing now to the quotient setting $\mathfrak{P}^\pm_\Gamma$, we identify these admissible finite edge paths with the finite paths $\mathcal{P}_m$ on the finite graph $\mathfrak{G}_\Gamma$ defined below. On these, we introduce the notions of closed, periodic, and primitive paths, together with the equivalence relation identifying closed paths that differ only in their choice of starting edge; see, e.g., \cite{StarkTerras1996, Terras10}.

\begin{definition}[Finite closed, periodic, and primitive paths] \label{def:finite_paths}
    Let $\mathfrak{G}_\Gamma=(\mathfrak{X}_\Gamma, \mathfrak{E}_\Gamma)$ be a finite $(q+1)$-regular graph with $q>1.$
    
    For $n\geq 1$, let $\mathcal{P}_m \subset\mathfrak{G}_\Gamma$ denote the set of \emph{finite paths} of length $m$. 
    Its elements are one-sided finite sequences of directed edges
    $\mathbf{p}_m=(\vec{e}_1, \vec{e}_2, \dots, \vec{e}_m) \in \mathfrak{E}_\Gamma^{(m)},$
    such that consecutive edges are concatenated and non-backtracking (also called \emph{reduced} in the literature; see, e.g., \cite{Terras10}).
    \begin{itemize}
        \item[(a)] We say that a finite path is \emph{closed} if it returns to its starting vertex, i.e., $\tau(\vec{e}_m)=\iota(\vec{e}_{1})$ and if it is in addition non-backtracking at the seam, i.e., $\tau(\vec{e}_1) \neq \iota(\vec{e}_m)$ for $m \geq 3.$
        We denote its set by 
        $$
        \mathcal{P}_m^{\mathrm{cl}} \coloneqq \{\mathbf{p}_m \in \mathcal{P}_m \;|\; \mathbf{p}_m \text{ is closed}\} \subset \mathcal{P}_m.
        $$
        \item[(b)] Two closed finite paths are \emph{equivalent} if they differ only by the choice of starting edge.
        \item[(c)] A closed finite path is called \emph{primitive} if it is not a repetition (power) of a shorter closed finite path.
        \item[(d)] A path is \emph{periodic} if it is obtained by infinitely repeating a finite closed path.
    \end{itemize}
\end{definition}

Note that by the non-backtracking condition, closed finite paths automatically have no tails. 
In the literature, a closed finite path without a tail is also referred to as a \emph{regular path}.
Moreover, in the setting of simple $(q+1)$-regular graphs with no loops and no multiple edges, there is \emph{no} non-backtracking closed finite path of length $m=1$ or $2$.

\begin{rem} \label{rem:basis_delta_p}
    By \cite[Rem.~4.1~(i)]{AFH23}, we know that
    $F_1(\mathfrak{P}^\pm_\Gamma) \cong \mathrm{Maps}(\mathfrak{E}_\Gamma, \C)$. This canonical linear isomorphism can be generalised for every $m\geq 1$:
    $$F_m(\mathfrak{P}^\pm_\Gamma) \cong \mathrm{Maps}(\mathcal{P}_m, \C)= \C^{\mathcal{P}_m}.$$  
    In particular, $\mathrm{dim}(F_m(\mathfrak{P}^\pm_\Gamma))=|\mathcal{P}_m|< \infty$ equals the number of districts of level $m$.
    Hence, $F_m(\mathfrak{P}^\pm_\Gamma)$ is identified with the space of functions on the finite set $\mathcal{P}_m$, with basis $\delta_{\mathbf{p}_m}$, for $ \mathbf{p}_m\in \mathcal{P}_m$, 
    given by the district indicator functions.
\end{rem}

\begin{lem} \label{lem:nb_Pm}
    The number of non-backtracking paths of length $m$ in a on a $(q+1)$-regular graph $\mathfrak{G}_\Gamma$ is
    \begin{equation} \label{eq:NT}
        |\mathcal{P}_m| = |\mathfrak{X}_\Gamma| (q+1)q^{m-1}, \quad \text{ for } m \geq 1.
    \end{equation}
\end{lem}

\begin{proof}
    Fix a starting vertex. For a non-backtracking path of length $m=1$, one simply chooses one of its neighbours, so there are $q+1$ possibilities.
    
    For $m\geq 2$, there are $q+1$ choices for the first step, and at each subsequent step there are $q$ choices (since immediate backtracking is forbidden). Hence the number of such paths of length $m$ starting from a fixed vertex is $(q+1)q^{m-1}$. 
    
    Since we can start anywhere in $\mathfrak{X}_\Gamma$, multiplying by the number of starting vertices $|\mathfrak{X}_\Gamma|$
    gives the total number of non-backtracking paths of length $m$.
\end{proof}

With this notation, we define (primitive) cycles in $\mathfrak{P}^\pm_\Gamma$.

\begin{definition}[(Primitive) cycles] \label{def:cycles}
    For $n \geq 3$, let $\mathcal{P}^{\mathrm{cl}}_n$ be the set of closed non-backtracking finite paths.
    \begin{itemize}
        \item[(i)] We define the set of all
    \emph{cycles} in $\mathfrak{P}^\pm_\Gamma$ by
    \[
        \mathscr{P} \coloneqq \bigcup_{n\geq 3} \mathcal{P}^{\mathrm{cl}}_n/\!\sim,
    \]
    i.e., the equivalence classes of closed non-backtracking finite paths under cyclic shifts. 
    For $\mathbf{p},\mathbf{q}\in\mathcal{P}^{\mathrm{cl}}_n$, the equivalence relation $\sim$ induced by cyclic shifts is given by
    \[
    \mathbf{p}\sim\mathbf{q} \quad\Longleftrightarrow\quad
    \mathbf{q}=\sigma_c^k(\mathbf{p})
    \quad\text{for some }k\in\{0,\ldots,n-1\},
    \]
    where
    $
    \sigma_c:\mathcal{P}^{\mathrm{cl}}_n
    \longrightarrow
    \mathcal{P}^{\mathrm{cl}}_n,$
    $\sigma_c(\vec{e}_1,\ldots,\vec{e}_n)
    =
    (\vec{e}_2,\ldots,\vec{e}_n,\vec{e}_1)
    $
    denotes the one-step cyclic shift, and for  $0\leq k \leq n-1$, $\sigma_c^k$ gives the $k$-fold cyclic shift with convention $\sigma_c^0=\mathrm{Id}.$\\
    Thus, for $\mathbf{p}=(\vec{e}_1, \vec{e}_2, \dots, \vec{e}_n) \in \mathcal{P}^{\mathrm{cl}}_n,$ its corresponding cycle is the equivalence class
    \begin{eqnarray*}
        [\mathbf{p}] 
        &\coloneqq& 
        \{(\vec{e}_1, \dots, \vec{e}_n), (\vec{e}_2, \dots, \vec{e}_n, \vec{e}_1), \dots, (\vec{e}_n, \vec{e}_1, \dots, \vec{e}_{n-1})\} \\
        %=\{\mathbf{p},\sigma_c(\mathbf{p}), \dots, \sigma_c^{n-1}(\mathbf{p})\},
        &=&\{\sigma_c^k(\mathbf{p}) : 0\leq k\leq n-1\} \in \mathscr{P}.
    \end{eqnarray*}
    \item[(ii)] We denote by $\mathscr{P}_0\subset\mathscr{P}$ the set of all \emph{primitive cycles}, i.e., cycles represented by primitive closed non-backtracking finite paths.
    \item[(iii)] Every $[\mathbf{p}] \in \mathscr{P}$ of length $n=\ell(\mathbf{p})$ is a power 
    \begin{equation*} \label{eq:power_primitive}
        [\mathbf{p}]=[\mathbf{p}_0]^{n/\ell_0} \qquad n/\ell_0 \geq 2
    \end{equation*}
    of a unique primitive cycle $[\mathbf{p}_0] \in \mathscr{P}_0$ of length $\ell_0 \coloneqq \ell(\mathbf{p}_0)$ dividing $n.$
    \end{itemize}
\end{definition}

Note that $\mathscr{P}$ contains no elements of length 1 or 2 in the simple-graph setting.

\begin{ex}[Simple graph, $q=2$] \label{exmp:finite_paths}
    Consider the simple $3$-regular graph $\mathfrak{G}_\Gamma = K_4$ on four vertices $\mathfrak{X}_\Gamma = \{x_1,x_2,x_3,x_4\}$, with a directed edge $\vec{e}_{ij}=(x_i,x_j)$ for every pair $i\neq j$, see Figure~\ref{fig:quotient_space}.
    Its universal cover $\mathfrak{G}$ is the $3$-regular tree. Then, for 
    \begin{itemize}
      \item $n=3$: the path $\mathbf{p}_3 = (\vec{e}_{12}, \vec{e}_{23}, \vec{e}_{31})$ is non-backtracking 
      and closed, since $\tau(\vec{e}_{31}) = x_1 = \iota(\vec{e}_{13})$. 
      As no shorter closed non-backtracking path exists in a simple graph, $\mathbf{p}_3$ is automatically primitive: $[\mathbf{p}_3]=\{\sigma_c^k(\mathbf{p}_3) \colon 0 \leq k \leq 2\} \in \mathscr{P}_0$.
      
      \item $n=4$: the path $\mathbf{p}_4 = (\vec{e}_{12}, \vec{e}_{23}, \vec{e}_{34}, \vec{e}_{41})$ is again non-backtracking and closed. Since no closed non-backtracking path of length $1$ or $2$ exists, $\mathbf{p}_4$ cannot be a proper power of a shorter closed path, so $[\mathbf{p}_4]=\{\sigma_c^k(\mathbf{p}_4) \colon 0 \leq k \leq 3\} \in \mathscr{P}_0$ as well.
   
      \item  $n=6$: traversing the triangle $\mathbf{p}_3$ twice gives
      \[
      \mathbf{p}_6 = (\vec{e}_{12}, \vec{e}_{23}, \vec{e}_{31},\vec{e}_{12}, \vec{e}_{23}, \vec{e}_{31})=\mathbf{p}_3^{\,2},
      \]
      a closed and non-backtracking path, which is literally the repetition of the shorter closed path $\mathbf{p}_3$.
      We have $[\mathbf{p}_6]=\{\sigma_c^k(\mathbf{p}_6) \colon 0 \leq k \leq 5\}\in \mathscr{P}.$ %\setminus\mathscr{P}_0
      Since simple graphs admit no closed non-backtracking paths of length 1 or 2, the shortest primitive cycle has length 3, and the \emph{shortest possible non-primitive cycle} has length $2\cdot 3 = 6$ due $[\mathbf{p}_6]=[\mathbf{p}^2_3]$.
    \end{itemize}

    Following this, consider the closed finite path $\mathbf{p}_3 = (\vec{e}_{12}, \vec{e}_{23}, \vec{e}_{31}) \in \mathcal{P}_3^{\mathrm{cl}}$ representing the primitive cycle $[\mathbf{p}_3] \in \mathscr{P}_0$.
    Repeating $\mathbf{p}_3$ indefinitely produces a period path
    $$\widehat{\mathbf{p}}_{+,3} = (\vec{e}_{12}, \vec{e}_{23}, \vec{e}_{31}, \ldots) \in \mathfrak{P}_\Gamma^+, \qquad \widehat{\mathbf{p}}_{-,3} = (\ldots, \vec{e}_{12}, \vec{e}_{23}, \vec{e}_{31}) \in \mathfrak{P}_\Gamma^-,$$ 
    which is an infinite non-backtracking chain.
    By construction, these chains have period 3 under the corresponding shift $\sigma_\pm$ on $\mathfrak{P}_\Gamma^\pm$ as in Definition~\ref{def:shiftspaces}, namely, $\sigma_\pm^3(\widehat{\mathbf{p}}_{\pm,3})=\widehat{\mathbf{p}}_{\pm,3}.$
    Since $\mathbf{p}_3$ is primitive, this period is \emph{minimal}.
    Moreover, the cyclic shifts of $\mathbf{p}_3$ correspond precisely to the shifts of its unrolling $[\mathbf{p}_3]$:
    $$\widehat{\sigma^k_c(\mathbf{p}_3)} =\sigma_\pm^k(\widehat{\mathbf{p}}_{\pm,3}), 
    \qquad 0 \leq k \leq 2.$$
    Thus, the cycle $[\mathbf{p}_3]$ corresponds to the periodic $\sigma_\pm$-orbit
    $\{\widehat{\mathbf{p}}_{\pm,3}, \sigma_\pm(\widehat{\mathbf{p}}_{\pm,3}), \sigma_\pm^2(\widehat{\mathbf{p}}_{\pm,3})\}$ in $\mathfrak{P}_\Gamma^\pm$.
    In this way, we naturally identified $\mathscr{P}_0$ with the set 
    of periodic $\sigma_\pm$-orbits in $\mathfrak{P}_\Gamma^\pm$ having minimal period.

    Lifting $\mathbf{p}_{6}=\mathbf{p}_{3}^{\,2}$ does not produce a new periodic path in $\mathfrak{P}_\Gamma^\pm$, since $\mathbf{p}_{6}$
    traverses the same cycle twice. 
    In fact,
    $
    \widehat{\mathbf{p}}_{6}=\widehat{\mathbf{p}}_{3}.
    $
    Consequently, the orbit generated by $\mathbf{p}_{6}$ coincides with
    that generated by $\mathbf{p}_{3}$:
    \[
    \begin{aligned}
    \bigl\{
    \widehat{\mathbf{p}}_{\pm,6},
    \sigma_\pm(\widehat{\mathbf{p}}_{\pm,6}),
    \dots,
    \sigma_\pm^5(\widehat{\mathbf{p}}_{\pm,6})
    \bigr\}
    &=
    \bigl\{
    \widehat{\mathbf{p}}_{\pm,3},
    \sigma_\pm(\widehat{\mathbf{p}}_{\pm,3}),
    \sigma_\pm^2(\widehat{\mathbf{p}}_{\pm,3})
    \bigr\}.
    \end{aligned}
    \]
    Thus, although $\mathbf{p}_{6}$ has length $6$, its unrolling $\widehat{\mathbf{p}}_{\pm,6}$ has minimal period $3$ under $\sigma_\pm$.
    The six steps of $\mathbf{p}_{6}$ simply traverse the same periodic orbit twice. 
\end{ex}

Example~\ref{exmp:finite_paths}  shows that $\mathscr{P}_0$ corresponds to the periodic $\sigma_\pm$-orbits of minimal period in $\mathfrak{P}_\Gamma^\pm$, whereas $\mathscr{P}$ consists of their nontrivial powers, whose finite representatives traverse the same periodic orbits, associated with their underlying primitive cycle,  multiple times.

Let us generalise this observation in the following lemma.

\begin{lem}[Primitive cycles as periodic $\sigma_\pm$-orbits]
\label{lem:cycles_periodic_orbits}
Let $\mathfrak{G}_\Gamma$ be a finite $(q+1)$-regular graph with $q>1$.
Then, for $n\geq 3$, the map
\begin{equation} \label{eq:closed_paths_fix}
    \mathcal{P}_n^{\mathrm{cl}}
    \stackrel{\sim}{\longrightarrow}
    \operatorname{Fix}(\sigma_\pm^n)\coloneqq \{\mathbf{p}_\pm \in \mathfrak{P}_\Gamma^\pm \;|\; \sigma_\pm^n(\mathbf{p}_\pm)=\mathbf{p}_\pm\}.
\end{equation}
is a well-defined bijection and is compatible with the cyclic shift $\sigma_c$ on $\mathcal{P}_n^{\mathrm{cl}}$ and the shift $\sigma_\pm$ on $\mathfrak{P}_\Gamma^\pm.$

Moreover, passing to primitive cyclic-equivalence classes induces a bijection
$$
    \mathscr{P}_0 
     \quad \stackrel{\sim}{\longrightarrow} \quad \mathscr{O}_{\mathrm{min}}(\mathfrak{P}^\pm_\Gamma) \coloneqq  
     \{\text{periodic $\sigma_\pm$-orbits of minimal period in } \mathfrak{P}^\pm_\Gamma\}.
$$
\end{lem}

\begin{proof}
    For $\mathbf{p}_n = (\vec{e}_1,\dots,\vec{e}_n) \in \mathcal{P}_n^{\mathrm{cl}}$ of length $n \geq 3$, define its \emph{unrolling}
    \[
    \widehat{\mathbf{p}}_{-,n}
    \coloneqq
    \bigl(\dots,\vec{e}_1,\dots,\vec{e}_n\bigr) \in \mathfrak{P}^-_\Gamma, 
    \qquad 
    \widehat{\mathbf{p}}_{+,n}
    \coloneqq
    \bigl(\vec{e}_1,\dots,\vec{e}_n,\dots\bigr) \in \mathfrak{P}^+_\Gamma,
    \]
    by repeating $\mathbf{p}_n$ in corresponding $\pm$-directions.
    This uniquely periodic extension is again non-backtracking everywhere precisely because of the seam condition in Definition~\ref{def:finite_paths}.
    Hence the unrolling $\widehat{\mathbf{p}}_{\pm,n}$ is a well-defined element of $\mathfrak{P}^\pm_\Gamma.$\\
    By construction, the edge sequence of $\widehat{\mathbf{p}}_{+,n}$ is $n$-periodic, i.e., $\vec{e}_{i+n}=\vec{e}_i,$ for all $i \geq 1$ (analogously for $\widehat{\mathbf{p}}_{-,n}$ by reversing). 
    Since $\sigma_\pm$ shifts the basepoint by $n$ edges, it reproduces the same infinite chain.
    Hence
    $\sigma^n_\pm(\widehat{\mathbf{p}}_{\pm,n}) = \widehat{\mathbf{p}}_{\pm,n}
    $.
    It follows that $\widehat{\mathbf{p}}_{\pm,n}\in \operatorname{Fix}(\sigma_\pm^n).$

    Conversely,
    let $\mathbf{p}_+ \in \operatorname{Fix}(\sigma_+^n)$, which by definition of $\sigma_+$ means $\vec{e}_{i+n}=\vec{e}_i,$ for all $i \geq 1$.
    Such a path is uniquely determined by its first $n$ edges $\mathbf{p}_n \coloneqq (\vec{e}_1,\dots,\vec{e}_n)$.
    Since $\mathbf{p}_+\in \mathfrak{P}_\Gamma^+$ is everywhere non-backtracking, so $\mathbf{p}_n$ (as sub-word) and periodicity forces $\tau(\vec{e}_n)=\iota(\vec{e}_{n+1})=\iota(\vec{e}_1)$ (closure) and $\tau(\vec{e}_1)=\tau(\vec{e}_{n+1})\neq \iota(\vec{e}_n)$ (non-backtracking at index $n+1$, i.e., the seam condition). 
    Hence $\mathbf{p}_n \in \mathcal{P}_n^{\mathrm{cl}}$.
    The argument for $\sigma_-$ on $\mathfrak{P}_\Gamma^-$ is identical after reversing all edge orientations.
    
    Note that, by shifting the marked starting edge of $\mathbf{p}_n$ by one, $\sigma_c(\mathbf{p}_n)=(\vec{e}_2, \ldots, \vec{e}_n, \vec{e}_1)$, corresponds exactly to shifting its unrolling $\widehat{\mathbf{p}}_{\pm,n}$ by one step. Hence for $0 \leq k \leq n-1,$ we have
    $$
    \widehat{\sigma^k_c(\mathbf{p}_n)}
    =\sigma^k_\pm(\widehat{\mathbf{p}}_{\pm,n}).
    $$
    Thus, the bijective map \eqref{eq:closed_paths_fix} intertwines the cyclic shift $\sigma_c$ on $\mathcal{P}_n^{\mathrm{cl}}$ with the shift $\sigma_\pm$ on $\operatorname{Fix}(\sigma_\pm^n)$.\\

    Now for the passage to orbits, let $\mathbf{q}_n=(\vec{e}_{k+1}, \ldots, \vec{e}_n, \vec{e}_1, \ldots, \vec{e}_k) \in  \mathcal{P}_n^{\mathrm{cl}}$ be the the rotation of $\mathbf{p}_n \in  \mathcal{P}_n^{\mathrm{cl}}$ by $k$ edges, both of length $n \geq 3$.\\
    These two paths represent the same cycle $[\mathbf{p}]=\{\sigma^k_c(\mathbf{p}): 0 \leq k \leq n-1\} \in \mathscr{P}$ (i.e., differ by a cyclic relabelling of the starting edge) if and only if their periodic extensions satisfy  $\widehat{\mathbf{q}}_{\pm,n}=\sigma_\pm^k(\widehat{\mathbf{p}}_{\pm,n})$ for some $k$, i.e., $\widehat{\mathbf{q}}_{\pm,n}$ belongs to the same $\sigma_\pm$-orbit as $\widehat{\mathbf{p}}_{\pm,n}.$
    By the compatibility just shown, this follows directly and thus
    $\widehat{[\mathbf{p}]} = \{\sigma^k_\pm(\widehat{\mathbf{p}}_{\pm}): 0 \leq k \leq n-1\}$ is exactly the $\sigma_\pm$-orbit of $\widehat{\mathbf{p}}_{\pm}$ inside $\operatorname{Fix}(\sigma_\pm^n)$ for $n$ equal to (a multiple of) its period.
    Since the map \eqref{eq:closed_paths_fix} is a bijection intertwining $\sigma_c$ with $\sigma_\pm$, %and letting $n$ range over $\N_{\geq 3}$
    it descends to a bijection between $\sigma_c$-orbits and $\sigma_\pm$-orbits, i.e., giving the bijection
    \begin{equation} \label{eq:cycle_orbits_bijection}
        [\mathbf{p}] \in \mathscr{P} 
         \quad \stackrel{\sim}{\longrightarrow} \quad \widehat{[\mathbf{p}]} \in \mathscr{O}(\mathfrak{P}^\pm_\Gamma) \coloneqq  
         \{\text{periodic $\sigma_\pm$-orbits on } \mathfrak{P}^\pm_\Gamma\}.
    \end{equation}
    By the power decomposition $[\mathbf{p}]=[\mathbf{p}_0]^{n/\ell_0}$ with $[\mathbf{p}_0] \in \mathscr{P}_0$ of length $\ell_0=\ell([\mathbf{p}_0])$ dividing $n=\ell([\mathbf{p}])$, the orbit $\widehat{[\mathbf{p}]}$ has period exactly $\ell_0$.\\
    Indeed, since $\mathbf{p}$ is literally $\mathbf{p}_0$ repeated $n/\ell_0$ times, we have $\widehat{\mathbf{p}}_\pm = \widehat{\mathbf{p}}_{\pm,0}$ and $\widehat{\mathbf{p}}_{\pm,0}$ has exact minimal period $\ell_0$ under $\sigma_\pm.$ 
    Note that a shorter period would produce a closed path of length $<\ell_0$ of which $\mathbf{p}_0$ is a power, contradicting primitivity.
    Hence $[\mathbf{p}] \in \mathscr{P}_0$ (i.e., $\ell_0=n$ and $\mathbf{p}$ is itself a primitive) if and only if $\widehat{[\mathbf{p}]} \in \mathscr{O}_{\mathrm{min}}(\mathfrak{P}^\pm_\Gamma)$.
    So the bijection \eqref{eq:cycle_orbits_bijection} restricts to the second claimed bijection.
\end{proof}

Note that Lemma~\ref{lem:cycles_periodic_orbits} naturally extends to the space of bi-infinite chains $(\mathfrak{P}_\Gamma,\sigma)$ with its corresponding shift operator as defined in Definition~\ref{def:space_biinfinite_chains}.

%-------------------------------%%--------------------------------------%
\subsection{Trace via periodic paths} \label{sect:trace_formula}
Consider the weighted transfer operator $\mathcal{L}_{\Gamma,\pm,\omega}$ associated with $\omega \in \mathcal{F}_\vartheta(\mathfrak{P}^\pm_\Gamma)$ introduced in Definition~\ref{def:Ruelle_transfer_op}.
Choose $m$ large enough such that $\omega \in F_m(\mathfrak{P}^\pm_\Gamma)$, i.e., $\omega$ is constant on every district of level $m\in \N.$
We call the minimal such $m$ the \emph{memory} of $\omega$.

By Remark~\ref{rem:basis_delta_p}, for $\mathbf{p}_m\in \mathcal{P}_m$, let $\delta_{\mathbf{p}_m} \in F_m(\mathfrak{P}^\pm_\Gamma)$ denote the indicator function of the district determined by $\mathbf{p}_m$, i.e., for $\mathbf{q}_\pm \in \mathfrak{P}^\pm_\Gamma$
$$
\delta_{\mathbf{p}_m}(\mathbf{q}_\pm) = \mathbbm{1}_{\mathcal{C}(\mathbf{p}_m)}(\mathbf{q}_\pm)
=
\begin{cases}
    1, & \text{ if the first $m$ edges of $\mathbf{q}_\pm$ equal $\mathbf{p}_m$,}\\
    0,& \text{ otherwise,}
\end{cases}
$$
where $\mathcal{C}(\mathbf{p}_m) \coloneqq \{ \mathbf{q}_\pm \in \mathfrak{P}_\Gamma^\pm \;|\; \text{the first $m$ edges of $\mathbf{q}_\pm$ equal $\mathbf{p}_m$}\}.$

Since $\omega \in F_m(\mathfrak{P}^\pm_\Gamma)$, by Lemma~\ref{lem:invariant_F}, $F_m(\mathfrak{P}^\pm_\Gamma)$ is invariant under $\mathcal{L}_{\Gamma,\pm,\omega}$.
Hence, the restriction $\mathcal{L}_{\omega, F_m} \coloneqq \mathcal{L}_{\Gamma,\pm,\omega}\big|_{F_m(\mathfrak{P}^\pm_\Gamma)}$ on $F_m(\mathfrak{P}^\pm_\Gamma)$ 
%of $\mathcal{L}_{\Gamma,\pm,\omega}$ to $F_m(\mathfrak{P}^\pm_\Gamma)$ 
is a finite-dimensional operator.
Its matrix in the basis of delta functions $(\delta_{\mathbf{p}_m})_{\mathbf{p}_m\in  \mathcal{P}_m}$ is given by
$$
(\mathcal{L}_{\omega, F_m})_{\mathbf{p}_m,\mathbf{p}'_m}= (\mathcal{L}_{\omega, F_m} \delta_{\mathbf{p}'_m})(\mathbf{q}_\pm)= \sum_{\sigma_\pm(\mathbf{r}_\pm)= \mathbf{q}_\pm} \omega(\mathbf{r}_\pm) \delta_{\mathbf{p}'_m}(\mathbf{r}_\pm), \qquad \mathbf{q}_\pm \in \mathcal{C}(\mathbf{p}_m),
$$
where the indicator $\delta_{\mathbf{p}'_m}$ is nonzero precisely when $\mathbf{r}_\pm \in \mathcal{C}(\mathbf{p}'_m)$.
Thus there is a nonzero contribution exactly when $\mathbf{p}'_m$ can be followed by $\mathbf{p}_m$ in one admissible shift step and $\omega\in F_m(\mathfrak{P}^\pm_\Gamma)$ is constant in $\mathcal{C}(\mathbf{p}'_m)$. Write $\omega_{\mathbf{p}'_m} \coloneqq \omega(\mathbf{r}_\pm)$. 
Then
\begin{align} \label{eq:matrix_L_Fm}
\bigl(\mathcal{L}_{\omega, F_m}\bigr)_{\mathbf{p}_m,\mathbf{p}'_m} 
=
\begin{cases}
\omega_{\mathbf{p}'_m},
&
\text{if $\mathbf{p}'_m,\mathbf{p}_m$ are joined by an admissible chain},\\
0,
&\text{otherwise}
\end{cases}
\end{align}
for $\mathbf{p}'_m,\mathbf{p}_m \in \mathcal{P}_m$.
Here, an admissible (one-step) chain means that $\mathbf{p}_m$ is obtained from $\mathbf{p}'_m$ by dropping its first edge and appending one new edge at the end of the sequence.
Note that the matrix of $\mathcal{L}_{\omega, F_m}$ is well-defined since $\mathcal{L}_{\omega, F_m} \delta_{\mathbf{p}'_m}\in F_m(\mathfrak{P}^\pm_\Gamma)$ and therefore is constant on $\mathcal{C}(\mathbf{p}_m).$

\begin{prop}[Trace of $\mathcal{L}_{\omega, F_m}^n$] \label{prop:trace_Lnw}
    For $3\leq n \leq m$, consider the iterated weighted transfer operator $\mathcal{L}_{\omega, F_m}^n$ given in \eqref{eq:iterated_transfer_op} associated with $\omega \in F_m(\mathfrak{P}^\pm_\Gamma)$.

    Its trace is
        \begin{equation} \label{eq:trace_Ln}
        \operatorname{Tr}\bigl(\mathcal{L}_{\omega, F_m}^n\bigr)
        =
        \sum_{\mathbf{p}_\pm \in \operatorname{Fix}(\sigma_\pm^n)} \omega_n(\mathbf{p}_\pm)
        = \sum_{\substack{[\mathbf{p}_0] \in \mathscr{P}_0 \\ \ell_0 \,\mid\, n}} \ell_0 \; \omega_{\ell_0}(\mathbf{p}_0)^{n/\ell_0},
    \end{equation}
    where $\ell_0$ denotes the length of the underlying primitive cycle $[\mathbf{p}_0] \in \mathscr{P}_0$, $\mathbf{p}_0$ denotes any fixed representative of $[\mathbf{p}_0] \in \mathscr{P}_0,$ and the Birkhoff product $\omega_n$ of $\omega$ is defined in \eqref{eq:w_m_product}.
\end{prop}

\begin{proof}
    Fix the memory $m$ of $\omega$ sufficiently large and let $3 \leq n \leq m$.

    \medskip
    \noindent\textit{Step 1: The matrix $\mathcal{L}_{\omega,F_m}^n$ and its diagonal entries.}
    By definition of the matrix $(\mathcal{L}_{\omega, F_m}\bigr)_{\mathbf{p},\mathbf{p}'}$ in \eqref{eq:matrix_L_Fm}, its $n$-th power has entries
    $$
        (\mathcal{L}^n_{\omega, F_m}\bigr)_{\mathbf{p},\mathbf{p}'}
        =
        \sum_{\mathbf{p}_1,\dots,\mathbf{p}_{n-1} \in \mathcal{P}_m} (\mathcal{L}_{\omega, F_m})_{\mathbf{p},\mathbf{p}_{n-1}} \cdots (\mathcal{L}_{\omega, F_m})_{\mathbf{p}_{1},\mathbf{p}'}, \qquad \mathbf{p},\mathbf{p}' \in \mathcal{P}_m.
    $$
    Each factor $(\mathcal{L}_{\omega, F_m})_{\mathbf{p}_{i+1},\mathbf{p}_{i}} \neq 0$ if $\mathbf{p}_i,\mathbf{p}_{i+1} \in \mathcal{P}_m$, $i=0, \ldots, n -1$ (with convention $\mathbf{p}_0 \coloneqq \mathbf{p}'$ and $\mathbf{p}_{n-1} \coloneqq \mathbf{p}$) are consecutive finite length-$m$ paths of a single admissible chain.
    Hence, the $n$-fold matrix product collapses to
    \[
    (\mathcal{L}_{\omega,F_m}^n)_{\mathbf{p},\mathbf{p}'}
    =
    \begin{cases}
    \omega_{n,\mathbf{p}'}, & \text{if $\mathbf{p}',\mathbf{p} \in \mathcal{P}_m$ are joined by an admissible $n$-step chain in $\mathcal{P}_m$},\\
    0, & \text{otherwise},
    \end{cases}
    \]
    where an admissible $n$-step chain $\mathbf{p}' \to \mathbf{p}_1 \to \cdots \to \mathbf{p}_{n-1}\to \mathbf{p}$ glues into a single non-backtracking edge sequence $\vec{e}_1,\dots,\vec{e}_{n+m-1}$, and
    $$
    \omega_{n,\mathbf{p}'} \coloneqq \omega_{\mathbf{p}_{n-1}}\,\omega_{\mathbf{p}_{n-2}}\cdots\omega_{\mathbf{p}_1}\,\omega_{\mathbf{p}'}
    $$
    is well-defined independently of any continuation of this finite path. Indeed, for any $\mathbf{r}_\pm \in \mathcal{C}(\mathbf{p}')$ whose first $n+m-1$ edges realize this glued sequence, $\omega$ having memory $m$ means every factor $\omega(\sigma_\pm^k(\mathbf{r}_\pm))$, $k=0,\dots,n-1$, in \eqref{eq:w_m_product} depends only on edges $\vec{e}_{k+1},\dots,\vec{e}_{k+m}$, all of which lie among $\vec{e}_1,\dots,\vec{e}_{n+m-1}$.
    Hence $\omega_n(\mathbf{r}_\pm) = \omega_{n,\mathbf{p}'}$ for every $\mathbf{r}_\pm$ in the district $\mathcal{C}(\mathbf{p}')$ compatible with the chain, matching \eqref{eq:w_m_product}.

    \medskip
    \noindent\textit{Step 2: Diagonal entries correspond to closed paths.}
    As $F_m(\mathfrak{P}_\Gamma^\pm)$ is finite-dimensional with basis $(\delta_{\mathbf{p}})_{\mathbf{p}\in\mathcal{P}_m}$, the trace of $\mathcal{L}_{\omega,F_m}^n$ is the sum of its diagonal entries over the finite index set $\mathcal{P}_m$:
    \[
    \operatorname{Tr}\bigl(\mathcal{L}_{\omega,F_m}^n\bigr) = \sum_{\mathbf{p}\in\mathcal{P}_m} (\mathcal{L}_{\omega,F_m}^n)_{\mathbf{p},\mathbf{p}}.
    \]
    By Step 1, $(\mathcal{L}_{\omega,F_m}^n)_{\mathbf{p},\mathbf{p}}\neq 0$ precisely when $\mathbf{p}\in\mathcal{P}_m$ admits an admissible $n$-step chain from $\mathbf{p}$ back to itself.
    More explicitly, for a diagonal entry starting with $\mathbf{p} = (\vec{e}_1,\dots,\vec{e}_m) \in \mathcal{P}_m$, this requires that after $n$ shifts we return to the same $m$-block, that means $\mathbf{p}$ is $n$-periodic, i.e., $\vec{e}_{n+i} = \vec{e}_i$ for all $1 \leq i \le m-n$.
    Hence the entire $m$-block is determined by its first $n$ edges $\mathbf{p}_n \coloneqq (\vec{e}_1,\dots,\vec{e}_n)$, and those first $n$ edges form a closed $n$-path, i.e., $\mathbf{p}_n \in \mathcal{P}_n^{\mathrm{cl}}$.

    Conversely, every $\mathbf{p}_n\in\mathcal{P}_n^{\mathrm{cl}}$ determines a unique such $\mathbf{p}\in\mathcal{P}_m$ by $n$-periodic continuation up to length $m \geq n$. %, namely the unique element of $\mathcal{C}(\mathbf{p}_n)\cap\mathcal{P}_m$ compatible with this periodicity. 
    This gives a bijection
    \[
    \bigl\{\mathbf{p}\in\mathcal{P}_m : (\mathcal{L}_{\omega,F_m}^n)_{\mathbf{p},\mathbf{p}}\neq 0\bigr\}
    \;\xrightarrow{\ \sim\ }\;
    \mathcal{P}_n^{\mathrm{cl}}, \qquad \mathbf{p} \longmapsto \mathbf{p}_n,
    \]
    under which $(\mathcal{L}_{\omega,F_m}^n)_{\mathbf{p},\mathbf{p}} =\omega_{n,\mathbf{p}_n}$, since $\omega\in F_m(\mathfrak{P}_\Gamma^\pm)$ depends only on the first $m\ge n$ edges and the Birkhoff product along the self-returning chain coincides with $\omega_n$ evaluated on the periodic extension of $\mathbf{p}_n$. Equivalently, on any $\mathbf{r}_\pm \in \mathcal{C}(\mathbf{p}_n) \subset \mathfrak{P}_\Gamma^\pm$ that is itself $n$-periodic.
    Thus
    \[
    \operatorname{Tr}\bigl(\mathcal{L}_{\omega,F_m}^n\bigr) 
    = \sum_{\mathbf{p}_n \in \mathcal{P}_n^{\mathrm{cl}}} \omega_{n,\mathbf{p}_n}.
    \]

    \medskip
    \noindent\textit{Step 3: Passage to $\operatorname{Fix}(\sigma_\pm^n)$.}
    By Lemma~\ref{lem:cycles_periodic_orbits}, the periodic-extension map is a bijection $\mathcal{P}_n^{\mathrm{cl}} \xrightarrow{\sim} \operatorname{Fix}(\sigma_\pm^n)$, sending $\mathbf{p}_n$ to its unique periodic extension $\widehat{\mathbf{p}}_{\pm,n} \in \operatorname{Fix}(\sigma_\pm^n),$ %\in \mathcal{C}(\mathbf{p}_n) \subset \mathfrak{P}_\Gamma^\pm$, 
    under which $\omega_{n,\mathbf{p}_n}$ corresponds to $\omega_n(\widehat{\mathbf{p}}_{\pm,n})$. Hence
    \[
    \operatorname{Tr}\bigl(\mathcal{L}_{\omega,F_m}^n\bigr) = \sum_{\mathbf{p}_n\in\mathcal{P}_n^{\mathrm{cl}}}\omega_{n,\mathbf{p}_n} = \sum_{\widehat{\mathbf{p}}_{\pm,n}\in\operatorname{Fix}(\sigma_\pm^n)}\omega_n(\widehat{\mathbf{p}}_{\pm,n}),
    \]
    which proves the first claimed equality in \eqref{eq:trace_Ln}.

    \medskip
    \noindent\textit{Step 4: Reorganisation via primitive cycles.}
    Since equivalence of closed finite paths, i.e., 
    cyclic relabeling of the starting edge, leaves $\omega_n$ unchanged, so $\omega_n$ is constant on classes $[\mathbf{p}_n]\in \mathscr{P}$ of length $n \geq 3$.
    
    By the power decomposition, there exists a unique primitive cycle $[\mathbf{p}_0]\in\mathscr{P}_0$ of length $\ell_0=\ell([\mathbf{p}_0])$ dividing $n$
    such that
    $[\mathbf{p}_n]=[\mathbf{p}_0]^{n/\ell_0}$.
    Each such $[\mathbf{p}_0]$ contributes exactly $\ell_0$ distinct cyclic representatives in $\mathcal{P}_n^{\mathrm{cl}}$ and all of them carry the same potential weight.\\
    Indeed, by Step 3, let $\widehat{\mathbf{p}}_{\pm,n}\in \operatorname{Fix}(\sigma_\pm^n)$
    be the unique periodic extension corresponding to $\mathbf{p}_n\in \mathcal{P}_n^{\mathrm{cl}}.$ 
    Since $\mathbf{p}_n$ is obtained by repeating $\mathbf{p}_0$ exactly $n/ \ell_0$ times, their periodic unrollings coincide $\widehat{\mathbf{p}}_{\pm,0}
    =\widehat{\mathbf{p}}_{\pm,n}\in \operatorname{Fix}(\sigma_\pm^n)$ and $\widehat{\mathbf{p}}_{\pm,0}$ has exact minimal period $\ell_0$ under $\sigma_\pm.$
    Therefore
    \begin{eqnarray*}
         \omega_{n,\mathbf{p}_n} 
        = \omega_{n}(\widehat{\mathbf{p}}_{\pm,n})
        = \Biggr(\prod_{k=0}^{\ell_0-1} \omega(\sigma_\pm^k(\widehat{\mathbf{p}}_{\pm,0}))\Biggl)^{n/\ell_0}
        &=&
        \omega_{\ell_0}(\widehat{\mathbf{p}}_{\pm,0})^{n/\ell_0}\\
        %\eqqcolon\omega_{\ell_0}([\mathbf{p}_0])^{n/\ell_0}.
        &=& \omega_{\ell_0}(\mathbf{p}_0)^{n/\ell_0} \quad \forall [\mathbf{p}_0] \in \mathscr{P}_0, \ell_0\mid n.
    \end{eqnarray*}
    This is well-defined on the primitive cycle $[\mathbf{p}_0]$, since changing the starting point along the cycle cyclically permutes the factors in the product.\\
    Consequently, the entire primitive cycle contributes $\ell_0 \cdot \omega_{\ell_0}(\mathbf{p}_0)^{n/\ell_0}.$
    Summing over all primitive cycles whose lengths divide $n$ and using Step 3 yields the second equality in \eqref{eq:trace_Ln}.
\end{proof}

\begin{rem}
    Lemma~\ref{lem:ess_spec_radius_bound} shows that the essential spectral radius of $\mathcal{L}_{\Gamma, \pm, \omega}$ on $\F_{\vartheta}(\mathfrak{P}^\pm_\Gamma)$ is
    $$r_{\mathrm{ess}}(\mathcal{L}_{\Gamma, \pm, \omega}) \leq \vartheta q \|\omega|\|_\infty, \quad \text{ for } \|\omega\|_\infty >0,
    $$
    which is strictly positive for any fixed $\vartheta \in (0,1)$.
    A trace-class (let alone nuclear) operator has essential spectral radius 0.
    So $\mathcal{L}_{\Gamma, \pm, \omega}$ is not compact on $\F_{\vartheta}(\mathfrak{P}^\pm_\Gamma)$.
    However, the trace $\mathrm{Tr}[\mathcal{L}^n_{\Gamma, \pm, \omega}]$ in Proposition~\ref{prop:trace_Lnw} is defined as a flat (or algebraic) trace via the finite-dimensional matrix on $F_m(\mathfrak{P}^\pm_\Gamma)$.
    This needs no operator-trace-class hypothesis on $\F_{\vartheta}(\mathfrak{P}^\pm_\Gamma)$ at all.
\end{rem}

\begin{cor}[Trace of $\mathcal{L}_{\omega,F_m}^n$ against $f\in C^{\mathrm{lc}}(\mathfrak{P}_\Gamma^\pm)$]
\label{cor:trace_Lnw_f}
    Choose $m$ large enough so that
    $\omega\in F_m(\mathfrak{P}_\Gamma^\pm)$ have memory $m$, and $f\in C^{\mathrm{lc}}(\mathfrak{P}_\Gamma^\pm)$ also have depth at most $m$.
    For $3\le n\le m$
    \[
    \operatorname{Tr}\bigl(\mathcal{L}_{\omega,F_m}^n\, f\bigr)
    \coloneqq
    \operatorname{Tr}\bigl(M_f\,\mathcal{L}_{\omega,F_m}^n\bigr)
    =
    \sum_{\mathbf{p}_n\in\mathcal{P}_n^{\mathrm{cl}}} \omega_n(\mathbf{p}_n)\, f(\mathbf{p}_n)
    =
    \sum_{\substack{[\mathbf{p}_0]\in\mathscr{P_0}\\ \ell_0\mid n}} \omega_{\ell_0}(\mathbf{p}_0)^{n/\ell_0} \sum_{k=0}^{\ell_0-1} f\bigl(\sigma_\pm^k(\widehat{\mathbf{p}}_0)\bigr),
    \]
    where $M_f$ denotes the diagonal matrix representing pointwise multiplication by $f$ on $F_m(\mathfrak{P}_\Gamma^\pm)$ and
    $f(\mathbf{p}_n)$ denotes $f$ evaluated at the depth-$m$ truncation of the periodic extension of $\mathbf{p}_n \in \mathcal{P}_n^{\mathrm{cl}}$.
    Here $\mathbf{p}_0$ denotes any fixed representative of $[\mathbf{p}_0] \in \mathscr{P}_0$ and $\widehat{\mathbf{p}}_0 \in \operatorname{Fix}(\sigma_\pm^n) \subset \mathfrak{P}_\Gamma^\pm$ its periodization (see Proposition~\ref{prop:trace_Lnw}).
\end{cor}

\begin{proof}
    Since $f\in C^{\mathrm{lc}}(\mathfrak{P}_\Gamma^\pm)$ also have depth at most $m$, i.e., $f\in F_m(\mathfrak{P}_\Gamma^\pm)$, $f$ acts on $F_m(\mathfrak{P}_\Gamma^\pm)$ by pointwise multiplication:
    $$
    M_f: F_m(\mathfrak{P}_\Gamma^\pm) \longrightarrow F_m(\mathfrak{P}_\Gamma^\pm), \qquad (M_f\varphi)(\mathbf{p}_m)=f(\mathbf{p}_m)\varphi(\mathbf{p}_m), \quad \mathbf{p}_m \in \mathcal{P}_m.
    $$
    Hence both $M_f$ and $\mathcal{L}_{\omega,F_m}$ are endomorphisms of the same finite-dimensional space
    $F_m(\mathfrak{P}_\Gamma^\pm)$ and therefore
    $M_f\mathcal{L}_{\omega,F_m}^n$ is an endomorphism whose trace is well defined.
    
    In the basis $(\delta_{\mathbf{p}_m})_{\mathbf{p}_m\in \mathcal{P}_m},$ $(M_f)_{\mathbf{p}_m,\mathbf{p}'_m}=f(\mathbf{p}_m)\delta_{\mathbf{p}_m,\mathbf{p}'_m}.$
    By definition, 
    $$\operatorname{Tr}(M_f\mathcal{L}_{\omega,F_m}^n) 
    = \sum_{\mathbf{p}_m\in\mathcal{P}_m} (M_f\mathcal{L}_{\omega,F_m}^n)_{\mathbf{p}_m,\mathbf{p}_m}
    = \sum_{\mathbf{p}_m\in\mathcal{P}_m} f(\mathbf{p}_m)\,(\mathcal{L}_{\omega,F_m}^n)_{\mathbf{p}_m,\mathbf{p}_m}.$$
    By Step~2 of the proof of Proposition~\ref{prop:trace_Lnw}, $(\mathcal{L}_{\omega,F_m}^n)_{\mathbf{p}_m,\mathbf{p}_m}\neq 0$ 
    precisely for those $\mathbf{p}_m\in\mathcal{P}_m$ that are the depth-$m$ periodic continuation of a (unique) closed finite path $\mathbf{p}_n\in\mathcal{P}_n^{\mathrm{cl}}$, with value $\omega_{n,\mathbf{p}_n}$.\\
    Since $f$ has depth at most $m$, $f(\mathbf{p}_m)$ depends only on this same periodic continuation, so we may write $f(\mathbf{p}_n)\coloneqq f(\mathbf{p}_m)$ unambiguously.
    Hence
    \[
    \operatorname{Tr}(M_f\mathcal{L}_{\omega,F_m}^n) 
    = \sum_{\mathbf{p}_n\in\mathcal{P}_n^{\mathrm{cl}}} \omega_{n,\mathbf{p}_n}\,f(\mathbf{p}_n).
    \]
    
    The second equality follows exactly as in Step~4 of the proof of Proposition~\ref{prop:trace_Lnw}: 
    every $[\mathbf{p}_n]\in\mathscr{P}$ of length $n$ decomposes uniquely as $[\mathbf{p}_n]=[\mathbf{p}_0]^{n/\ell_0}$ for a primitive $[\mathbf{p}_0]\in\mathscr{P_0}$ of length $\ell_0 \mid n$, contributing exactly $\ell_0$ cyclic representatives to $\mathcal{P}_n^{\mathrm{cl}}$.\\
    Each carrying the same weight $\omega_n=\omega_{\ell_0}(\mathbf{p}_0)^{n/\ell_0}$ but (possibly) different values of $f$.
    Namely $f$ evaluated at the $\ell_0$ points $\sigma_\pm^k(\widehat{\mathbf{p}}_0)$, $k=0,\dots,\ell_0-1$, on the orbit of the periodization $\widehat{\mathbf{p}}_0 \in \operatorname{Fix}(\sigma_\pm^n)$.
\end{proof}

Consider now the space of bi-infinite chains $\mathfrak{P}_\Gamma \subset \mathfrak{P}^+_\Gamma \times \mathfrak{P}^-_\Gamma$ with its shift operator $\sigma$ from Definition~\ref{def:space_biinfinite_chains}.
Recall the canonical future projection
$\pi_{+}: \mathfrak{P}_\Gamma \rightarrow \mathfrak{P}_\Gamma^+$ given by $\pi_\pm(\mathbf{p})=(\vec{e}_1, \vec{e}_2, \ldots)$
from \eqref{eq:pi_shift_proj}.
It commutes with the shift maps:
\begin{equation} \label{eq:pi+_sigma+}
    \pi_+ \circ \sigma = \sigma_+ \circ \pi_+.
\end{equation}

After applying a sufficiently large iterate of the shift $\sigma$, the (two-sided) function $f \in C^{\mathrm{lc}}(\mathfrak{P}_\Gamma)$ factors through the canonical future projection. More precisely, we have the following.

\begin{lem}[One-sided reduction of a two-sided function] \label{lem:one-sided-reduction}
    Let $f\in C^{\mathrm{lc}}(\mathfrak{P}_\Gamma).$
    Then there exist $m\geq 1$ and a function $f_+\in C^{\mathrm{lc}}(\mathfrak{P}_\Gamma^+)$ such that
    $$f \circ \sigma^m = f_+ \circ \pi_+.$$
\end{lem}

\begin{proof}
    Since $f$ is locally constant on the compact space $\mathfrak{P}_\Gamma$, its level sets are open and cover $\mathfrak{P}_\Gamma$.
    By compactness, finitely many of them already cover the space. 
    Each of these finitely many sets in the subcover, being open in the product topology on $\mathfrak{P}_\Gamma$, depends on only finitely many coordinates around the basepoint, and since there are only finitely many, we may take $m,m'\geq 1$ large enough to work uniformly for every element of the (finite) subcover.
    Therefore, due to compactness, it allows us to turn the pointwise dependence of $f$ on finitely many edges into a single pair $(m,m')$ that works uniformly over all of $\mathfrak{P}_\Gamma.$
    
    More precisely, there exist $m,m' \geq 1$ large enough and a function $ \tilde{f}\in F_{m+m'}(\mathfrak{P}_\Gamma)$ , depending only on $m$ most recent past and $m'$ future edges at the basepoint, such that for every $\mathbf{p}=(\mathbf{p}_+,\mathbf{p}_-) \in \mathfrak{P}_\Gamma$ and $\mathbf{p}_{m,m'} \in \mathcal{P}_{m+m'}$, we have
    $$f(\mathbf{p})
    =f(\mathbf{p}_{+}, \mathbf{p}_{-})
    =\tilde{f}(\mathbf{p}_{+,m'}, \mathbf{p}_{-,m})
    = \tilde{f}(\mathbf{p}_{m,m'}).$$
    
    Write $\mathbf{p}=((\vec{e}_{1}, \vec{e}_{2}, \ldots),(\ldots, \vec{e}_{-2}, \vec{e}_{-1}))\in \mathfrak{P}_\Gamma,$ so that $$\mathbf{p}_{m,m'}=((\vec{e}_1, \ldots, \vec{e}_{m'}),(\vec{e}_{-m}, \ldots, \vec{e}_{-1})) \in \mathcal{P}_{m+m'}.$$
    By definition of the shift, $\sigma$ shifts the basepoint one step forward, hence after $m$ iterations we have
    $
    \sigma^m(\mathbf{p}) = \big((\vec{e}_{m+1}, \vec{e}_{m+2}, \ldots), (\ldots, \vec{e}_{-1}, \vec{e}_1, \ldots, \vec{e}_m)\big).
    $
    Evaluating $\tilde{f}$ at $\sigma^m$, therefore gives
    $$
    (f \circ \sigma^m)(\mathbf{p})
    = \tilde f\big((\vec{e}_1, \ldots, \vec{e}_m), (\vec{e}_{m+1}, \ldots, \vec{e}_{m+m'})\big)
    = \tilde f(\vec{e}_1, \ldots, \vec{e}_{m+m'}),
    $$
    where in the last step we used that $\tilde f$ depends only on the chronologically ordered chain of
    $m + m'$ consecutive edges,
    i.e., the two-sided chain consisting of the past and future chain
    $\big((\vec{e}_1,\ldots,\vec{e}_m),(\vec{e}_{m+1},\ldots,\vec{e}_{m+m'})\big)$ may equivalently be read as
    the single ordered string $(\vec{e}_1, \ldots, \vec{e}_{m+m'})$.
    
    Now define $f_+ \in \mathrm{Maps}(\mathfrak{P}_\Gamma^+, \C)$ by
    $
    f_+(\vec{e}_1, \vec{e}_2, \ldots) \coloneqq \tilde f(\vec{e}_1, \ldots, \vec{e}_{m+m'}).
    $
    This depends only on the first $m+m'$ coordinates of a point of $\mathfrak{P}_\Gamma^+$, hence
    $f_+ \in C^{\mathrm{lc}}(\mathfrak{P}_\Gamma^+)$. Finally, since $\pi_+(\mathbf{p}) = (\vec{e}_1, \vec{e}_2, \ldots)$,
    $$
    (f_+ \circ \pi_+)(\mathbf{p}) = f_+(\vec{e}_1, \vec{e}_2, \ldots) = \tilde f(\vec{e}_1, \ldots, \vec{e}_{m+m'})
    = (f \circ \sigma^m)(\mathbf{p}),
    $$
    for every $\mathbf{p} \in \mathfrak{P}_\Gamma$, which is the desired identity.
\end{proof}

For $f\in C^{\mathrm{lc}}(\mathfrak{P}_\Gamma)$ and $n\ge 3$, let $\mathbf{p}\in\operatorname{Fix}(\sigma^n)$ and write $\ell_0=\ell_0(\mathbf{p})$ for its minimal period, so $\ell_0\mid n$. Define the associated \emph{periodic weight} by
\begin{equation} \label{eq:weight_f}
     \mathbf{f}_{\ell_0}(\mathbf{p})
    \coloneqq
    \sum_{k=0}^{\ell_0-1} f(\sigma^k(\mathbf{p})), \qquad \mathbf{p}\in \operatorname{Fix}(\sigma^{n}), \; n \geq 3.
\end{equation}
The same formula, with $\sigma$ replaced by $\sigma_+$, defines $\mathbf{f}_{\ell_0,+}$ on $\operatorname{Fix}(\sigma_+^{n})\subset\mathfrak{P}_\Gamma^+$ for $f_+\in C^{\mathrm{lc}}(\mathfrak{P}_\Gamma^+)$.

Due the $n$-periodicity, $\mathbf{f}_{\ell_0}$ is constant on $\sigma$-orbits within $\operatorname{Fix}(\sigma^{n})$, i.e., it is  $\sigma$-invariant.
Hence, by iteration on $m\geq 1$, we derive $ \mathbf{f}_{\ell_0}\circ \sigma^m= \mathbf{f}_{\ell_0}.$\\
Moreover, by Lemma~\ref{lem:one-sided-reduction} and using \eqref{eq:pi_shift_proj} so that $\pi_+ \circ \sigma^k = \sigma_+^k \circ \pi_+$, $0 \leq k \leq \ell_0-1$, we get for $\mathbf{p}\in \operatorname{Fix}(\sigma^{n})$
\begin{equation} \label{eq:fnweight_sigma_f+}
    \mathbf{f}_{\ell_0}(\mathbf{p})
    = \mathbf{f}_{\ell_0}(\sigma^m(\mathbf{p})) 
    = \sum_{k=0}^{\ell_0-1}(f\circ\sigma^m)(\sigma^k(\mathbf{p})) 
    = \sum_{k=0}^{\ell_0-1} f_+(\sigma_+^k(\pi_+(\mathbf{p}))) 
    = \mathbf{f}_{\ell_0,+}(\pi_+(\mathbf{p})),
\end{equation}
since  $\pi_+(\mathbf{p})\in\operatorname{Fix}(\sigma_+^n)$ has the same minimal period $\ell_0$ as $\mathbf{p}$.
Thus the periodic-orbit weight obtained from $f\in C^{\mathrm{lc}}(\mathfrak{P}_\Gamma)$ is exactly the periodic-orbit weight generated by the one-sided potential $f_+\in C^{\mathrm{lc}}(\mathfrak{P}_\Gamma^+)$.

Now let $m$ be sufficiently large so that $f\in C^{\mathrm{lc}}(\mathfrak{P}_\Gamma)$ has depth at most $m$. For $3\leq n\leq m$, we define the \emph{diagonal trace} against $f$ by
\begin{equation} \label{eq:diagonal_trace_f}
    \operatorname{Tr}\bigl[\mathcal{L}_{1,F_m}^n f\bigr] \coloneqq \sum_{\mathbf{p}_n\in\mathcal{P}_n^{\mathrm{cl}}} f(\widehat{\mathbf{p}}_n),
\end{equation}
where $\widehat{\mathbf{p}}_n \in \operatorname{Fix}(\sigma^n) \subset \mathfrak{P}_\Gamma$ is the periodic extension of the closed path $\mathbf{p}_n\in\mathcal{P}_n^{\mathrm{cl}}$ in $\mathfrak{P}_\Gamma$
and $\mathcal{L}^n_{1,F_m}$ is the unweighted Ruelle transfer operator on $F_m(\mathfrak{P}_\Gamma^\pm)$.

\begin{lem}[Diagonal trace of $\mathcal{L}^n_{1,F_m}$ against $f\in C^{\mathrm{lc}}(\mathfrak{P}_\Gamma)$] \label{lem:trace_unweighted_multiplicative}
    %Let $m$ be sufficiently large.
    Let $f\in C^{\mathrm{lc}}(\mathfrak{P}_\Gamma)$, and let $m$ be sufficiently large so that $f$ has depth at most $m$.
    Then for $3\leq n\leq m$
    $$
    \operatorname{Tr}[\mathcal{L}^n_{1,F_m} f]
    = \sum_{\substack{[\mathbf{p}_0]\in\mathscr{P}_0\\ \ell_0\mid n}} \mathbf{f}_{\ell_0}(\widehat{\mathbf{p}}_0),
    $$
    where $\widehat{\mathbf{p}}_0 \in \operatorname{Fix}(\sigma^n) \subset \mathfrak{P}_\Gamma$ is the periodization of any representative $\mathbf{p}_0$ of the primitive cycle $[\mathbf{p}_0] \in \mathscr{P}_0$ and $\mathbf{f}_{\ell_0}$ is defined in \eqref{eq:weight_f}.

    Moreover, let  $f_+\in C^{\mathrm{lc}}(\mathfrak{P}_\Gamma^+)$ of depth at most $m$ be the one-sided reduction of $f$ as in Lemma~\ref{lem:one-sided-reduction}, then
    \begin{equation} \label{eq:trace_unweighted_multiplicative}
        \operatorname{Tr}[\mathcal{L}^n_{1,F_m} f]
        = \sum_{\substack{[\mathbf{p}_0]\in\mathscr{P}_0\\ \ell_0\mid n}} \mathbf{f}_{\ell_0,+}(\pi_+(\widehat{\mathbf{p}}_0))
        = \operatorname{Tr}[M_{f_+}\mathcal{L}^n_{1,+,F_m}]
        \eqqcolon \operatorname{Tr}[\mathcal{L}^n_{1,+,F_m} f_+],
    \end{equation}
    where $\mathbf{f}_{\ell_0,+}$ is defined in \eqref{eq:weight_f} and $\mathcal{L}_{1,+,F_m}$ is the unweighted Ruelle transfer operator restricted to $F_m(\mathfrak{P}_\Gamma^+).$
\end{lem}

\begin{proof}
    By Lemma~\ref{lem:cycles_periodic_orbits} applied to the bi-infinite setting $(\mathfrak{P}_\Gamma, \sigma)$, the closed paths in $\mathcal{P}_n^{\mathrm{cl}}$ in $\mathfrak{P}_\Gamma$ are in bijection with periodic points in $\operatorname{Fix}(\sigma^n)$.
    Moreover, the $\operatorname{Fix}(\sigma^n)$ decomposes into $\sigma$-orbits of minimal period dividing $n$. These minimal periodic orbits
    are in bijection with the primitive cyclic-equivalence classes $[\mathbf{p}_0]\in\mathscr{P}_0$.
    Therefore
    $$
    \operatorname{Tr}\bigl[\mathcal{L}_{1,F_m}^n f\bigr]
    = \sum_{\widehat{\mathbf{p}}_n\in\operatorname{Fix}(\sigma^n)} f(\widehat{\mathbf{p}}_n)
    = \sum_{\substack{[\mathbf{p}_0]\in\mathscr{P}_0\\ \ell_0\mid n}}
    \sum_{k=0}^{\ell_0-1}
    f(\sigma^k(\widehat{\mathbf{p}}_0)) 
    =
    \sum_{\substack{[\mathbf{p}_0]\in\mathscr{P}_0\\\ell_0\mid n}}
    \mathbf{f}_{\ell_0}(\widehat{\mathbf{p}}_0),$$
    by the definition \eqref{eq:weight_f} of the periodic weight
    $\mathbf{f}_{\ell_0}$. Using \eqref{eq:fnweight_sigma_f+}, this implies the first equality in \eqref{eq:trace_unweighted_multiplicative} directly.
    
    On the other hand, applying Corollary~\ref{cor:trace_Lnw_f} for $f_+\in C^{\mathrm{lc}}(\mathfrak{P}_\Gamma^+)$ of depth at most $m$ and with $\omega \equiv 1$, we obtain
    $$
    \operatorname{Tr}[\mathcal{L}_{1,+,F_m}^n f_+]
    =
    \sum_{\mathbf{p}_{n,+} \in\mathcal{P}_{n,+}^{\mathrm{cl}}}
    f_+(\mathbf{p}_{n,+})
    = \sum_{\substack{[\mathbf{p}_{0,+}]\in\mathscr{P}_{0,+}\\ \ell_0\mid n}}
    \mathbf{f}_{\ell_0,+}
    (\widehat{\mathbf{p}}_{0,+}),
    $$
    where $\widehat{\mathbf{p}}_{0,+} \in \operatorname{Fix}(\sigma_\pm^n) \subset \mathfrak{P}_\Gamma^+$ is the periodization of any fixed representative $\mathbf{p}_{0,+}$ of $[\mathbf{p}_{0,+}]\in\mathscr{P}_{0,+}.$
    Here we distinguish with the index $+$ in $\mathcal{P}_{n,+}^{\mathrm{cl}}$ and $\mathscr{P}_{0,+}$ the finite closed paths respectively primitive cycles living in $\mathfrak{P}_\Gamma^+.$ 
    Now by using the canonical projection $\widehat{\mathbf{p}}_{0,+}=\pi_+(\widehat{\mathbf{p}}_{n})$ for $\widehat{\mathbf{p}}_{n} \in \operatorname{Fix}(\sigma^n) \subset \mathfrak{P}_\Gamma$, this complete the proof.
\end{proof}

% -------------------------------%--------------------------------------%
\subsection{Weighted trace formula}

We now introduce the weighted dynamical Ihara zeta function on finite $(q+1)$-regular graphs.

\begin{definition}[Weighted dynamical Ihara zeta function]
\label{def:zeta_weighted}
    Let $\mathfrak{G}_\Gamma$ be a finite $(q+1)$-regular graph with $q>1$.
    The \emph{weighted dynamical Ihara zeta function} associated to the weight $f\in C^{\mathrm{lc}}(\mathfrak{P}_\Gamma)$ is 
    \begin{equation} \label{eq:Zeta}
        \mathbf{Z}_f(s) 
        \coloneqq  
        \sum_{[\mathbf{p}_0]\in\mathscr{P}_0} \frac{q^{(\frac12+is)\ell_0}}{1-q^{(\frac12+is)\ell_0}}  \mathbf{f}_{\ell_0}(\widehat{\mathbf{p}}_0),
        \quad s\in \C,
    \end{equation}
     where the sum runs over all primitive cycles $[\mathbf{p}_0]\in\mathscr{P}_0$ of length $\ell_0$ and $\mathbf{f}_{\ell_0}$ is the associated periodic weight of $f$ as in \eqref{eq:weight_f} acting on the periodization
    $\widehat{\mathbf{p}}_0 \in \operatorname{Fix}(\sigma^{\ell_0}) \subset \mathfrak{P}_\Gamma$ of any representative $\mathbf{p}_0$ of $[\mathbf{p}_0]$.
\end{definition}
Expanding the geometric series shows that $\mathbf{Z}_f$ is in fact a sum over \emph{all} (not necessarily primitive) closed orbits of $\sigma$, each weighted by its underlying primitive period:
\begin{equation*}
        \mathbf{Z}_f(s) 
        =  
        \sum_{[\mathbf{p}_0]\in\mathscr{P}_0} 
        \sum_{k=1}^\infty q^{(\frac12+is)k\ell_0} \mathbf{f}_{\ell_0}(\widehat{\mathbf{p}}_0)
        = \sum_{[\mathbf{p}]\in\mathscr{P}} q^{(\frac12+is)\ell([\mathbf{p}])} \mathbf{f}_{\ell_0}(\widehat{\mathbf{p}}_0).
    \end{equation*}
This exhibits $\mathbf{Z}_f$ as the discrete, non-Archimedean counterpart of the weighted zeta function of Schütte--Weich~\cite[Eq.~(1.3)]{SchutteWeichBarkhofen23_weightedzeta}.

Note that by applying \eqref{eq:fnweight_sigma_f+} to each corresponding periodization occuring in the definition of  $\mathbf{Z}_f$, we obtain that
the weighted dynamical zeta function associated with $f$ agrees with the one associated with the future-dependent observable $f_+$:
$$\mathbf{Z}_f(s)=\mathbf{Z}_{f_+}(s),$$ 
where $f_+\in C^{\mathrm{lc}}(\mathfrak{P}_\Gamma^+)$ is as in Lemma~\ref{lem:one-sided-reduction}.

Using Lemma~\ref{lem:trace_unweighted_multiplicative}, we relate \eqref{eq:Zeta} to the unweighted resolvent \eqref{eq:resolvent} through the \emph{weighted trace formula}.

\begin{prop}[Weighted trace formula]
\label{prop:weighted_trace}
    Fix $\vartheta \in (0,1)$ and let $\omega \equiv 1.$
    Consider the unweighted resolvent 
    $
    \mathbf{R}(s) = (\mathrm{Id}-q^{\frac12+is} \mathcal{L}_{\Gamma, +})^{-1}$ on $\F_{\vartheta}(\mathfrak{P}^+_\Gamma).
    $
   Let $f\in C^{\mathrm{lc}}(\mathfrak{P}_\Gamma)$, and let $f_+\in C^{\mathrm{lc}}(\mathfrak{P}_\Gamma^+)\subset\mathcal{F}_\vartheta(\mathfrak{P}_\Gamma^+)$ be its one-sided reduction as in Lemma~\ref{lem:one-sided-reduction}, viewed as a bounded multiplication operator $M_{f_+}$ on $\mathcal{F}_\vartheta(\mathfrak{P}_\Gamma^+)$. 
   
   Then, for $\mathrm{Im}(s)> \frac{3}{2}$, the weighted trace formula holds
    \begin{equation} \label{eq:trace_formula}
        \operatorname{Tr}\bigl[\mathbf{R}(s)f_+\bigr] \coloneqq \operatorname{Tr}\bigl[M_{f_+}\mathbf{R}(s)\bigr] = \mathbf{Z}_f(s). %\quad \text{for }\;  \mathrm{Im}(s)>\frac{3}{2}.
    \end{equation}
\end{prop}

Note that the weighted trace formula \eqref{eq:trace_formula} is the graph-analogue of the weighted Atiyah--Bott--Guillemin trace formula for open systems in \cite[Lem.~3.1]{SchutteWeichBarkhofen23_weightedzeta}.

\begin{proof}
    Let $m\ge1$ be large enough that both $f$ %(depth on $\mathfrak{P}_\Gamma$) 
    and $f_+$ %(depth on $\mathfrak{P}_\Gamma^+$) 
    have depth at most $m$.

    By Proposition~\ref{prop:L_bound} applied with $\omega \equiv 1$ %(this is is \cite[Cor.~5.8]{BHW23}), 
    $\mathcal{L}_{\Gamma,+}$ acts as a bounded endomorphism of $\F_\vartheta(\mathfrak{P}_\Gamma^+)$ with operator norm satisfying $\|\mathcal{L}_{\Gamma,+}^n\|_{\mathrm{op}, \infty} \le q^n$ (note that it does not depend on $\vartheta \in (0,1)$).
    Therefore $\|q^{n(\frac12+is)}\mathcal{L}_{\Gamma,+}^n\|_{\mathrm{op}, \infty} \le |q^{n(\frac12+is)}|q^n =(|q^{(\frac12+is)}|q)^n$.
    Consequently, the Neumann series
    \[
    \mathbf{R}(s) 
    = (\mathrm{Id}-q^{\frac12+is} \mathcal{L}_{\Gamma, +})^{-1}
    =\sum_{n=0}^\infty \bigl(q^{\frac12+is}\mathcal{L}_{\Gamma,+}\bigr)^n
    \]
    converges in operator norm provided $|q^{(\frac12+is)}|q < 1$.
    That is, $|q^{\frac12+is}|<q^{-1}$.
    Since $|q^{\frac12 +is}|=q^{\frac12-\mathrm{Im}(s)}$,
    this is equivalent to $\mathrm{Im}(s)> \tfrac{3}{2}.$

    Since $f_+$ has finite depth $m$, the operator $M_{f_+}\mathcal{L}_{\Gamma,+}^n$ can be restricted to the finite-dimensional subspace $F_m(\mathfrak{P}_\Gamma^+)$. 
    This finite-dimensional restriction allows us to compute the trace of $M_{f_+}\mathcal{L}_{\Gamma,+}^n$ at depth $m$.
    By the above convergence criterion, we may exchange trace and summation and thus, using Lemma~\ref{lem:trace_unweighted_multiplicative}, we get
    \begin{eqnarray*}
        \operatorname{Tr}\bigl[\mathbf{R}(s)f_+\bigr] 
        = \sum_{n= 0}^\infty q^{n(\frac12+is)}\operatorname{Tr}\bigl[M_{f_+}\mathcal{L}_{\Gamma,+}^n\bigr] 
        &=& \sum_{n\ge 0} q^{n(\frac12+is)}\operatorname{Tr}[\mathcal{L}_{1,+,F_m}^n f_+] \\
        &=&\sum_{n\ge 0} q^{n(\frac12+is)}\operatorname{Tr}[\mathcal{L}^n_{1,F_m}f] \\
        &=& \sum_{n\ge 0} q^{n(\frac12+is)}\sum_{\substack{[\mathbf{p}_0]\in\mathscr{P}_0\\ \ell_0\mid n}} \mathbf{f}_{\ell_0}(\widehat{\mathbf{p}}_0).
    \end{eqnarray*}

    Since $\mathfrak{G}_\Gamma$ is finite $(q+1)$-regular graph with vertex set $\mathfrak{X}_\Gamma$, by Lemma~\ref{lem:nb_Pm}, the number of non-backtracking closed paths $\mathcal{P}_n^{\mathrm{cl}} \subseteq \mathcal{P}_n$ of length $n$ is at most $|\mathfrak{X}_\Gamma|(q+1)q^{n-1}$.
    So with $\|f_+\|_\infty < \infty$, we have
    \begin{equation} \label{eq:trace_conv}
        \sum_{n\geq 0} |q^{n(\frac12 +is)}| |\operatorname{Tr}\bigl[\mathcal{L}^n_{\Gamma,\pm} f_+\bigr]\,| 
        \leq \sum_{n\geq 0} |q^{n(\frac12 +is)}| |\mathcal{P}_n^{\mathrm{cl}}| \, \|f_+\|_\infty 
        \leq C\sum_{n \geq 0} (q|q^{\frac12 +is}|)^n
    \end{equation}
    for a constant depending on $\mathfrak{X}_\Gamma$.
    Hence this trace estimate yields absolute convergence whenever $|q^{\frac12 +is}|< q^{-1}$, which is the same domain as above.\\
    Therefore, this 
    permits exchanging the order of summation over $n$ and $[\mathbf{p}_0]$ via $n=k\ell_0$, $k\ge1$:
    \begin{eqnarray*}
        \operatorname{Tr}[\mathbf{R}(s)f_+] 
        = \sum_{n\ge 0} q^{n(\frac12+is)} \sum_{\substack{[\mathbf{p}_0]\in\mathscr{P}_0\\ \ell_0\mid n}} \mathbf{f}_{\ell_0}(\widehat{\mathbf{p}}_0) 
        &=& \sum_{[\mathbf{p}_0]\in\mathscr{P}_0} \mathbf{f}_{\ell_0}(\widehat{\mathbf{p}}_0) \sum_{k\ge1} q^{k\ell_0(\frac12+is)} \\
        &=& \sum_{[\mathbf{p}_0]\in\mathscr{P}_0} \frac{q^{\ell_0(\frac12+is)}}{1-q^{\ell_0(\frac12+is)}}\,\mathbf{f}_{\ell_0}(\widehat{\mathbf{p}}_0) \\
        &=& \mathbf{Z}_f(s),
    \end{eqnarray*}
    by Definition~\ref{def:zeta_weighted}.
    Hence \eqref{eq:trace_conv} proves absolute convergence of $\mathbf{Z}_f(s)$ for $\mathrm{Im}(s) > \tfrac32.$
\end{proof}

% -------------------------------%%--------------------------------------%
\subsection{Meromorphic continuation of $\mathbf{Z}_f$ (Theorem~\ref{thm:mero_dynfct})} \label{sect:mero_cont}

Having established all the necessary tools, we are now in a position to prove Theorem~\ref{thm:mero_dynfct}.

% -------------------------------%%--------------------------------------%
\subsubsection{Proof of Theorem~\ref{thm:mero_dynfct}} \label{sect:proof_merom}
Throughout the proof, fix $f\in C^{\mathrm{lc}}(\mathfrak{P}_\Gamma)$ and let $m\geq1$ be large enough that $f$ has depth at most $m$. 
By Lemma~\ref{lem:one-sided-reduction}, $f$ admits a one-sided reduction $f_+\in C^{\mathrm{lc}}(\mathfrak{P}_\Gamma^+)$, of depth at most $m$, satisfying $f\circ\sigma^m=f_+\circ\pi_+$.
As noted in the discussion directly Definition~\ref{def:zeta_weighted}, via \eqref{eq:fnweight_sigma_f+}, this yields
$$
    \mathbf{Z}_f(s)=\mathbf{Z}_{f_+}(s), \qquad s\in\C,
$$
so that $\mathbf{Z}_f$ and $\mathbf{Z}_{f_+}$ are literally the same function, and it suffices to prove all three assertions for $\mathbf{Z}_{f_+}(s)=\operatorname{Tr}[\mathbf{R}(s)f_+]$ using the one-sided weight $f_+$, viewed as the bounded finite-rank multiplication operator $M_{f_+}$ on $\F_\vartheta(\mathfrak{P}_\Gamma^+)$ as in Proposition~\ref{prop:weighted_trace}.
We write $f_+$ (rather than $f$) throughout the proof accordingly, and revert to $f$ only in the final statements to match the notation of the theorem.\\

For $\mathrm{(i)}$, let us first show that $\mathbf{Z}_f(s)$ is holomorphic on the half-plane $\mathcal{H}_s \coloneqq \{s\in\C\mid \mathrm{Im}(s)>\tfrac32\}$ by proving that the defining series \eqref{eq:Zeta} converges locally uniformly there.
Using Proposition~\ref{prop:weighted_trace} and \eqref{eq:trace_unweighted_multiplicative}, it suffices to show local uniform convergence of
\begin{equation} \label{eq:serie_trace}
     \sum_{n\ge0} q^{n(\frac12+is)}\operatorname{Tr}\bigl[\mathcal{L}_{\Gamma,+}^n f_+\bigr].
\end{equation}
Fix a compact set $K\subset \mathcal{H}_s$ and choose $\varepsilon>0$ with $\mathrm{Im}(s)\ge\tfrac32+\varepsilon$ for all $s\in K$.
Then
$$
    |q^{n(\frac12+is)}|=q^{n(\frac12-\mathrm{Im}(s))}
    \le q^{-n(1+\varepsilon)},\qquad s\in K.
$$
Combined with the bound $|\mathcal P_n^{\mathrm{cl}}|\le|\mathfrak X_\Gamma|(q+1)q^{n-1}$ from \eqref{eq:trace_conv}, this gives, uniformly for $s\in K$,
$$
    \bigl|q^{n(\frac12+is)}\operatorname{Tr}[\mathcal{L}_{\Gamma,+}^nf]\bigr|
    \le |\mathfrak X_\Gamma|(q+1)q^{n-1}\,q^{-n(1+\varepsilon)}\,\|f_+\|_\infty
    \le C\,q^{-n\varepsilon},
$$
for a constant $C=|\mathfrak X_\Gamma|(q+1)q^{-1}\|f\|_\infty$ independent of $s\in K$ and $n$. \\
Since $\sum_{n\geq 0} q^{-n\varepsilon}<\infty$, the series \eqref{eq:serie_trace} converges uniformly on $K$.
As each term $s\mapsto q^{n(\frac12+is)}\operatorname{Tr}[\mathcal{L}_{\Gamma,+}^nf]$ is entire in $s$, the uniform limit $\mathbf{Z}_f(s)$ is holomorphic on $\mathcal{H}_s$.\\

Next, we prove $\mathrm{(ii)}$.  
By Theorem~\ref{thm:mero_resolvent}, applied to the unweighted transfer operator $\mathcal{L}_{\Gamma,+}$ (which is non-zero since $\mathfrak{G}_\Gamma$ is connected with $q>1$), the resolvent
$
    \mathbf{R}(s)=(\mathrm{Id}-q^{\frac12+is}\mathcal{L}_{\Gamma,+})^{-1}
$
extends, as a family of bounded operators on $\F_\vartheta(\mathfrak{P}_\Gamma^+)$, meromorphically from $\mathcal{H}_s$ to the half-plane $\mathscr{H}_{\vartheta,1}$ \eqref{eq:region}.
Moreover, every pole $s_0\in\mathscr{H}_{\vartheta,1}$ of $\mathbf{R}(s)$ has finite-rank principal part. 
Hence, by Lemma~\ref{lem:resolvent_decomp_hol_finiteranks} (applied with $\omega\equiv1$), near such a pole $s_0$ we have the finite-rank Laurent expansion
\begin{equation}\label{eq:laurent_R}
    \mathbf{R}(s)
    =\mathbf{R}_H(s)+\sum_{j=1}^{J(s_0)}\frac{B_j}{(\lambda(s)-\lambda(s_0))^j},
    \qquad B_j=\mathcal{L}_{\Gamma,+}\bigl(\mathcal{L}_{\Gamma,+}-\lambda(s_0)\bigr)^{j-1}\Pi_{s_0},
\end{equation}
with $\mathbf{R}_H(s)$ holomorphic near $s_0$ and each $B_j$ finite rank, where $\lambda(s)\coloneqq q^{-(\frac12+is)}$ is the isolated Ruelle resonance parameter attached to $s_0$.
Since $f_+$ has finite depth $m$, the multiplication operator $M_{f_+}$ has finite rank: its range lies in the finite-dimensional space $F_m(\mathfrak{P}_\Gamma^+)$ for $m$ large enough.\\ 
Hence, for any bounded operator $A$ on $\F_\vartheta(\mathfrak{P}_\Gamma^+)$,
the linear functional $A\mapsto \operatorname{Tr}[Af_+]$ is bounded, hence holomorphic.
Applying this term-by-term to \ref{prop:weighted_trace} by using Proposition~\ref{prop:weighted_trace}, gives
\begin{equation} \label{eq:zeta_laurent}
    \mathbf{Z}_f(s) 
    = \operatorname{Tr}[\mathbf{R}(s)f_+]
    = \mathbf{Z}_{f,H}(s) + \sum_{j=1}^{J(s_0)}\frac{\operatorname{Tr}[B_jf_+]}{(\lambda(s)-\lambda(s_0))^j},
    \qquad \mathbf{Z}_{f,H}(s)\coloneqq\operatorname{Tr}[\mathbf{R}_H(s)f_+],
\end{equation}
which exhibits $\mathbf{Z}_f(s)$ as meromorphic near $s_0$, with pole of order at most $J(s_0)$. \\
As $s_0\in\mathscr{H}_{\vartheta,1}$ was an arbitrary pole of $\mathbf{R}(s)$, and since $\mathbf{R}(s)$ is holomorphic (as a bounded operator family) away from its poles, we obtain a meromorphic continuation of $s\mapsto\operatorname{Tr}[\mathbf{R}(s)f_+]$ to all of $\mathscr{H}_{\vartheta,1}$.

By Proposition~\ref{prop:weighted_trace}, this continuation agrees with $\mathbf{Z}_f(s)$ on the open half-plane $\mathcal{H}_s\subset\mathscr{H}_{\vartheta,1}$. 
Since $\mathscr{H}_{\vartheta,1}$ is connected and $\mathcal{H}_s$ is a nonempty open subset, the identity theorem for meromorphic functions shows that this continuation is \emph{the} meromorphic continuation of $\mathbf{Z}_f(s)$ to $\mathscr{H}_{\vartheta,1}$. 
This proves the first assertion of $\mathrm{(ii)}$.

For the second assertion, note that by construction every pole of $\mathbf{Z}_f(s)=\operatorname{Tr}[\mathbf{R}(s)f_+]$ in $\mathscr{H}_{\vartheta,1}$ must occur at a pole $s_0$ of the %operator-valued resolvent 
$\mathbf{R}(s)$ itself, since $\mathbf{R}(s)$ is holomorphic and bounded away from such points, forcing $\operatorname{Tr}[\mathbf{R}(s)f_+]$ to be holomorphic there as well. 
By Theorem~\ref{thm:mero_resolvent} together with Lemma~\ref{lem:resolvent_decomp_hol_finiteranks}, the poles of $\mathbf{R}(s)$ occur precisely at those $s_0\in\mathscr{H}_{\vartheta,1}$ for which $\lambda(s_0)=q^{-(\frac12+is_0)}$ is an isolated Ruelle resonance
of $\mathcal{L}_{\Gamma,+}$ on $\F_\vartheta(\mathfrak{P}_\Gamma^+)$.\\

To complete the proof, it remains to show $\mathrm{(iii)}$, for which we derive an explicit formula for the Laurent coefficients near $s_0 \in \mathscr{H}_{\vartheta,1}$. \\
Fix a pole $s_0\in\mathscr{H}_{\vartheta,1}$ of $\mathbf{R}(s)$ and
set $\lambda(s_0)=q^{-(\frac12+is_0)}\neq 0$.
Let $\Pi_{s_0}$ be the Riesz projector of $\mathcal{L}_{\Gamma,+}$ at $\lambda(s_0)$ as in Lemma~\ref{lem:resolvent_decomp_hol_finiteranks}.

On $\operatorname{Ran}(\Pi_{s_0})$, we write 
$$
    \mathcal{L}_{\Gamma,+}
    =
    \lambda_0\mathrm{Id}+N_{s_0},
    \qquad
    N_{s_0}\coloneqq
    (\mathcal{L}_{\Gamma,+}-\lambda(s_0))\Pi_{s_0}.
$$
The operator $N_{s_0}$ is nilpotent, and by definition of the Jordan order $J(s_0)$, one has
$
    N_{s_0}^{J(s_0)}
    =
    (\mathcal{L}_{\Gamma,+}-\lambda(s_0))^{J(s_0)}\Pi_{s_0}
    =
    0.
$
Hence multiplying the Laurent expansion \eqref{eq:zeta_laurent} by $(\lambda(s)-\lambda(s_0))^k$, for $0 \leq k \leq J(s_0)$, gives
\begin{eqnarray*}
    \mathbf{Z}_f(s)(\lambda(s)-\lambda(s_0))^k
    &=& \mathbf{Z}_{f,H}(s)(\lambda(s)-\lambda(s_0))^k \\
    && \qquad + \lambda(s) \sum_{j=0}^{J(s_0)-1}(\lambda(s)-\lambda(s_0))^{k-j-1}\operatorname{Tr}[(\mathcal{L}_{\Gamma,+}-\lambda(s_0))^j \Pi_{s_0}f_+].
\end{eqnarray*}
Since $\mathbf{Z}_{f,H}(s)$ is holomorphic near $s_0$, the first term is holomorphic near $s_0$ and hence does not contribute to the residue.

Using the local coordinate $\lambda=\lambda(s)$, by \eqref{eq:lambda_prime} and the holomorphic inverse function theorem, $\lambda$ is a valid local holomorphic coordinate near $s_0$.
Moreover,
$
    \frac{\mathrm{d}s}{\mathrm{d}\lambda}
    = \frac{i}{\log q}\lambda^{-1}.
$
Therefore, by the change-of-variables formula for residues, and since the factor $\lambda$ cancels exactly against $\lambda^{-1}$ coming from the Jacobian, we get
\begin{eqnarray*}
    &&\operatorname{Res}_{s=s_0}
    \bigl[\mathbf{Z}_f(s)(\lambda(s)-\lambda(s_0))^k\bigr] \\
    &=& \operatorname{Res}_{\lambda=\lambda(s_0)} 
    \left[\lambda \sum_{j=0}^{J(s_0)-1}(\lambda-\lambda(s_0))^{k-j-1} \operatorname{Tr}\big[(\mathcal{L}_{\Gamma,+}-\lambda(s_0))^j \Pi_{s_0}f_+\big] \, \frac{i}{\log q} \frac{\mathrm{d}\lambda}{\lambda}\right] \\
    &=& \frac{i}{\log q}
    \operatorname{Res}_{\lambda=\lambda(s_0)}
    \left[\sum_{j=0}^{J(s_0)-1}(\lambda-\lambda(s_0))^{k-j-1}\operatorname{Tr}[(\mathcal{L}_{\Gamma,+}-\lambda(s_0))^j \Pi_{s_0}f_+]\right].
\end{eqnarray*}
The expression inside the last residue is a finite Laurent polynomial in $(\lambda-\lambda(s_0))$.
Hence, the residue at $\lambda=\lambda(s_0)$ is the coefficient $(\lambda-\lambda(s_0))^{-1}.$
Therefore, its residue picks out exactly the term for which $k-j-1=-1,$ that is
$j=k$.

If $0\le k\le J(s_0)-1$, the term $j=k$ occurs in the sum, and thus we obtain
$$
    \operatorname{Res}_{s=s_0}
    \bigl[\mathbf{Z}_f(s)(\lambda(s)-\lambda_0)^k\bigr]
    =\frac{i}{\log q}\operatorname{Tr}\bigl[(\mathcal{L}_{\Gamma,+}-\lambda(s_0))^k\Pi_{s_0}f_+\bigr].
$$

It remains to consider the endpoint $k=J(s_0)$.
In this case, there is no term with $j=k$ in the sum, since the sum only runs over $0\le k\le J(s_0)-1$.
So the residue is $0$. 
On the other hand, by nilpotency, we have $(\mathcal{L}_{\Gamma,+}-\lambda_0)^{J(s_0)}\Pi_{s_0}=0,$
and therefore
$$
    \operatorname{Tr}
    \bigl[
        (\mathcal{L}_{\Gamma,+}-\lambda_0)^{J(s_0)}\Pi_{s_0}f_+
    \bigr]=0.
$$
Thus the same formula also holds for $k=J(s_0)$.

Consequently, for every $0\le k\le J(s_0)$ we obtain, since $\operatorname{Tr}[(\mathcal{L}_{\Gamma,+}-\lambda_0)^k\Pi_{s_0}f_+]=\operatorname{Tr}[(\mathcal{L}_{\Gamma,+}-\lambda_0)^k\Pi_{s_0}f]$ by the same finite-rank pairing as in $\mathrm{(ii)}$, exactly \eqref{eq:residue_R}, completing the proof of $\mathrm{(iii)}$ and of the theorem. \qed

% -------------------------------%%--------------------------------------%
\subsection{Dynamical determinants} \label{sect:dyn_determinants}
Before introducing the dynamical determinants and their relations, we first define the weighted zeta function associated with a parameter $\beta \in \C$ and a weight function $f\in C^{\mathrm{lc}}(\mathfrak{P}_\Gamma)$.

\begin{definition}[$\beta$-parametrized weighted Ihara zeta function]
\label{def:zetabeta}
    Let $\mathfrak{G}_\Gamma$ be a finite $(q+1)$-regular graph with $q>1$ and $f\in C^{\mathrm{lc}}(\mathfrak{P}_\Gamma)$ with its associated periodic weight $\mathbf{f}_{\ell_0}$ as in \eqref{eq:weight_f}.\\
    For a parameter $\beta\in\mathbb{C}$, the \emph{$\beta$-parametrized weighted Ihara zeta function} associated with $f$ is defined by
    \begin{equation} \label{eq:Ihara}
        \zeta(s,\beta) 
        \coloneqq \prod_{[\mathbf{p}_0]\in\mathscr{P}_0} \Big(1-q^{(\frac12+is)\ell_0} e^{-\beta \mathbf{f}_{\ell_0}(\widehat{\mathbf{p}}_0)} \Big)^{-1}, \quad s \in \mathbb{C},
    \end{equation}
    where product runs over all primitive cycles $[\mathbf{p}_0]\in\mathscr{P}_0$ with $\ell_0=\ell(\mathbf{p}_0)$, and $\widehat{\mathbf{p}}_0 \in \operatorname{Fix}(\sigma^{\ell_0}) \subset \mathfrak{P}_\Gamma$ denotes the periodization of any representative $\mathbf{p}_0$ of $[\mathbf{p}_0]$.
\end{definition}

Note that for $\beta=0$, this reduces to the classical \emph{Ihara zeta function} $\zeta(s)=\zeta(s,0)$, see, e.g., \cite[Def.~2]{Terras10}.

Now taking the logarithm of 
$\zeta(s, \beta)^{-1}$, and using the  the absolutely convergent expansion $\log(1-q^{\frac12+is})=-\sum_{k=1}^{\infty}\frac{q^{k(\frac12+is)}}{k},$ $ |q^{\frac12+is}|<1$, we obtain
\begin{eqnarray}     \label{eq:log_weighted_zeta}
     \log \zeta(s,\beta)^{-1}
    =
    \sum_{[\mathbf{p}_0]\in\mathscr{P}_0}
    \log\left(
        1-q^{(\frac12+is)\ell_0}
        e^{-\beta
        \mathbf{f}_{\ell_0}(\widehat{\mathbf{p}}_0)}
    \right)
    \notag
    &=&
    -\sum_{[\mathbf{p}_0]\in\mathscr{P}_0}
      \sum_{k=1}^{\infty}
      \frac{q^{(\frac12+is)k\ell_0}}{k}
      e^{-k\beta
      \mathbf{f}_{\ell_0}(\widehat{\mathbf{p}}_0)}.
    \notag\\
      &&
\end{eqnarray}
Exponentiating
\eqref{eq:log_weighted_zeta} motivates the following definition.

\begin{definition}[Dynamical determinant of weight $f$] \label{def:dynamical_det_f}
    Let $\mathfrak{G}_\Gamma$ 
    be a finite  $(q+1)$-regular graph with $q>1$.
    For every primitive cycle
    \([\mathbf{p}_0]\in\mathscr{P}_0\), let
    \(\ell_0=\ell(\mathbf{p}_0)\), and let
    \(\widehat{\mathbf{p}}_0\in
    \operatorname{Fix}(\sigma^{\ell_0})\)
    be the periodisation of a representative of
    $[\mathbf{p}_0]$.
    Consider
    $f \in C^{\mathrm{lc}}(\mathfrak{P}_\Gamma)$ with its associated periodic-cycle weight $\mathbf{f}_{\ell_0}$ as in \eqref{eq:weight_f}.\\
    On the domain of absolute convergence, we define the associated \emph{dynamical determinant} at $(s, \beta) \in \C^2$ as
    \begin{eqnarray*}
        d_f(s,\beta)
        &\coloneqq & 
        \exp\Biggl(- \sum_{k=1}^\infty \sum_{[\mathbf{p}_0]\in\mathscr{P}_0} \frac{q^{(\frac12+is)k\ell_0}}{k} e^{-k\beta \mathbf{f}_{\ell_0}(\widehat{\mathbf{p}}_0)}\Biggr), 
    \end{eqnarray*}
    where the sum stretches over all primitive cycles $[\mathbf{p}_0]\in\mathscr{P}_0$.
\end{definition}

A useful property is to relate the dynamical determinant to the determinant on $F_m(\mathfrak{P}_\Gamma^+)$.

\begin{prop}[Dynamical determinant as a finite-dimensional determinant]
\label{prop:dyn_det_Fredholm}
    Let $f \in C^{\mathrm{lc}}(\mathfrak{P}_\Gamma)$ with associated periodic-cycle weight $\mathbf{f}_{\ell(\mathbf{p})}$ as in \eqref{eq:weight_f} and let $f_+\in C^{\mathrm{lc}}(\mathfrak{P}^+_\Gamma)$ be its one-sided reduction as in Lemma~\ref{lem:one-sided-reduction}, having depth at most $m$, for $m$ sufficiently large.
    For $\beta \in \C$, define the parameter-dependent potential
    $$\omega_{\beta}(\mathbf{p}) 
    \coloneqq
    e^{-\beta \, f_+(\mathbf{p})}, \quad \mathbf{p} \in \mathfrak{P}_\Gamma^+,$$
    and let $\mathcal{L}_{\beta,F_m} \coloneqq \mathcal{L}_{\Gamma,+,\omega_{\beta}}|_{F_m(\mathfrak{P}^+_\Gamma)}$ be the corresponding weighted Ruelle transfer operator restricted to the finite-dimensional space $F_m(\mathfrak{P}^+_\Gamma)$. 
    
    Then, for $\mathrm{Im}(s)> \tfrac{3}{2} + \tfrac{|\beta| \; \|f_+\|_\infty}{\log q}$,
    one has
    $$
    d_f(s,\beta)=
    \mathrm{det}\bigl(\mathrm{Id}-q^{\frac12+is}\mathcal{L}_{\beta, F_m}\bigr),
    $$
    where the determinant is the ordinary finite-dimensional determinant.
    
    In particular, the dynamical determinant $d_f(s,\beta)$ continues to a holomorphic function in $(s,\beta) \in \C^2$.
\end{prop}

\begin{proof}
    Fix $m$ sufficiently large such that $f_+$ has finite memory $m$.
    By Lemma~\ref{lem:invariant_F},
    $$
        \mathcal{L}_{\Gamma,+, \omega_\beta}
        \bigl(F_m(\mathfrak{P}_\Gamma^+)\bigr)
        \subseteq
        F_{m-1}(\mathfrak{P}_\Gamma^+)
        \subseteq
        F_m(\mathfrak{P}_\Gamma^+).
    $$
    Thus $F_m(\mathfrak{P}_\Gamma^+)$ is invariant and
    $\mathcal{L}_{\beta,m} \coloneqq  \mathcal{L}_{\Gamma,+, \omega_\beta}|_{F_m(\mathfrak{P}_\Gamma^+)}$
    is an endomorphism of the finite-dimensional vector space $F_m(\mathfrak{P}_\Gamma^+)$.
    It is therefore a finite-rank operator.
    
    By iterating the weighted Ruelle transfer operator
    $\mathcal{L}_{\beta, F_m}$, gives 
    $$
    (\mathcal{L}_{\beta, F_m}^n\varphi)(\mathbf{p}_+)
    = \sum_{\sigma_+^n(\mathbf{p}'_+)=\mathbf{p}_+}
    \exp\Bigl(-\beta\sum_{j=0}^{n-1}
       f_+(\sigma_+^j(\mathbf{p}'_+)) \Bigr)\varphi(\mathbf{p}'_+), \qquad \mathbf{p}_+\in \mathfrak{P}_\Gamma^+, \varphi \in F_m(\mathfrak{P}_\Gamma^+).
    $$
    Let $\widehat{\mathbf{p}}\in\operatorname{Fix} (\sigma^n) \subset \mathfrak{P}_\Gamma$. Using the one-side reduction \eqref{eq:fnweight_sigma_f+}, each such periodic paths contributes
    $
    \exp\bigl(-\beta\sum_{j=0}^{n-1} f_+(\sigma_+^j(\pi_+(\widehat{\mathbf{p}}))
    \bigr)
    =
    e^{-\beta\mathbf{f}_{n,+}(\pi_+(\widehat{\mathbf{p}}))}
    =
    e^{-\beta\mathbf{f}_n(\widehat{\mathbf{p}})}
    $
    by definition of \eqref{eq:weight_f}.
    Therefore, the periodic-path trace formula in Proposition~\ref{prop:trace_Lnw} gives
    \begin{equation} \label{eq:trace_Ln_beta}
         \operatorname{Tr}(\mathcal{L}_{\beta,F_m}^n)
        = \sum_{\substack{[\mathbf{p}_0]\in\mathscr{P}_0\\
        \ell_0 \,\mid\, n}}
        \ell_0\,
        e^{-n/\ell_0 \, \beta\mathbf{f}_n(\widehat{\mathbf{p}})}.
    \end{equation}
    Whenever $\lVert q^{\frac12+is} \mathcal{L}_{\beta,F_m}\rVert_{\mathrm{op}, \infty} <1$, the usual
    finite-dimensional determinant identity gives
    \begin{eqnarray*}
        \log\big(\det(\operatorname{Id}-q^{\frac12+is} \mathcal{L}_{\beta,F_m})\bigr)
        &=&
        -\sum_{n=1}^{\infty}
        \frac{q^{n(\frac12+is)}}{n}
        \operatorname{Tr}(\mathcal{L}_{\beta,F_m}^n) \\
        &\stackrel{\eqref{eq:trace_Ln_beta}}{=}& 
        -\sum_{n=1}^{\infty} \frac{q^{n(\frac12+is)}}{n}
        \sum_{\substack{[\mathbf{p}_0]\in\mathscr{P}_0\\
        \ell_0 \,\mid\, n}}
        \ell_0\,
        e^{-n/\ell_0 \, \beta\mathbf{f}_n(\widehat{\mathbf{p}})} \\
        &=& -\sum_{k=1}^{\infty} \frac{q^{k\ell_0(\frac12+is)}}{k}
        \sum_{[\mathbf{p}_0]\in\mathscr{P}_0}
        e^{-k \, \beta\mathbf{f}_{\ell_0}(\widehat{\mathbf{p}})},
    \end{eqnarray*}
    where in the last equalities we used \eqref{eq:trace_Ln_beta} and set $n=k\ell_0,$ $k \geq 1.$
    By Definition~\ref{def:dynamical_det_f} of the dynamical determinant, the right-hand side is
    $\log(d_f(s,\beta))$. 
    Hence
    $
        d_f(s,\beta)
        =\det(\operatorname{Id}-q^{\frac12+is}\mathcal{L}_{\beta,F_m})
    $
    whenever the defining series converges.

    Moreover, by Proposition~\ref{prop:L_bound}, $\mathcal{L}_{\beta,F_m}$ acts as a bounded endomorphism of $F_m(\mathfrak{P}_\Gamma^+)$ with operator norm satisfying $\|\mathcal{L}^n_{\beta,F_m}\|_{\mathrm{op},\infty} \leq q^n \|\omega_{\beta, m}\|^n$.
    Therefore 
    $$\lVert q^{\frac12+is} \mathcal{L}_{\beta,F_m}\rVert_{\mathrm{op}, \infty} <|q^{n(\frac12+is)}| q^n \|\omega_{\beta,m}\|_\infty^n = (|q^{\frac12+is}| q \|\omega_{\beta,m}\|_\infty)^n.$$
    Consequently, the defining series converges in operator norm provided $|q^{\frac12+is}| <(q \|\omega_{\beta,m}\|_\infty)^{-1}$.
    Since $|q^{\frac12+is}| = q^{\frac12-\mathrm{Im}(s)}$
    and $\log(\|\omega_{\beta,m}\|_\infty)= |\beta| \|f_+\|_\infty$, this equivalent to 
    $$\mathrm{Im}(s) > \frac{3}{2} +\frac{|\beta| \|f_+\|_\infty}{\log q}.$$

    As for the holomorphic continuation, since the space $F^m(\mathfrak{P}_\Gamma^+)$ is finite dimensional, the entries of $\mathcal{L}_{\beta,F_m}$ are finite sums of functions of the form $e^{-\beta c},$ $c\in\C.$
    They are entire functions of $\beta$. 
    Since $s\longmapsto q^{\frac12+is}$ is entire, it follows that
    $$
    (s,\beta)\longmapsto \det\left(\operatorname{Id} -q^{\frac12+is}\mathcal{L}_{\beta,F_m}\right)
    $$
    is holomorphic on $\C^2$.
    By uniqueness of analytic continuation, this gives the claimed holomorphic continuation of $d_f$.
\end{proof}

Next, let us prove the connection between the weighted dynamical Ihara zeta function $ \mathbf{Z}_f(s)$ in \eqref{eq:Zeta} and the dynamical determinant $d_f$.

\begin{prop}[Weighted dynamical zeta function]
\label{prop:dyn_zeta_weight_zeta}
    Given a weight function $f \in C^{\mathrm{lc}}(\mathfrak{P}_\Gamma)$, the weighted dynamical Ihara zeta function $ \mathbf{Z}_f(s)$ at $s\in \C$ coincides with the logarithmic derivative of the dynamical determinant evaluated at $\beta=0$:
    \begin{equation*}
        \mathbf{Z}_f(s) = \frac{\partial_\beta d_f(s, 0)}{d_f(s, 0)}.
    \end{equation*}
\end{prop}

\begin{proof}
    By computing the logarithmic derivative of the dynamical determinant at $\beta=0$, we get
    \begin{align*}
        \frac{\partial_\beta d_f(s, \beta)}{d_f(s, \beta)}\Bigg|_{\beta=0}
        &= \partial_\beta \log (d_f(s, \beta))\big|_{\beta=0} \\
        &\stackrel{\eqref{eq:log_weighted_zeta}}{=} \partial_\beta \Big(-
     \sum_{[\mathbf{p}_0]\in\mathscr{P}_0} \sum_{k=1}^{\infty}
      \frac{q^{(\frac12+is)k\ell_0}}{k}
      e^{-k\beta
      \mathbf{f}_{\ell_0}(\widehat{\mathbf{p}}_0)}\Big)\Big|_{\beta=0} \\
        &=  \sum_{[\mathbf{p}_0]\in\mathscr{P}_0} \sum_{k=1}^\infty q^{(\frac12+is)k\ell_0} \mathbf{f}_{\ell_0}(\widehat{\mathbf{p}}_0)\\
        &= \sum_{[\mathbf{p}_0]\in\mathscr{P}_0} \frac{q^{(\frac12+is)\ell_0}}{1-q^{(\frac12+is)\ell_0}}  \mathbf{f}_{\ell_0}(\widehat{\mathbf{p}}_0)\\
        &=\mathbf{Z}_f(s),
    \end{align*}
    by definition \eqref{eq:Zeta} of $\mathbf{Z}_f$. 
    Finally, by Proposition~\ref{prop:weighted_trace} and Proposition~\ref{prop:dyn_det_Fredholm}, the function $\mathbf{Z}_f(s)$ coincides with the logarithmic derivative of $d_f$ evaluated at $\beta=0$ in the same domain of convergence $\mathrm{Im}(s)> \frac{3}{2}$.
\end{proof}

% -------------------------------%%--------------------------------------%
\subsection{Residue formula for Patterson-Sullivan distributions (Theorem~\ref{thm:residue_PS})} \label{sect:residue}

Adopting the notation of \cite[§3]{BHW23} (see \cite[Rem.~1.1]{AFH23} for various notions of Laplacian on graphs), we consider the \emph{(vertex) Laplace operator} $\Laplace$ on $\mathrm{Maps}(\mathfrak{X},\C)$ acting by averaging over neighbours:
$$
    (\Laplace \varphi)(x) 
    \coloneqq \frac{1}{q + 1} \sum_{\substack{y\in\mathfrak{X}\\d(x,y)=1}} \varphi(y), \qquad \forall x \in \mathfrak{X},\varphi \in \mathrm{Maps}(\mathfrak{X}, \mathbb{C}),
$$
where $d(\cdot, \cdot)$ denotes the graph distance on $\mathfrak{X}$.
For $s \in \C\backslash\{0\}$, consider the potential
$
    \chi(s) \coloneqq \frac{\sqrt{q}}{q + 1}\big(q^{is} + q^{-is}\big) \in \C
$
and define its associated (non-zero) \emph{eigenspace}
\begin{equation} \label{eq:Laplace_eigenspace}
    \mathcal{E}_{\chi(s)}(\Laplace; \mathrm{Maps}(\mathfrak{X}, \mathbb{C})) \coloneqq \left\{ \varphi \colon \mathfrak{X} \rightarrow \mathbb{C} \mid \forall \, x \in \mathfrak{X}\colon \Laplace \varphi(x) = \chi(s) \varphi(x) \right\}.
\end{equation}
Pulling back along $\pi_{\mathfrak{X}} \colon \mathfrak{X} \rightarrow \mathfrak{X}_\Gamma$ induces an isomorphism

$$
    \pi^{*}_{\mathfrak{X}} \colon 
    \mathcal{E}_{\chi(s)}(\Laplace_{\Gamma}; \mathrm{Maps}(\mathfrak{X}_\Gamma, \C)) 
    \xrightarrow{\ \sim\ } 
    \mathcal{E}_{\chi(s)}(\Laplace; \mathrm{Maps}(\mathfrak{X}, \C))^{\Gamma}, 
    \qquad \pi^{*}_{\mathfrak{X}}(\varphi)(x) = \varphi(\pi_{\mathfrak{X}}(x)),
$$
identifying the Laplacian $\Laplace_\Gamma$ on the finite quotient $\mathfrak{X}_\Gamma$ with the space of $\Gamma$-invariants in the corresponding eigenspace on the universal cover $\mathfrak{X}$, see \cite[Rem.~3.4]{AFH23}.

Applying the quantum--classical correspondence (\cite[Thms.~8.3 and 11.5]{BHW23}, for non-exceptional spectral parameters, and \cite[Thm.~4.4]{AFH23}, for all spectral parameters, via edge Laplacians),
we recast the dynamical \emph{Patterson-Sullivan distributions} of \cite[Thm.~4]{ArendsPalmirotta26} as follows.

\begin{definition}[Patterson\textendash Sullivan distributions, {\cite[Thm.~4]{ArendsPalmirotta26}}]
\mbox{}\\
Let $\phi \in \mathcal{E}_{\chi(s)}(\Laplace;\mathrm{Maps}(\mathfrak{X}, \C))^\Gamma$ be a $\Gamma$-invariant Laplace eigenfunction for a non-exceptional spectral parameter $s \in \C$. The \emph{dynamical (diagonal) Patterson\textendash Sullivan distribution} $\mathrm{PS}_{\phi,\phi} \in \mathcal{D}'(\mathfrak{P}_\Gamma)$ associated to $\phi$ is the distributional tensor product
\begin{equation} \label{eq:PS}
    \mathrm{PS}_{\phi,\phi} = u_{+,\phi} \otimes u_{-,\phi} \in \mathcal{D}'(\mathfrak{P}_\Gamma)
\end{equation}
of the \emph{resonant state} $u_{+,\phi} \in \mathcal{D}'(\mathfrak{P}^+_\Gamma)$ and the \emph{co-resonant state} $u_{-,\phi} \in \mathcal{D}'(\mathfrak{P}^-_\Gamma)$, which are the (non-zero) elements of $\mathcal{E}_s(\mathcal{L}'_{\Gamma,\pm}; \mathcal{D}'(\mathfrak{P}^\pm_\Gamma))$ associated to $\phi$ via the quantum--classical correspondence.
\end{definition}

% -------------------------------%%--------------------------------------%
\subsubsection{Invariant Ruelle distributions} \label{sect:inv_Ruelle_dist}
Consider a finite-rank operator 
$$A \colon C^{\mathrm{l c}}(\mathfrak{P}^+_\Gamma ) \longrightarrow \mathcal{D}'(\mathfrak{P}^-_\Gamma)
$$
with its corresponding kernel $k_A \in \mathcal{D}'(\mathfrak{P}^+_\Gamma \times \mathfrak{P}^-_\Gamma)$ under the identification of \cite[Lem.~4.1]{ArendsPalmirotta26}.

Since $\mathfrak{P}^\pm_\Gamma$ is compact, the algebraic tensor product $C^{\mathrm{lc}}(\mathfrak{P}_\Gamma^+) \otimes C^{\mathrm{lc}}(\mathfrak{P}_\Gamma^-)$ can be canonically identified with $C^{\mathrm{lc}}(\mathfrak{P}_\Gamma^+ \times \mathfrak{P}_\Gamma^-)$, see \cite[§1.3]{AFH23Pairing}.
This identification is what allows two distributions $u_\pm \in \mathcal{D}'(\mathfrak{P}^\pm_\Gamma)$ to be paired with a single test function on the product space via their tensor product
$u_+ \otimes u_- \in \mathcal{D}'(\mathfrak{P}^+_\Gamma \times \mathfrak{P}^-_\Gamma)$, defined as in \cite[Eq.~(3)]{AFH23Pairing} by
\begin{equation} \label{eq:tensor_product_def}
    (u_+ \otimes u_-)(\varphi) \coloneqq \langle u_+, \mathbf{p}_+ \mapsto \bigl\langle u_-, \varphi(\mathbf{p}_+, \cdot)\rangle \bigr\rangle, 
    \quad \mathbf{p}_+\in \mathfrak{P}_\Gamma^+, \varphi\in C^{\mathrm{lc}}(\mathfrak{P}_\Gamma^+\times\mathfrak{P}_\Gamma^-).
\end{equation}
As noted in \cite[§1.3.]{AFH23Pairing}, $\mathfrak{P}_\Gamma$ is both open and closed in $\mathfrak{P}_\Gamma^+ \times \mathfrak{P}_\Gamma^-$ (being cut out by district conditions on the finitely many initial edges), so its indicator function $\mathbf{1}_{\mathfrak{P}_\Gamma}$ lies in $C^{\mathrm{lc}}(\mathfrak{P}_\Gamma^+ \times \mathfrak{P}_\Gamma^-)$.
So the tensor product $u_+ \otimes u_-$ can be evaluated on indicator functions of open and closed subsets of $\mathfrak{P}_\Gamma^+ \times \mathfrak{P}_\Gamma^-$.

More generally, extension by zero identifies $C^{\mathrm{lc}}(\mathfrak{P}_\Gamma)$ with the subspace of $C^{\mathrm{lc}}(\mathfrak{P}_\Gamma^+ \times \mathfrak{P}_\Gamma^-)$ consisting of functions supported on $\mathfrak{P}_\Gamma$. Under this identification $u_+ \otimes u_-$ may equally be paired against test functions on $\mathfrak{P}_\Gamma$, i.e., regarded as an element of $\mathcal{D}'(\mathfrak{P}_\Gamma)$.

With this convention in hand, the trace of $A$ \cite[Def.~4.2]{ArendsPalmirotta26}
is defined as
  \begin{equation} \label{eq:trace_finiterank_op}
    \mathrm{Tr}(A) \coloneqq \langle k_A , \mathbbm{1}_{\mathfrak{P}_\Gamma} \rangle,
  \end{equation}
the pairing of $k_A$ against the constant function $\mathbbm{1}_{\mathfrak{P}_\Gamma}$ on $\mathfrak{P}_\Gamma$, viewed as above as an element of $C^{\mathrm{lc}}(\mathfrak{P}_\Gamma^+ \times \mathfrak{P}_\Gamma^-)$.
More generally, for $f\in C^{\mathrm{lc}}(\mathfrak{P}_\Gamma)$, we write $f\cdot A$ for the operator with kernel $f\cdot k_A$, in the notation of
\cite[Lem.~4.1]{ArendsPalmirotta26}.

By \cite[Rem.~4.1]{ArendsPalmirotta26}, suppose $k_A$ admits a finite decomposition of the form 
$$k_A= \sum_{i,j} a_{ij} u_i \otimes v_j, 
\qquad a_{ij} \in \C, 
u_i \in \mathcal{D}'(\mathfrak{P}^+_\Gamma), 
v_j \in \mathcal{D}'(\mathfrak{P}^-_\Gamma )$$ 
with $u_i,v_j$ normalized so that $\langle u_i \otimes v_j, \mathbbm{1}_{\mathfrak{P}_\Gamma} \rangle = \delta _{ij}$.
Then
$
    \mathrm{Tr}(A) = \sum_{i,j} a_{ij}\delta _{ij} = \sum_{i} a_{ii}.
$

\begin{definition}[Invariant Ruelle distributions, {\cite[Def.~4.3]{ArendsPalmirotta26}}] \label{def:invariant_Ruelle_dist}
    Let $s\in \C$ be a non-exceptional resonance (i.e., $\chi(s) \neq \pm 1$) of multiplicity $\mathfrak{m}\in \N_0$, such that $q^{\frac{1}{2} + is} \notin \{\pm 1, \pm q\}$ and assume that no Jordan block $J(s)$ occurs at $s$. 

    The \emph{invariant Ruelle distribution} is defined by
    \begin{equation} \label{eq:invariantRuelledist}
        \mathcal{T}_s \colon
        \begin{cases}
            C^{\mathrm{l c}}(\mathfrak{P}_\Gamma) &\rightarrow \mathbb{C} \\
            \qquad f &\mapsto \mathrm{Tr}(f \Pi_s),
        \end{cases}
         \end{equation}
        where $\Pi_s= \sum_{\ell ,n=1}^{\mathfrak{m}} u_\ell \otimes v_{n}$ denotes the projection onto the eigenspace
        $$
        \mathcal{E}_s(\mathcal{L}'_{\Gamma , +}; \mathcal{D}'(\mathfrak{P}^+_\Gamma)) \otimes \mathcal{E}_{-{s}}(\mathcal{L}'_{\Gamma,-};\mathcal{D}'(\mathfrak{P}^-_\Gamma)) \subseteq \mathcal{D}'(\mathfrak{P}_\Gamma^+ \times \mathfrak{P}^-_\Gamma).
        $$
\end{definition}
Note that $-s$ in the second eigenspace is precisely what gives the same one-step eigenvalue on the past side. The assumption that no Jordan block occurs ensures that we are dealing with genuine eigenvectors.

%----
\begin{lem} \label{lem:Ruelledist_sigma_invariance}
    The distribution $\mathcal{T}_s\in \mathcal{D}'(\mathfrak{P}_\Gamma)$ is
    $\sigma$-invariant, i.e.,
    $$
       \mathcal{T}_s(f\circ \sigma)=\mathcal{T}_s(f)
       \qquad
       \forall\, f\in C^{\mathrm{lc}}(\mathfrak{P}_\Gamma).
    $$
\end{lem}

\begin{proof} 
    Let $\lambda(s_0)=q^{-(\frac12+is_0)}\neq 0$ denote the transfer-operator eigenvalue corresponding to the resonance parameter $s_0\in \mathscr{H}_{\vartheta,1}$. 
    With the conventions of Definition~\ref{def:invariant_Ruelle_dist}, if
    $
    u\in \mathcal{E}_{s_0}(\mathcal{L}'_{\Gamma,+};
    \mathcal{D}'(\mathfrak{P}^+_\Gamma))$ 
    and
    $v\in \mathcal{E}_{-s_0}(\mathcal{L}'_{\Gamma,-};
    \mathcal{D}'(\mathfrak{P}^-_\Gamma)),
    $
    then
    \begin{equation}\label{eq:eigenrelations}
       \mathcal{L}'_{\Gamma,+}u=\lambda(s_0) u,
       \qquad
       \mathcal{L}'_{\Gamma,-}v=\lambda(s_0) v .
    \end{equation}
    It is enough to prove the claim for elementary tensors $u\otimes v$, because
    the kernel of $\Pi_s$ is a finite linear combination of such tensors.
    
    Write $\mathbf{p} = (\mathbf{p}_+,\mathbf{p}_-)=((\vec{e}_1,\vec{e}_2,\ldots),(\ldots,\vec{b}_2,\vec{b}_1))\in \mathfrak{P}_\Gamma$ and consider its corresponding admissible and non-backtracking bi-finite chain
    $$
    \mathbf{p}_{m,m'}=(\mathbf{p}_{m',+},\mathbf{p}_{m,-}) =((\vec{e}_1,\ldots,\vec{e}_{m'}),(\vec{b}_m,\ldots,\vec{b}_1)) \in \mathcal{P}_{m+m'}
    \subset \mathfrak{P}_\Gamma
    $$
    with prescribed past segment $\mathbf{p}_{m,-}\in \mathcal{P}_{m}$ and prescribed future segment $\mathbf{p}_{m',+} \in \mathcal{P}_{m'}$ for large enough $m,m'\geq 1.$
    The characteristic function $\mathbbm{1}_{\mathbf{p}_{m,m'}}$, viewed
    as a function on $\mathcal{P}_{m'} \times \mathcal{P}_{m} \subset \mathfrak{P}^+_\Gamma\times \mathfrak{P}^-_\Gamma$ by extension
    by zero, factors as
    $$
    \mathbbm{1}_{\mathbf{p}_{m,m'}}
    = \mathbbm{1}_{\mathbf{p}_{m',+}} \otimes \mathbbm{1}_{\mathbf{p}_{m,-}}
    \coloneqq \mathbbm{1}_{[\vec{e}_1,\ldots,\vec{e}_{m'}]_+} \otimes \mathbbm{1}_{[\vec{b}_m,\ldots,\vec{b}_1]_-}.
    $$
    Hence
    $$
    \langle u\otimes v,\mathbbm{1}_{\mathbf{p}_{m,m'}} \rangle
    =
    \langle u,\mathbbm{1}_{\mathbf{p}_{m',+}}\rangle
    \langle v,\mathbbm{1}_{\mathbf{p}_{m,-}}\rangle .
    $$
    
    By Definition~\ref{def:space_biinfinite_chains} of the two-sided shift and by Remark~\ref{rem:biinfinite_chains}~(iii), $\sigma^{-1}$ removes the first past edge and appends it to the future, one has
    $
    \sigma^{-1}(\mathbf{p}_{m,m'})
    =  ((\vec{b}_1, \vec{e}_1,\ldots,\vec{e}_{m'}),(\vec{b}_m,\ldots,\vec{b}_2))
    $
    Thus
    $
    \mathbbm{1}_{\mathbf{p}_{m,m'}} \circ \sigma = \mathbbm{1}_{\sigma^{-1}(\mathbf{p}_{m,m'})}.
    $
    Consequently,
    $$
    \langle u \otimes v,\mathbbm{1}_{\mathbf{p}_{m,m'}} \circ \sigma\rangle
    = \langle u\otimes v,\mathbbm{1}_{\sigma^{-1}(\mathbf{p}_{m,m'})} \rangle 
    = \langle u,\mathbbm{1}_{[\vec{b}_1,\vec{e}_1,\ldots,\vec{e}_{m'}]_+}\rangle
    \langle v, \mathbbm{1}_{[\vec{b}_m,\ldots,\vec{b}_2]_-}\rangle.
    $$
    
    Using the transfer-operator eigenrelations \eqref{eq:eigenrelations}, we get
    $$
    \mathcal{L}_{\Gamma,+}
    \mathbbm{1}_{[\vec{b}_1,\vec{e}_1,\ldots,\vec{e}_{m'}]_+}
    =
    \mathbbm{1}_{[\vec{e}_1,\ldots,\vec{e}_{m'}]_+} = \mathbbm{1}_{\mathbf{p}_{m',+}}.
    $$
    Therefore
    $$
    \langle u, \mathbbm{1}_{\mathbf{p}_{m',+}} \rangle
    = \langle u, \mathcal{L}_{\Gamma,+}
    \mathbbm{1}_{[\vec{b}_1,\vec{e}_1,\ldots,\vec{e}_{m'}]_+} \rangle
    = \langle \mathcal{L}'_{\Gamma,+}u, \mathbbm{1}_{[\vec{b}_1,\vec{e}_1,\ldots,\vec{e}_{m'}]_+} \rangle
    = \lambda(s_0) \langle u, \mathbbm{1}_{[\vec{b}_1,\vec{e}_1,\ldots,\vec{e}_{m'}]_+} \rangle .
    $$
    Hence
    \begin{equation}\label{eq:future-cylinder}
    \langle u, \mathbbm{1}_{[\vec{b}_1,\vec{e}_1,\ldots,\vec{e}_{m'}]_+} \rangle
    = \lambda(s_0)^{-1} \langle u, \mathbbm{1}_{\mathbf{p}_{m',+}} \rangle.
    \end{equation}
    Similarly, on the past side,
    $ \mathcal{L}_{\Gamma,-} \mathbbm{1}_{\mathbf{p}_{m,-}} = \mathbbm{1}_{[\vec{b}_m,\ldots,\vec{b}_2]_-}.
    $
    Therefore
    $$
    \langle v, \mathbbm{1}_{[\vec{b}_m,\ldots,\vec{b}_2]_-} \rangle
    = \langle v, \mathcal{L}_{\Gamma,-} \mathbbm{1}_{\mathbf{p}_{m,-}} \rangle
    = \langle \mathcal{L}'_{\Gamma,-}v, \mathbbm{1}_{\mathbf{p}_{m,-}} \rangle
    = \lambda(s_0) \langle v, \mathbbm{1}_{\mathbf{p}_{m,-}} \rangle .
    $$
    Combining this with \eqref{eq:future-cylinder}, the factors $\lambda(s_0)^{-1}$ and $\lambda(s_0)$ cancel:
    $$
    \begin{aligned}
    \langle u\otimes v,\mathbbm{1}_{\mathbf{p}_{m,m'}} \circ\sigma\rangle
    &=
    \lambda(s_0)^{-1} \langle u,\mathbbm{1}_{\mathbf{p}_{m',+}} \rangle
    \cdot
    \lambda(s_0) \langle v,\mathbbm{1}_{\mathbf{p}_{m,-}} \rangle 
    &=
    \langle u\otimes v,\mathbbm{1}_{\mathbf{p}_{m,m'}} \rangle .
    \end{aligned}
    $$
    By Remark~\ref{rem:basis_delta_p} and \eqref{eq:Clc_Fm}, since $C^{\mathrm{lc}}(\mathfrak{P}_\Gamma)$ is spanned by characteristic
    functions of such paths, it follows that
    $$
    \langle u\otimes v,f\circ\sigma \rangle
    = \langle u\otimes v,f\rangle
    \qquad \forall f\in C^{\mathrm{lc}}(\mathfrak{P}_\Gamma).
    $$
    
    Now denote the kernel of the finite-rank projection $\Pi_{s_0}$ as
    $$
    k_{\Pi_{s_0}}
    =
    \sum_{\ell,n=1}^{\mathfrak{m}} u_\ell\otimes v_n, 
    $$
    with 
    $u_\ell\in \mathcal{E}_s(\mathcal{L}'_{\Gamma,+}; \mathcal{D}'(\mathfrak{P}^+_\Gamma)),$ and 
    $v_n\in\mathcal{E}_{-s}(\mathcal{L}'_{\Gamma,-}; \mathcal{D}'(\mathfrak{P}^-_\Gamma))$.
    By the previous computation, each tensor $u_\ell\otimes v_n$ is
    $\sigma$-invariant after restriction to $\mathfrak{P}_\Gamma$. 
    Hence
    $$
    \langle k_{\Pi_{s_0}},f\circ\sigma\rangle
    =
    \langle k_{\Pi_{s_0}},f\rangle .
    $$
    Finally, by the definition of $\mathcal{T}_{s_0}$, we get
    $
    \mathcal{T}_{s_0}(f)
    =
    \operatorname{Tr}(f\Pi_{s_0})
    =
    \langle k_{\Pi_{s_0}},f\rangle ,
    $
    where $f$ is viewed as a locally constant function on
    $\mathfrak{P}^+_\Gamma\times \mathfrak{P}^-_\Gamma$ by extension by zero from
    $\mathfrak{P}_\Gamma$. 
    Therefore
    $$
    \mathcal{T}_{s_0}(f\circ\sigma)
    = \langle k_{\Pi_{s_0}},f\circ \sigma \rangle
    = \langle k_{\Pi_{s_0}},f\rangle
    = \mathcal{T}_{s_0}(f).
    $$
    Thus $\mathcal{T}_s$ is $\sigma$-invariant.
\end{proof}
%----

Note that in the case $J(s_0)=1$, the Riesz spectral projector coincides with the invariant Ruelle projector in \eqref{eq:invariantRuelledist}.
To establish this rigorously, we introduce a \emph{past-to-future reconstruction map} that identifies the range of the one-sided Riesz projector \eqref{eq:rangPis0_Es0}  with the past eigenspace of $\mathcal{L}'_{\Gamma,-}$.

\begin{lem}[Past-to-future reconstruction]
\label{lem:past_future_reconstruction}
    Let $\lambda(s_0)=q^{-(\frac12+is_0)}\neq 0$ be non-exceptional transfer-operator eigenvalue corresponding to the resonance parameter $s_0\in \mathscr{H}_{\vartheta,1}$ with $J(s_0)=1$.
    Consider the past-to-future reconstruction map
    $$
    \mathcal{J}_- \colon \mathcal{E}_{-s_0}(\mathcal{L}'_{\Gamma,-};
    \mathcal{D}'(\mathfrak{P}_\Gamma^-)) \longrightarrow \operatorname{Ran}(\Pi_{s_0})=\mathcal{E}_{s_0}(\mathcal{L}'_{\Gamma,+};
    \mathcal{D}'(\mathfrak{P}_\Gamma^+))
    $$
    defined by
    $
    (\mathcal{J}_- v)(\mathbf{p}_+)
    \coloneqq
    \langle v,\mathbbm{1}_{\mathbb{F}(\mathbf{p}_+)}\rangle,$
    where
    $\mathbb{F}(\mathbf{p}_+)
    \coloneqq
    \{\mathbf{p}_-\in\mathfrak{P}_\Gamma^-:
    (\mathbf{p}_+,\mathbf{p}_-)\in\mathfrak{P}_\Gamma\}
    $ is the fibre.
    Then
    $
    \mathcal{J}_- v \in \mathcal{E}_{s_0}(\mathcal{L}'_{\Gamma,+};     \mathcal{D}'(\mathfrak{P}_\Gamma^+))
    $ and the map $\mathcal{J}_-$ is an isomorphism.
\end{lem}

\begin{proof} 
    Let $v\in \mathcal{E}_{-s_0}(\mathcal{L}'_{\Gamma,-}; \mathcal{D}'(\mathfrak{P}_\Gamma^-))$. 
    Recall the edge and vertex projections $$\pi_\pm^{\mathfrak{E}}:\mathfrak{P}_\Gamma^\pm \longrightarrow  \mathfrak{E}_\Gamma \qquad \text{and} \qquad \pi_\pm^{\mathfrak{X}}:\mathfrak{P}_\Gamma^\pm \longrightarrow  \mathfrak{X}_\Gamma, $$ 
    which map a chain to its first oriented edge and its initial vertex, respectively, as defined in Definition~\ref{def:space_biinfinite_chains}.
    Write the edge pushforward of $v$ by
    $$
    \alpha_v(\vec{e})
    \coloneqq \langle v, \mathbbm{1}_{(\pi_-^{\mathfrak{E}})^{-1}(\vec{e})} \rangle
    = (\pi_-^{\mathfrak{E}})_*v(\vec{e}), \qquad \vec{e}\in \mathfrak{E}_\Gamma,
    $$
    and by 
    $$
    \beta_v(x)
    \coloneqq \langle v, \mathbbm{1}_{(\pi^{\mathfrak{X}}_-)^{-1}(x)} \rangle
    = (\pi^{\mathfrak{X}}_-)_*v(x), \qquad x\in \mathfrak{X}_\Gamma
    $$
    its vertex pushforward.

    Let $\mathbf{p}_+=(\vec{e}_1,\vec{e}_2,\ldots) \in \mathfrak{P}_\Gamma^+$. 
    By definition of $\mathfrak{P}_\Gamma$, the fiber
    $$
    \mathbb{F}(\mathbf{p}_+)
    = \{\mathbf{p}_- \in \mathfrak{P}_\Gamma^- \colon (\mathbf{p}_+,\mathbf{p}_-)\in\mathfrak{P}_\Gamma\}
    =  \{\mathbf{p}_-\in \mathfrak{P}_\Gamma^- \colon
    \tau(\mathbf{p}_-)= \iota(\vec{e}_1), 
    \tau(\mathbf{p}_-) \neq \vec{e}_1^{\, \mathrm{op}}\}
    $$
    is the set of past chains whose last edge ends at $\iota(\vec{e}_1)$, but whose
    last edge is not $\vec{e}_1^{\,\mathrm{op}}$. 
    Hence
    $$
    \mathbbm{1}_{\mathbb{F}(\mathbf{p}_+)}
    =
    \mathbbm{1}_{(\pi^{\mathfrak{X}}_-)^{-1}(\iota(\vec{e}_1))}
    -
    \mathbbm{1}_{(\pi_-^{\mathfrak{E}})^{-1}(\vec{e}_1^{\,\mathrm{op}})},
    $$
    where the first term counts all past paths ending at the correct vertex and the second removes the one forbidden last edge.
    Therefore
    \begin{equation}\label{eq:Jminus-edge-formula}
    (\mathcal{J}_-v)(\mathbf{p}_+)
    = \beta_v(\iota(\vec{e}_1)) - \alpha_v(\vec{e}_1^{\,\mathrm{op}}).
    \end{equation}
    Note that at this point, $\mathcal{J}_-v$ appears to depend on the first edge of the future path, which is exactly what we want.

    For the transfer identity for $\alpha_v$, we use the eigenrelation
    $ \mathcal{L}'_{\Gamma,-}v=\lambda(s_0)v.$
    Testing it against the edge-indicator
    $\mathbbm{1}_{(\pi_-^{\mathfrak{E}})^{-1}(\vec{e})}$, we get
    $$
    \lambda(s_0) \alpha_v(\vec{e})
    = \langle \mathcal{L}'_{\Gamma,-}v, \mathbbm{1}_{(\pi_-^{\mathfrak{E}})^{-1}(\vec{e})} \rangle
    = \langle v, \mathcal{L}_{\Gamma,-} \mathbbm{1}_{(\pi_-^{\mathfrak{E}})^{-1}(\vec{e})} \rangle .
    $$
    But
    $
    \mathcal{L}_{\Gamma,-}
    \mathbbm{1}_{(\pi_-^{\mathfrak{E}})^{-1}(\vec{e})}
    $
    is the characteristic function of those past chains whose last edge may be followed by $\vec{e}$, i.e.,
    $\mathcal{L}_{\Gamma,-}
    \mathbbm{1}_{(\pi_-^{\mathfrak{E}})^{-1}(\vec{e})}
    = \sum_{\substack{\vec{b}: \tau(\vec{b})=\iota(\vec{e})\\ \vec{b} \neq \vec{e}^{\, \mathrm{op}}}} \mathbbm{1}_{(\pi_-^{\mathfrak{E}})^{-1}(\vec{b})}.$
    Since $\beta_v(\iota(\vec{e}))=\sum_{\vec{b}:\,\tau(\vec{b})=\iota(\vec{e})} \alpha_v(\vec{b})$, we obtain the relation
    $$
    \lambda(s_0)\,\alpha_v(\vec{e})
    = \sum_{\substack{\vec{b}:\,\tau(\vec{b})=\iota(\vec{e})\\
                     \vec{b}\neq \vec{e}^{\,\mathrm{op}}}}
    \alpha_v(\vec{b})
    = \beta_v(\iota(\vec{e}))-\alpha_v(\vec{e}^{\,\mathrm{op}}).
    $$
    Combining this with \eqref{eq:Jminus-edge-formula}, we obtain the simple formula
    \begin{equation}\label{eq:Jminus-is-lambda-alpha}
    (\mathcal{J}_-v)(\mathbf{p}_+)
    =
    \lambda(s_0)\,\alpha_v(\vec{e}_1),
    \qquad
    \mathbf{p}_+=(\vec{e}_1,\vec{e}_2,\ldots)\in \mathfrak{P}_\Gamma^+.
    \end{equation}
    In particular, \(\mathcal{J}_-v\) depends only on the first future edge.

    Next, we show that $\mathcal{J}_-v$ is a $\lambda(s_0)$-eigenfunction of $\mathcal{L}_{\Gamma,+}$. 
    For $\mathbf{p}_+=(\vec{e}_1,\vec{e}_2,\ldots) \in \mathfrak{P}_\Gamma^+$, using \eqref{eq:Jminus-is-lambda-alpha}, we then have
    $$
    \begin{aligned}
    (\mathcal{L}_{\Gamma,+}\mathcal{J}_-v)(\mathbf{p}_+)
    =
    \sum_{\substack{\vec{b}:\,\tau(\vec{b})=\iota(\vec{e}_1)\\
                     \vec{b}\neq \vec{e}_1^{\,\mathrm{op}}}}
    (\mathcal{J}_-v)(\vec{b},\vec{e}_1,\vec{e}_2,\ldots) 
    &=
    \sum_{\substack{\vec{b}:\,\tau(\vec{b})=\iota(\vec{e}_1)\\
                     \vec{b}\neq \vec{e}_1^{\,\mathrm{op}}}}
    \lambda(s_0)\,\alpha_v(\vec{b}) \\
    &=
    \lambda(s_0)\bigl(\beta_v(\iota(\vec{e}_1))-\alpha_v(\vec{e}_1^{\,\mathrm{op}})\bigr) \\
    &=
    \lambda(s_0)\,(\mathcal{J}_-v)(\mathbf{p}_+).
    \end{aligned}
    $$
    Thus
    $
    \mathcal{J}_-v\in\ker(\mathcal{L}_{\Gamma,+}-\lambda(s_0)\operatorname{Id})
    =
    \mathcal{E}_{s_0}(\mathcal{L}'_{\Gamma,+};     \mathcal{D}'(\mathfrak{P}_\Gamma^+)).
    $
    
    It remains to prove that $\mathcal{J}_-$ is an isomorphism. Suppose $\mathcal{J}_-v=0$. 
    By \eqref{eq:Jminus-is-lambda-alpha} and $\lambda(s_0) \neq0$, we get
    $ \alpha_v(\vec{e})=0$, for all $\vec{e}\in\mathfrak{E}_\Gamma.$
    Therefore also
    $$
    \beta_v(x)
    = \sum_{\substack{\vec{e}:\,\tau(\vec{e})=x}} \alpha_v(\vec{e})
    = 0,
    \qquad \forall\,x\in\mathfrak{X}_\Gamma.
    $$
    Hence $(\pi_-)_*v=0$.
    By the non-exceptional pushforward isomorphism from \cite[Thms.~8.3 and 11.5]{BHW23}
    $$
    \pi_{-,*}:\mathcal{E}_{-s_0}(\mathcal{L}'_{\Gamma,-}; \mathcal{D}'(\mathfrak{P}_\Gamma^-))
    \xrightarrow{\sim}
    \mathcal{E}_{\chi(s)}(\Laplace_{\Gamma}; \mathrm{Maps}(\mathfrak{X}_\Gamma, \mathbb{C})),
    $$
    we conclude that $v=0$. 
    Thus $\mathcal{J}_-$ is injective.

    Finally, by the same pushforward isomorphism on the future side \cite[Thms.~8.3 and 11.5]{BHW23} and since $\chi(s_0)=\chi(-s_0)$, we have
    $$
    \dim \mathcal{E}_{-s_0}(\mathcal{L}'_{\Gamma,-}; \mathcal{D}'(\mathfrak{P}_\Gamma^-)) =
    \dim \mathcal{E}_{\chi(s)}(\Laplace_{\Gamma}; \mathrm{Maps}(\mathfrak{X}_\Gamma, \mathbb{C}))
    =
    \dim \mathcal{E}_{s_0}(\mathcal{L}'_{\Gamma,+};     \mathcal{D}'(\mathfrak{P}_\Gamma^+)).
    $$
    Hence, since \(\mathcal{J}_-:\mathcal{E}_{-s_0}(\mathcal{L}'_{\Gamma,-}; \mathcal{D}'(\mathfrak{P}_\Gamma^-))\to \mathcal{E}_{s_0}(\mathcal{L}'_{\Gamma,+};     \mathcal{D}'(\mathfrak{P}_\Gamma^+))\) is injective between
    finite-dimensional vector spaces of the same dimension, it is an isomorphism.
\end{proof}

\begin{prop}[Riesz--Ruelle identification]
\label{prop:Riesz_Ruelle_identification}
    Fix $\vartheta\in (0,1)$ and set $\omega \equiv 1$.
    Let $\lambda(s_0)=q^{-(\frac12+is_0)}$ be the transfer-operator eigenvalue corresponding to the resonance parameter $s_0\in \mathscr{H}_{\vartheta,1}$ of  multiplicity $\mathfrak{m} \coloneqq m_{\pm,1}(s_0) \in\mathbb N_0$ with no Jordan block, i.e., $J(s_0)=1$, such that
    $\chi(s_0)\neq\pm1$ and $ q^{1+2is_0}\notin\{\pm1,\pm q\}.$
    Suppose that the past-to-future reconstruction map
    \[
        \mathcal{J}_-:
        \mathcal{E}_{-s_0}(\mathcal{L}'_{\Gamma,-};
        \mathcal{D}'(\mathfrak{P}_\Gamma^-))
        \longrightarrow
        \operatorname{Ran}(\Pi_{s_0}) \subset \F_\vartheta(\mathfrak{P}_\Gamma^+)
    \]
    defined in Lemma~\ref{lem:past_future_reconstruction}
    is an isomorphism. 

    Then the kernel associated with the one-sided Riesz
    projector $\Pi_{s_0}$ of Lemma~\ref{lem:Pi_tensor_decomp} coincides with the 
    invariant Ruelle projector $k_{\Pi_{s_0}}$ of Definition~\ref{def:invariant_Ruelle_dist}.

    In particular, for every
    $f_+\in C^{\mathrm{lc}}(\mathfrak{P}_\Gamma^+)$
    $$
       \langle k_{\Pi_{s_0}}, f_+\circ \pi_+ \rangle_{\mathfrak{P}_\Gamma}
       =
       \operatorname{Tr}\!\left(\Pi_{s_0} M_{f_+} \right)
       = \operatorname{Tr}\!\left(M_{f_+}\Pi_{s_0}\right)
       = \operatorname{Tr}\!\left(f_+\Pi_{s_0}\right),
    $$
    where $M_{f_+}$ denotes the multiplication by $f_+$ on $\F_\vartheta(\mathfrak{P}_\Gamma^+).$
\end{prop}

\begin{proof}
    Since $J(s_0)=1$, the Riesz projector $\Pi_{s_0}$ is a genuine spectral projector. 
    Hence by \eqref{eq:rangPis0_Es0}
    $$
    \operatorname{Ran}(\Pi_{s_0})
    = \ker(\mathcal{L}_{\Gamma,+}-\lambda(s_0)\operatorname{Id})
    =\mathcal{E}_{s_0}(\mathcal{L}'_{\Gamma,+}; \mathcal{D}'(\mathfrak{P}_\Gamma^+)).
    $$
    Choose a basis $(\varphi_\ell)_{\ell=1}^{\mathfrak{m}}$
    of $\operatorname{Ran}(\Pi_{s_0}) \subset \F_\vartheta(\mathfrak{P}_\Gamma^+)$. 
    By Lemma~\ref{lem:Pi_tensor_decomp}, there is a unique dual family
    $
    (\varphi_\ell^*)_{\ell=1}^{\mathfrak{m}}
    \subset \F_\vartheta'(\mathfrak{P}_\Gamma^+)
    $
    characterized by 
    $\langle \varphi_\ell^*,\varphi_n\rangle=\delta_{\ell n}$ and $\varphi_\ell^*|_{\ker(\Pi_{s_0})}=0,$
    such that 
    $
    \Pi_{s_0}
    = \sum_{\ell=1}^{\mathfrak{m}}
    \varphi_\ell\otimes\varphi_\ell^*.
    $

    Since $\Pi_{s_0}$ commutes with $\mathcal{L}_{\Gamma,+}$, and $\mathcal{L}_{\Gamma,+}$ acts on $\operatorname{Ran}(\Pi_{s_0})$ by the scalar $\lambda(s_0)$, each $\varphi_\ell^*$ is a dual eigenvector:
    $$
    \mathcal{L}'_{\Gamma,+}\varphi_\ell^*
    = \lambda(s_0)\varphi_\ell^*, \qquad \ell=1, \ldots, \mathfrak{m}.
    $$
    Indeed, for $f_+\in \F_\vartheta(\mathfrak{P}_\Gamma^+)$, we have
    $$
    \begin{aligned}
    \langle \mathcal{L}'_{\Gamma,+}\varphi_\ell^*, f_+\rangle
    =
    \langle \varphi_\ell^*,\mathcal{L}_{\Gamma,+} f_+\rangle 
    =
    \langle \varphi_\ell^*,\Pi_{s_0}\mathcal{L}_{\Gamma,+} f_+\rangle 
    =
    \langle \varphi_\ell^*,\mathcal{L}_{\Gamma,+}\Pi_{s_0}f_+\rangle 
    &=
    \lambda(s_0)\langle \varphi_\ell^*,\Pi_{s_0}f_+\rangle \\
    &=
    \lambda(s_0)\langle \varphi_\ell^*,f_+\rangle.
    \end{aligned}
    $$
    Thus, after restriction to $C^{\mathrm{lc}}(\mathfrak{P}_\Gamma^+)$, by Remark~\ref{rem:unweighted_res_states}, we may regard $u_\ell\coloneqq\varphi_\ell^*$ as an element of
    $
    \mathcal{E}_{s_0}(\mathcal{L}'_{\Gamma,+};
    \mathcal{D}'(\mathfrak{P}_\Gamma^+)).
    $
    
    Moreover, since $\mathcal{J}_-$ is an isomorphism by Lemma~\ref{lem:past_future_reconstruction}, for each $\ell$ there is a unique $v_\ell \in \mathcal{E}_{-s_0}(\mathcal{L}'_{\Gamma,-};\mathcal{D}'(\mathfrak{P}_\Gamma^-))
    $ such that $\mathcal{J}_-v_\ell=\varphi_\ell.$
    Equivalently,
    \begin{equation}\label{eq:fiber_reconstruction_vell}
    \langle v_\ell,\mathbbm{1}_{\mathbb{F}(\mathbf{p}_+)}\rangle
    = \varphi_\ell(\mathbf{p}_+)
    \qquad
    \forall\,\mathbf{p}_+\in\mathfrak{P}_\Gamma^+.
    \end{equation}

    Now consider the two-sided distribution 
    $$
    k_{\Pi_{s_0}} \coloneqq \sum_{\ell=1}^{\mathfrak{m}} u_\ell\otimes v_\ell
    = \sum_{\ell=1}^{\mathfrak{m}} \varphi_\ell^*\otimes v_\ell
    \in \mathcal{D}'(\mathfrak{P}_\Gamma^+\times\mathfrak{P}_\Gamma^-).
    $$
    We claim that this is the invariant Ruelle kernel. 
    It suffices to verify the normalisation condition
    $
    \left\langle u_\ell\otimes v_n,
    \mathbbm{1}_{\mathfrak{P}_\Gamma}\right\rangle
    = \delta_{\ell n}.
    $
    Indeed, using the definition of the tensor product \eqref{eq:tensor_product_def} and the reconstruction identity \eqref{eq:fiber_reconstruction_vell}, we get
    $$
    \begin{aligned}
    \langle u_\ell\otimes v_n, \mathbbm{1}_{\mathfrak{P}_\Gamma}\rangle
    = \langle u_\ell,\, \mathbf{p}_+\mapsto \langle v_n, \mathbbm{1}_{\mathbb{F}(\mathbf{p}_+)} \rangle \rangle 
    = \langle \varphi_\ell^*,\, \mathbf{p}_+\mapsto \varphi_n(\mathbf{p}_+) \rangle 
    = \langle \varphi_\ell^*,\varphi_n\rangle 
    &= \delta_{\ell n}.
    \end{aligned}
    $$
    Therefore, the tensors $u_\ell\otimes v_n$ are normalised exactly as in the definition of the invariant Ruelle projector. 
    Hence summing over all $\ell, n$, claims that $k_{\Pi_{s_0}}$ is the invariant Ruelle kernel in Definition~\ref{def:Ruelle_transfer_op}.

    It remains to check that $k_{\Pi_{s_0}}$ is the kernel associated with the one-sided Riesz projector in Lemma~\ref{lem:Pi_tensor_decomp}. 
    Let $f_+\in C^{\mathrm{lc}}(\mathfrak{P}_\Gamma^+)$.
    As in the discussion preceding \eqref{eq:tensor_product_def}, we extend  $f_+\circ\pi_+\in C^{\mathrm{lc}}(\mathfrak{P}_\Gamma)$
    by zero to a function on $\mathfrak{P}_\Gamma^+\times\mathfrak{P}_\Gamma^-$.
    Concretely,
    $$
        (f_+\circ\pi_+)(\mathbf{p}_+,\mathbf{p}_-) =
        \begin{cases}
            f_+(\mathbf{p}_+), & (\mathbf{p}_+,\mathbf{p}_-)\in\mathfrak{P}_\Gamma,\\
            0, & \text{otherwise.}
        \end{cases}
    $$
    Then using \eqref{eq:tensor_product_def} and \eqref{eq:fiber_reconstruction_vell}, we obtain
    $$
    \begin{aligned}
    \langle k_{\Pi_{s_0}},f_+ \circ \pi_+ \rangle_{\mathfrak{P}_\Gamma}
    = \sum_{\ell=1}^{\mathfrak{m}}
    \langle u_\ell\otimes v_\ell,f_+ \circ\pi_+\rangle 
    &=
    \sum_{\ell=1}^{\mathfrak{m}} \langle
    u_\ell,\, \mathbf{p}_+\mapsto \langle v_\ell, (f_+ \circ\pi_+)(\mathbf{p}_+, \cdot) \rangle \rangle \\
    &=
    \sum_{\ell=1}^{\mathfrak{m}} \langle
    u_\ell,\, \mathbf{p}_+\mapsto f_+(\mathbf{p}_+) \langle v_\ell,
    \mathbbm{1}_{\mathbb{F}(\mathbf{p}_+)} \rangle \rangle \\
    &=
    \sum_{\ell=1}^{\mathfrak{m}} \langle \varphi_\ell^*,\,
    f_+\,\varphi_\ell \rangle.
    \end{aligned}
    $$
    On the other hand, since
    $
    \Pi_{s_0} = \sum_{\ell=1}^{\mathfrak{m}} \varphi_\ell\otimes\varphi_\ell^*,
    $
    we have
    $
    \operatorname{Tr}(\Pi_{s_0}M_{f_+})
    =
    \sum_{\ell=1}^{\mathfrak{m}} \langle
    \varphi_\ell^*,f_+\,\varphi_\ell \rangle,
    $
    where $M_{f_+}$ denotes the multiplication by $f_+$ on $\F_\vartheta(\mathfrak{P}_\Gamma^+).$ 
    Thus
    $$
    \langle k_{\Pi_{s_0}},f_+ \circ\pi_+\rangle_{\mathfrak{P}_\Gamma}
    =\operatorname{Tr}(\Pi_{s_0}M_{f_+}) 
    = \operatorname{Tr}(\Pi_{s_0} f_+)
    = \operatorname{Tr}(f_+\Pi_{s_0}),
    $$
    where the last equality follows from the cyclicity of the trace for finite-rank operators.
    This is precisely the compatibility between the one-sided Riesz projector and the two-sided invariant Ruelle kernel. 
    Hence the kernel associated with $\Pi_{s_0}$ coincides with $k_{\Pi_{s_0}}$.
\end{proof}

Consequently, by the Riesz-Ruelle identification, we can show that the trace of the invariant Ruelle projector against a two-sided observable can be reduced to the operator trace of the corresponding one-sided Riesz projector.

\begin{prop}[Trace reduction to the one-sided setting]
\label{prop:trace_reduction}
    Fix $\vartheta \in(0,1)$ and $\omega\equiv 1$. 
    Let $s_0\in\mathscr H_{\vartheta,1}$ satisfy the hypotheses of
    Proposition~\ref{prop:Riesz_Ruelle_identification}, and let $\Pi_{s_0}$ be the corresponding
    Riesz projector, identified with the invariant Ruelle projector.
    Let $f\in C^{\mathrm{lc}}(\mathfrak{P}_\Gamma)$ with its one-sided reduction
    $f_+\in C^{\mathrm{lc}}(\mathfrak{P}_\Gamma^+)$ as in Lemma~\ref{lem:one-sided-reduction}.
    Then
    $$
       \operatorname{Tr}(f\Pi_{s_0})
       =  \operatorname{Tr}(\Pi_{s_0}f_+).
    $$
\end{prop}

\begin{proof}
    By Definition~\ref{def:invariant_Ruelle_dist} of the invariant Ruelle distribution,
    $
       \operatorname{Tr}(f\Pi_{s_0})
       =
       \langle k_{\Pi_{s_0}},f \rangle_{\mathfrak{P}_\Gamma},
    $
    where $ k_{\Pi_{s_0}}$ denotes the invariant Ruelle kernel.
    By Lemma~\ref{lem:Ruelledist_sigma_invariance}, the distribution
    $k_{\Pi_{s_0}}$ is $\sigma$-invariant. 
    Hence iterating $m$ times and using the one-sided reduction in Lemma~\ref{lem:one-sided-reduction} for some $m\geq 1$, gives
    $$
       \langle k_{\Pi_{s_0}},f\bigr\rangle_{\mathfrak{P}_\Gamma}
       = \langle k_{\Pi_{s_0}},f\circ\sigma^m\bigr\rangle_{\mathfrak{P}_\Gamma}
       = \langle k_{\Pi_{s_0}},f_+\circ\pi_+\rangle_{\mathfrak{P}_\Gamma}.
    $$
    Finally, using Proposition~\ref{prop:Riesz_Ruelle_identification}, gives
    $$
       \operatorname{Tr}(f\Pi_{s_0})
       = \langle k_{\Pi_{s_0}},f\bigr\rangle_{\mathfrak{P}_\Gamma}
       = \langle k_{\Pi_{s_0}},f_+\circ\pi_+ \rangle_{\mathfrak{P}_\Gamma}
       =
       \operatorname{Tr}(\Pi_{s_0}M_{f_+})
       =\operatorname{Tr}(\Pi_{s_0}f_+).
    $$ 
\end{proof}

% -------------------------------%%--------------------------------------%
\subsubsection{Proof of Theorem~\ref{thm:residue_PS}} \label{sect:proof_Thm_residuePS}
Let $s_0\in \C$ be a non-exceptional resonance (i.e., $\chi(s) \neq \pm 1$) of multiplicity $\mathfrak{m}\in \N_0$ such that $q^{\frac12+is_0} \notin \{\pm 1, \pm q\}$ and such that we do not have a Jordan block for $s_0$.
%Applying %%Theorem~\ref{thm:mero_dynfct} (iii) with $J(s_0)=1$ and
Applying Theorem~\ref{thm:mero_dynfct}~(iii) with $0=k<J(s_0)=1$ and using Proposition~\ref{prop:trace_reduction} gives directly
$$
    \operatorname{Res}_{s=s_0}\bigl[\mathbf{Z}_f(s)\bigr]
    =
    \frac{i}{\log q}\operatorname{Tr}\bigl[(\mathcal{L}_{\Gamma,+}-\lambda(s_0))^{0}\Pi_{s_0}f\bigr]
    =
    \frac{i}{\log q}\operatorname{Tr}[\Pi_{s_0}f]
    =\frac{i}{\log q} \mathcal{T}_{s_0}(f).
$$
Here $f\in C^{\mathrm{lc}}(\mathfrak{P}_\Gamma)$ enters the trace in Theorem~\ref{thm:mero_dynfct}~(iii) through its multiplication action on $\F_\vartheta(\mathfrak{P}_\Gamma^+)$, i.e., via the one-sided reduction $f_+\in C^{\mathrm{lc}}(\mathfrak{P}_\Gamma^+)$ of
Lemma~\ref{lem:one-sided-reduction}, consistently with
$\mathbf{Z}_f(s)=\operatorname{Tr}[\mathbf{R}(s)f_+]$ in
Proposition~\ref{prop:weighted_trace}.

Now, using the relation in \cite[Thm.~5]{ArendsPalmirotta26}, gives the desired residue formula for the Patterson–Sullivan distributions
\begin{eqnarray*}
    \mathrm{Res}_{s=s_0} \big[\mathbf{Z}_{f}(s)\big] 
    = \frac{i}{\log q}\mathcal{T}_{s_0}(f) 
    %&=& \langle f\Pi_s , \mathbbm{1}_{\mathfrak{P}_\Gamma } \rangle \\
   % &=& \sum_{\ell = 1}^m(u_{\ell} \otimes v_{\ell})(f) \\
    &=&  \frac{i}{\log q} \frac{q^{1 + 2 i s_0} - 1}{q^{1 + 2 i s_0} - q} \sum_{\ell=1}^{\mathfrak{m}} \mathrm{PS}_{\phi_\ell, \phi_\ell}(f).
\end{eqnarray*}
for any $\ell^2(\mathfrak{X}_\Gamma)$-orthonormal basis $(\phi_\ell)^{\mathfrak{m}}_{\ell=1}$ of $\mathcal{E}_{\chi(s_0)}(\Laplace; \mathrm{Maps}(\mathfrak{X}, \mathbb{C}))^\Gamma$.
\hfill $\square$

% -------------------------------%%--------------------------------------%
\subsection{Semiclassical formula for Wigner distributions} \label{sect:Wigner}
In this subsection, we relate the residues of the dynamical zeta function $\mathbf{Z}_{f}$ defined in \eqref{eq:Zeta}, %defined purely in terms of primitive closed finite paths (i.e., primitive cycles), 
to quantum phase-space distributions, namely the \emph{Wigner distributions}, arising from pseudo-differential operators on finite $(q+1)$-regular graphs $\mathfrak{G}_\Gamma$.\\

Let $\Omega$ denote the \emph{boundary at infinity of the tree}, i.e.,  the space of equivalence classes $[\vec{e}_1, \vec{e}_2, \dots]$ of infinite non-backtracking chains.
We endow $\Omega$ with the  usual topology generated by the sets
$$
    \partial_+\vec{e}
    \coloneqq
    \bigl\{
        w \in\Omega \;:\;
        w =[\vec{e},\vec{e}_1,\vec{e}_2,\ldots]
        \text{ for some non-backtracking chain }
        (\vec{e},\vec{e}_1,\vec{e}_2,\ldots)
    \bigr\}
$$
and by reversing the orientation, $\partial_- \vec{e}\coloneqq \partial_+ \vec{e}^{\;\mathrm{op}}$.\\
Since non-backtracking chains on $\mathfrak{X}$ may be viewed as geodesic rays, one has
$$\mathfrak{P}^\pm \cong \mathfrak{X} \times \Omega = S\mathfrak{X},$$
see \cite[§2.4]{ArendsPalmirotta26} for details. 
Passing to the quotient by the $\Gamma$-action, we set
$S\mathfrak{X}_\Gamma \coloneqq \Gamma \backslash (\mathfrak{X} \times \Omega),$ which we call the \emph{phase space}.

\begin{definition}[Wigner distributions, {\cite[Def.~5.2]{ArendsPalmirotta26}}] \label{def:Wigner}
    For every eigenfunction $\phi \in \mathcal{E}_{\chi(s_0)}(\Delta_\Gamma; \mathrm{Maps}(\mathfrak{X}_\Gamma, \C))$ with spectral parameter $s_0\in \C$, we define the associated (diagonal) \emph{Wigner distribution} $\mathrm{W}_{\phi,\phi}\in \mathcal{D}'(S\mathfrak{X}_\Gamma)$ by means of the $\ell^2$-pairing
    \begin{equation*}
        W_{\phi,\phi}(a)= \langle \mathrm{Op}(a) \phi, \phi \rangle_{\ell^2(\mathfrak{X}_\Gamma)}, \quad a\in C^{\mathrm{lc}}(S\mathfrak{X}_\Gamma),
    \end{equation*}
    where $\mathrm{Op}(a)$ is the pseudo-differential operator on $\mathfrak{X}_\Gamma$ defined in \cite[Def.~5.1]{ArendsPalmirotta26}.
\end{definition}

For a fixed resonance $s_0$, \cite[Thm.~6]{ArendsPalmirotta26} establishes an exact correspondence between the Patterson–Sullivan distributions $\mathrm{PS}_{\phi,\phi'} \in \mathcal{D}'(\mathfrak{P}_\Gamma)$ defined in \eqref{eq:PS}, viewed as distributions on $S\mathfrak{X}_\Gamma$, and the Wigner distributions. More precisely, for every $a\in C^{\mathrm{lc}}(S\mathfrak{X}_\Gamma)$ and $n \in \N_0$  one has 
\begin{equation} \label{eq:W_PS_relation}
    W_{\phi , \phi}(a - q^{- n}\mathcal{L}_e^n a) 
    = \mathrm{P S}_{\phi , \phi }\left(a - q^{- n}\mathcal{L}_e^n a + \sum_{k = 1}^n q^{-2ki s_0}q^{- k}\mathcal{H}_k(a)\right),
\end{equation}
where $\mathcal{L}_e^n$ denotes the $n$-fold iteration of the vertex transfer operator as in \cite[§5.1.3]{ArendsPalmirotta26}, and $\mathcal{H}_k$ is defined as the sum of the values $a(y_1 , \ldots, y_k , x_{k + 1}, x_{k + 2}, \ldots )$ over all non-backtracking paths satisfying $y_k \neq x_k$, for $1 \leq k \leq n$.

For every $n \in \N$, define the test function
\begin{equation} \label{eq:fn}
    f_{n, s_0,a} \coloneqq f_{n,a}^{\mathrm{prin}}+  \sum_{k = 1}^n q^{-2ki s_0}q^{- k}\mathcal{H}_k(a), \qquad 
    \text{ with } f_{n,a}^{\mathrm{prin}}
    \coloneqq
    a-q^{-n}\mathcal{L}_e^n a.
\end{equation}
Since $a$ is locally constant and the operators $\mathcal{L}_e$ and $\mathcal{H}_k$ preserve local constancy, we have
$f_{n,s_0,a}\in C^{\mathrm{lc}}(S\mathfrak{X}_\Gamma)$.
The preceding identity \eqref{eq:W_PS_relation} may therefore be written compactly as
$$
\mathrm{W}_{\phi,\phi}
    \bigl(f_n^{\mathrm{prin}}(a)\bigr)
    =
    \mathrm{PS}_{\phi,\phi}
    \bigl(f_{n,s_0}(a)\bigr), \qquad a \in C^{\mathrm{lc}}(S\mathfrak{X}_\Gamma).
$$
Thus the Wigner term is precisely the principal part
$f_{n,a}^{\mathrm{prin}} \in C^{\mathrm{lc}}(S\mathfrak{X}_\Gamma)$, while the correction terms
are absorbed into the corresponding Patterson--Sullivan distribution.

Combining this identity with Theorem~\ref{thm:residue_PS}, we obtain the following consequence directly.

\begin{cor} \label{cor:Wigner}
    Let $\mathfrak{G}_\Gamma$ be a finite $(q+1)$-regular graph with $q > 1$.
    Fix a resonance $s_0 \in \mathbb{C}$ of finite algebraic multiplicity $\mathfrak{m}\coloneqq m_{\pm,1}(s_0) \in \mathbb{N}_0$ such that
    $q^{\frac{1}{2}+is_0} \notin \{\pm 1, \pm q\},$
    and assume that no Jordan block is associated with $s_0$.

    Then, for any $\ell^2(\mathfrak{X}_\Gamma)$-orthonormal basis $(\phi_\ell)^{\mathfrak{m}}_{\ell=1}$ of
    $\mathcal{E}_{\chi(s_0)}(\Delta_\Gamma;\mathrm{Maps}(\mathfrak{X}_\Gamma,\mathbb{C}))$, any $a \in C^{\mathrm{lc}}(S\mathfrak{X}_\Gamma)$ and every $n \in \mathbb{N}_0$ with $f_{n, s_0,a} \in C^{\mathrm{lc}}(S\mathfrak{X}_\Gamma)$ as in \eqref{eq:fn}, the following residue formula holds:
    \begin{equation*}
        \mathrm{Res}_{s=s_0}\big[\mathbf{Z}_{f_{n, s_0,a}}(s)\big]
        =
        c_{q,s_0}
        \sum_{\ell=1}^{\mathfrak{m}}
        \mathrm{W}_{\phi_\ell,\phi_\ell}\!\bigl(f_{n,a}^{\mathrm{prin}}\bigr),
    \end{equation*}
    where $c_{q,s_0}\in \C$ is the same constant as in Theorem~\ref{thm:residue_PS}.
\end{cor}

% -------------------------------%%--------------------------------------%
\bibliographystyle{amsalpha}
\bibliography{Literatur.bib}
\end{document}